\documentclass[11pt]{amsart}
\usepackage[margin=1in]{geometry}
\usepackage{latexsym}

\usepackage{amsfonts}
\usepackage{amsmath}
\usepackage{amssymb}
\usepackage{amsthm}
\usepackage{enumerate}
\usepackage[hang,flushmargin]{footmisc}
\usepackage{caption}
\usepackage{tabu}
\usepackage{mathrsfs}
\usepackage{amsaddr}
\usepackage{subfig}
\usepackage{mathtools}

\usepackage{xcolor}

\definecolor{MyBlue}{rgb}{0,0,1}
\definecolor{MyRed}{rgb}{1,0,0}
\definecolor{MyGreen}{rgb}{0,1,0}
\definecolor{MyIndigo}{rgb}{0.7254,0,1}
\definecolor{MyOrange}{rgb}{1,0.4431,0}

\usepackage{yhmath}
\usepackage{graphicx}
\usepackage{epstopdf}
\usepackage{epsfig}
\usepackage{caption}

\DeclareFontFamily{U}{mathx}{\hyphenchar\font45}
\DeclareFontShape{U}{mathx}{m}{n}{
      <5> <6> <7> <8> <9> <10>
      <10.95> <12> <14.4> <17.28> <20.74> <24.88>
      mathx10
      }{}
\DeclareSymbolFont{mathx}{U}{mathx}{m}{n}
\DeclareFontSubstitution{U}{mathx}{m}{n}
\DeclareMathAccent{\widecheck}{0}{mathx}{"71}
 
\usepackage{bm} 
\usepackage[noadjust]{cite}

\makeatletter

\newtheorem{theorem}{Theorem}[section]
\newtheorem{lemma}[theorem]{Lemma}
\newtheorem{proposition}[theorem]{Proposition}
\newtheorem{corollary}[theorem]{Corollary}

\newtheorem{conjecture}[theorem]{Conjecture}
\newtheorem{question}[theorem]{Question}

\theoremstyle{definition}
\newtheorem{definition}[theorem]{Definition}
\newtheorem{remarkx}[theorem]{Remark}
\newenvironment{remark}
  {\pushQED{\qed}\remarkx}
  {\popQED\endremarkx}
\newtheorem{examplex}[theorem]{Example}
\newenvironment{example}
  {\pushQED{\qed}\examplex}
  {\popQED\endexamplex}

\DeclareMathOperator{\skel}{skel}
\DeclareMathOperator{\slmax}{slmax}
\DeclareMathOperator{\fen}{fen}
\DeclareMathOperator{\Fen}{Fen}
\DeclareMathOperator{\Des}{Des}
\DeclareMathOperator{\syl}{syl}
\DeclareMathOperator{\shift}{shift}
\DeclareMathOperator{\des}{des}
\DeclareMathOperator{\post}{post}
\DeclareMathOperator{\VHC}{VHC}
\DeclareMathOperator{\Sky}{Sky}
\DeclareMathOperator{\DB}{DB}
\DeclareMathOperator{\AVHC}{AVHC}
\DeclareMathOperator{\rev}{rev}
\DeclareMathOperator{\revstack}{revstack}
\DeclareMathOperator{\stack}{\mathtt{s}}

\usepackage{etoolbox}

\newcommand{\includeSymbol}[1]{\ensuremath{%
	\mathchoice
		{\raisebox{-.7mm}{\includegraphics[height=2.2ex]{#1}}}	
		{\raisebox{-.7mm}{\includegraphics[height=2.2ex]{#1}}}
		{\raisebox{-.6mm}{\includegraphics[height=1.6ex]{#1}}}
		{\raisebox{-.5mm}{\includegraphics[height=1ex]{#1}}}
}}

\robustify{\includeSymbol}
\newcommand{\noneCirc}{\includeSymbol{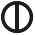}}
\newcommand{\upCirc}{\includeSymbol{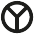}}
\newcommand{\downCirc}{\includeSymbol{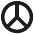}}
\newcommand{\upDownCirc}{\includeSymbol{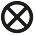}}

\newcommand{\dfn}[1]{\textcolor{blue}{\emph{#1}}}

\begin{document}
\title{Stack-Sorting for Coxeter Groups}
\author{Colin Defant}
\address{Princeton University \\ Department of Mathematics \\ Princeton, NJ 08544}
\email{cdefant@princeton.edu}

\begin{abstract}
Given an essential semilattice congruence $\equiv$ on the left weak order of a Coxeter group $W$, we define the \emph{Coxeter stack-sorting operator} ${\bf S}_\equiv:W\to W$ by ${\bf S}_\equiv(w)=w\left(\pi_\downarrow^\equiv(w)\right)^{-1}$, where $\pi_\downarrow^\equiv(w)$ is the unique minimal element of the congruence class of $\equiv$ containing $w$. When $\equiv$ is the sylvester congruence on the symmetric group $S_n$, the operator ${\bf S}_\equiv$ is West's stack-sorting map. When $\equiv$ is the descent congruence on $S_n$, the operator ${\bf S}_\equiv$ is the pop-stack-sorting map. We establish several general results about Coxeter stack-sorting operators, especially those acting on symmetric groups. For example, we prove that if $\equiv$ is an essential lattice congruence on $S_n$, then every permutation in the image of ${\bf S}_\equiv$ has at most $\left\lfloor\frac{2(n-1)}{3}\right\rfloor$ right descents; we also show that this bound is tight. 

We then introduce analogues of permutree congruences in types $B$ and $\widetilde A$ and use them to isolate Coxeter stack-sorting operators $\stack_B$ and $\widetilde\stack$ that serve as canonical type-$B$ and type-$\widetilde A$ counterparts of West's stack-sorting map. We prove analogues of many known results about West's stack-sorting map for the new operators $\stack_B$ and $\widetilde\stack$. For example, in type $\widetilde A$, we obtain an analogue of Zeilberger's classical formula for the number of $2$-stack-sortable permutations in $S_n$. 
\end{abstract}

\maketitle

\bigskip

\section{Introduction}\label{Sec:Intro} 

\subsection{Coxeter Stack-Sorting Operators}
A \dfn{semilattice congruence} on a meet-semilattice $M$ is an equivalence relation $\equiv$ on $M$ that respects meets. More precisely, this means that if $x_1\equiv x_2$ and $y_1\equiv y_2$, then $(x_1\wedge y_1)\equiv (x_2\wedge y_2)$. If $M$ has a minimal element and is locally finite, then every congruence class of $\equiv$ has a unique minimal element, and we denote by $\pi_\downarrow^\equiv:M\to M$ the projection map that sends each element of $M$ to the unique minimal element of its congruence class. We omit the superscript and write $\pi_\downarrow$ when the congruence $\equiv$ is clear from context. If $M$ is a lattice, then a \dfn{lattice congruence} on $M$ is an equivalence relation on $M$ that respects both meets and joins, meaning $x_1\equiv x_2$ and $y_1\equiv y_2$ together imply $(x_1\wedge y_1)\equiv (x_2\wedge y_2)$ and $(x_1\vee y_1)\equiv (x_2\vee y_2)$.

Let $(W,S)$ be a Coxeter system, and let $\leq_L$ and $\leq_R$ denote, respectively, the left and right weak orders on $W$. The posets $(W,\leq_L)$ and $(W,\leq_R)$ are isomorphic to each other, and a foundational theorem due to Bj\"orner \cite{Bjorner} states that they are complete meet-semilattices. We write $x\wedge y$ for the meet of two elements $x,y\in W$ in the left weak order. If $W$ is finite, then the left and right weak orders on $W$ are lattices; in this case, we write $x\vee y$ for the join of $x$ and $y$ in the left weak order. 

Semilattice congruences and lattice congruences on weak orders of Coxeter groups have been studied extensively \cite{Albertin, ChatelPilaud, Pilaud, Pilaud2, Pilaud3, Law, Giraudo, Hoang, ReadingCambrian, ReadingLattice2, ReadingLattice, ReadingSortable, ReadingSpeyerCambrian2, ReadingSpeyerCambrian, ReadingSpeyerFrameworks, ReadingSpeyerSortable, ReadingSpeyerSortable2}, especially due to their strong connections with polyhedral geometry, Hopf algebras, and cluster algebras. One of the quintessential examples of a lattice congruence is provided by the sylvester congruence $\equiv_{\syl}$ on the symmetric group $S_n$ \cite{Hivert}, which is closely related to Tamari lattices, associahedra, and the Hopf algebra of binary plane trees \cite{Loday, LodayRonco}; we define this congruence in Section~\ref{SecPermutrees}. The sylvester congruence is the prototypical example of a Cambrian congruence \cite{ReadingCambrian, ReadingSortable, ReadingSpeyerCambrian2, ReadingSpeyerCambrian, ReadingSpeyerFrameworks, ReadingSpeyerSortable, ReadingSpeyerSortable2}. Other notable lattice congruences on symmetric groups are the permutree congruences \cite{Pilaud}, the $k$-twist congruences \cite{Pilaud3}, and the Baxter congruences \cite{Law, Giraudo}.  

Another important semilattice congruence, which is defined on an arbitrary Coxeter group $W$, is the \dfn{descent congruence}, which we denote by $\equiv_{\des}$. Two elements of $W$ are equivalent in the descent congruence if and only if they have the same right descent set. The descent congruence on the symmetric group $S_n$ provides one of the other primary motivating examples of a permutree congruence besides the sylvester congruence. 

A semilattice congruence $\equiv$ on the left weak order of $W$ is called \dfn{essential}\footnote{As discussed in \cite{Albertin, Hoang, Pilaud2, ReadingLattice2}, lattice congruences of symmetric groups give rise to fans called \emph{quotient fans}. A lattice congruence on $S_n$ is essential if and only if the corresponding quotient fan is essential.} if the identity element $e\in W$ belongs to a singleton congruence class. 

\begin{definition}\label{DefCoxStackOp}
Let $(W,S)$ be a Coxeter system, and let $\equiv$ be an essential semilattice congruence on the left weak order of $W$. Define the \dfn{Coxeter stack-sorting operator} ${\bf S}_\equiv:W\to W$ to be the map given by ${\bf S}_\equiv(w)=w\left(\pi_\downarrow^\equiv(w)\right)^{-1}$ for all $w\in W$.   
\end{definition}

The motivation for the name \emph{Coxeter stack-sorting operator} comes from two special cases. First, when $W$ is the symmetric group $S_n$ and $\equiv_{\syl}$ is the sylvester congruence, the map ${\bf S}_{\equiv_{\syl}}$ is West's stack-sorting map. Indeed, this is the content of \cite[Corollary 16]{DefantPolyurethane} (where the map $\pi_\downarrow^{\equiv_{\syl}}$ goes by the name $\text{swd}$), and it is also a special case of Proposition~\ref{PropCoxx1} below. This map was originally defined by West \cite{West} as a deterministic variant of a stack-sorting machine introduced by Knuth \cite{Knuth}. West's stack-sorting map, which we will often simply call \emph{the stack-sorting map}, has now received vigorous attention and has found connections with several other parts of combinatorics \cite{Bona, BonaSurvey, BonaBoca, Branden3, DefantTroupes, DefantCounting, DefantCatalan, DefantPolyurethane, DefantFertilitopes, DefantEngenMiller, Goulden, Mularczyk, Singhal, Zeilberger, Fang}. The second motivation for our terminology comes from the fact that when $\equiv_{\des}$ is the descent congruence on $W=S_n$, the map ${\bf S}_{\equiv_{\des}}$ is the pop-stack-sorting map. This function, which is a deterministic analogue of a pop-stack-sorting machine introduced by Avis and Newborn in \cite{Avis}, first appeared in a paper of Ungar's about directions determined by points in the plane \cite{Ungar}; it has received a great deal of attention over the past few years \cite{AlbertVatter, Asinowski, Asinowski2, DefantCoxeterPop, Elder, ClaessonPop, ClaessonPop2, Pudwell}. 

\begin{remark}\label{RemCox1}
In order to earn the title of \emph{sorting operator}, a map $f:W\to W$ better have the property that for every $w\in W$, there is some $t\geq 0$ such that $f^t(w)=f^{t+1}(w)=e$. This is precisely why we require $\equiv$ to be essential in Definition~\ref{DefCoxStackOp}. Indeed, if we allowed for the case where $x\equiv e$ for some $x\in W\setminus \{e\}$, then $x$ would be a fixed point of ${\bf S}_\equiv$. 
\end{remark}

\begin{remark}
The authors of the recent paper \cite{Cerbai} introduced a different generalization of West's stack-sorting map that uses \emph{pattern-avoiding stacks}; this notion has spawned several subsequent articles in recent years \cite{Baril, Berlow, Cerbai3, Cerbai2, DefantZheng}. While these pattern-avoiding stacks are certainly interesting, we believe our Coxeter stack-sorting operators are more natural from an algebraic and lattice-theoretic point of view. 

The recent article \cite{PermutreeSorting} also generalizes stack-sorting to the realm of permutrees. However, the approach and the results in that paper are quite different from ours. 
\end{remark}   
   
\begin{remark}
The downward projection map $\pi_\downarrow^{\equiv}$ is crucial when $\equiv$ is a Cambrian congruence on a Coxeter group $W$ because its image is the set of \emph{sortable elements} \cite{ReadingSortable, ReadingSpeyerSortable, ReadingSpeyerSortable2}. Although Reading originally named these elements \emph{sortable} because of the connection with West's stack-sorting map and Knuth's stack-sorting machine, he did not study Coxeter stack-sorting operators. For an arbitrary semilattice congruence $\equiv$ on $W$, the elements of the image of $\pi_\downarrow^{\equiv}$ can still be called the \emph{sortable elements} in our setting because they are the elements of $W$ that get sorted into the identity $e$ after only a single application of ${\bf S}_\equiv$.  
\end{remark}   
   
In \cite{DefantCoxeterPop}, the author investigated the Coxeter stack-sorting operators ${\bf S}_{\equiv_{\des}}$ corresponding to descent congruences on Coxeter groups, which were called \dfn{Coxeter pop-stack-sorting operators}. Our goal in this paper is to initiate the investigation of Coxeter stack-sorting operators more generally and to consider specific Coxeter stack-sorting operators that are more closely related to West's stack-sorting map. 

\subsection{Outline}
In order to motivate many of our results, it is helpful to discuss previous work on West's stack-sorting map. We do this in Section~\ref{SecStackBack}, where we simultaneously present a more thorough synopsis of our main theorems from Sections~\ref{SecDescents}, \ref{SecTypeB}, and \ref{SecAffine} than what we give here. Section~\ref{SecCoxeter} summarizes some basic facts and terminology related to Coxeter groups that we will need later. In Section~\ref{SecGeneral}, we establish results that hold for arbitrary Coxeter stack-sorting operators. For example, we will prove that the statistic that keeps track of the number of preimages of an element of $W$ under ${\bf S}_{\equiv}$ is a decreasing function on the left weak order. Section~\ref{SecDescents} investigates the maximum number of right descents that a permutation in $S_n$ in the image of a Coxeter stack-sorting operator can have. In Section~\ref{SecPermutrees}, we discuss permutrees and permutree congruences, and we show that Coxeter stack-sorting operators associated to permutree congruences can be described in terms of postorder readings of decreasing permutrees. This provides a useful combinatorial model for dealing with these \emph{permutree stack-sorting operators}. In Section~\ref{SecTypeB}, we introduce analogues of permutrees and permutree congruences for the hyperoctahedral groups $B_n$. One specific type-$B$ permutree congruence on $B_n$, which we call the \emph{type-$B$ sylvester congruence}, yields a Coxeter stack-sorting operator $\stack_B$ that can be seen as the canonical type-$B$ analogue of West's stack-sorting map; much of Section~\ref{SecTypeB} is devoted to studying this operator. In Section~\ref{SecAffine}, we introduce analogues of permutrees and permutree congruences for the affine symmetric groups $\widetilde S_n$. One specific choice of an affine permutree congruence on $\widetilde S_n$, which we call the \emph{affine sylvester congruence}, yields a Coxeter stack-sorting operator $\widetilde\stack$ that serves as a canonical type-$\widetilde A$ analogue of West's stack-sorting map; a large portion of Section~\ref{SecAffine} concerns this operator. Finally, Section~\ref{SecConclusion} lists several suggestions for potential future work. 

\section{West's Stack-Sorting Map}\label{SecStackBack} 

In this paper, a \dfn{permutation} of size $n$ is a bijection $w:[n]\to X$ for some $n$-element set $X\subseteq\mathbb Z$. We write permutations as words in one-line notation. The \dfn{symmetric group} $S_n$ is the set of permutations of the set $[n]=\{1,\ldots,n\}$. The \dfn{standardization} of a permutation $w$ of size $n$ is the permutation in $S_n$ obtained by replacing the $i^\text{th}$-smallest entry in $w$ with $i$ for all $i$. For example, $7(-6)46$ is a permutation of size $4$ whose standardization is $4123$. 

Let $\stack$ denote West's stack-sorting map. A simple recursive definition of $\stack$ is as follows. First, define $\stack(\varepsilon)=\varepsilon$, where $\varepsilon$ is the empty permutation. Given a nonempty permutation $w$ with largest entry $m$, we can write $w$ in one-line notation as $\mathsf{L}m\mathsf{R}$. Then define $\stack(w)$ to be $\stack(\mathsf{L})\stack(\mathsf{R})m$. For example, \[\stack(4723165)=\stack(4)\stack(23165)\,7=4\,\stack(231)\,\stack(5)\,67=4\,\stack(2)\,\stack(1)\,3567=4213567.\] 

\medskip

We say a permutation $w$ is \dfn{$t$-stack-sortable} if $\stack^t(w)$ is an increasing permutation. It follows from Knuth's analysis \cite{Knuth} that a permutation is $1$-stack-sortable if and only if it avoids the pattern $231$; Knuth also showed that the number of $231$-avoiding permutations in $S_n$ is the $n^\text{th}$ Catalan number $C_n=\frac{1}{n+1}\binom{2n}{n}$ (see \cite[Chapters~4~and~8]{Bona}). There has been a great deal of work dedicated to understanding $2$-stack-sortable permutations (see \cite{Bona, BonaSurvey, BonaBoca, Branden3, DefantTroupes, DefantCounting, Goulden, Zeilberger, Fang} and the references therein); the first major result proved about these permutations was Zeilberger's theorem, which states that the number of $2$-stack-sortable permutations in $S_n$ is $\frac{2}{(n+1)(2n+1)}\binom{3n}{n}$ \cite{Zeilberger}. \'Ulfarsson characterized $3$-stack-sortable permutations \cite{Ulfarsson}; only recently were these permutations enumerated via a complicated recurrence relation \cite{DefantCounting}. 

Consider the affine symmetric group $\widetilde S_n$, as defined in Section~\ref{SecAffine}. We will introduce a specific Coxeter stack-sorting operator $\widetilde\stack:\widetilde S_n\to\widetilde S_n$. Let us say an affine permutation $w\in\widetilde S_n$ is \dfn{$t$-stack-sortable} if $\widetilde\stack^t(w)=e$. In Section~\ref{SecAffine}, we will see that $1$-stack-sortable affine permutations are also characterized by the property of avoiding the pattern $231$. A theorem due to Crites \cite{Crites} states that the number of $231$-avoiding affine permutations in $\widetilde S_n$ is $\binom{2n-1}{n}$. We will enumerate $2$-stack-sortable affine permutations, thereby obtaining an affine analogue of Zeilberger's seminal formula. To be more precise, let us write $\widetilde{\mathcal W}_2(n)$ for the set of $2$-stack-sortable elements of $\widetilde S_n$, and let $I(q)=\sum_{n\geq 0}\frac{2}{(n+1)(2n+1)}\binom{3n}{n}q^n$ and $\widetilde I(q)=\sum_{n\geq 1}\lvert \widetilde{\mathcal W}_2(n)\rvert q^n$. We will prove that \[\widetilde I(q)=\frac{qI'(q)}{I(q)(I(q)-1)}-1.\] 

\medskip

Given a set $X$, a function $f:X\to X$, and an element $x\in X$, we define the \dfn{forward orbit} of $x$ under $f$ to be the set $O_f(x)=\{x,f(x),f^2(x),\ldots\}$. When confronted with such a dynamical system, it is natural to consider the sizes of the forward orbits, and it is particularly natural to ask for $\sup\limits_{x\in X}\lvert O_f(x)\rvert $.  Additional motivation for studying this quantity when $f$ is a sorting operator comes from the observation that it measures the worst-case complexity of the sorting procedure. West \cite{West} proved that $\max\limits_{w\in S_n}\lvert O_{\stack}(w)\rvert $ is $n$, which happens to be the Coxeter number of $S_n$. In Section~\ref{SecGeneral}, we invoke a theorem from \cite{DefantCoxeterPop} to see that if $W$ is an arbitrary finite irreducible Coxeter group and $w\in W$, then the size of the forward orbit of $w$ under a Coxeter stack-sorting operator is at most the Coxeter number of $W$. It follows that every forward orbit of the specific map $\stack_B:B_n\to B_n$ that we define in Section~\ref{SecTypeB} has size at most the Coxeter number of $B_n$, which is $2n$. However, we improve this estimate in Section~\ref{SecTypeB}, showing that \[\max_{w\in B_n}\left\lvert O_{\stack_B}(w)\right\rvert =n+1.\]

\medskip

If $w\in S_n$, then the number of right descents of $\stack(w)$ is at most $\left\lfloor\frac{n-1}{2}\right\rfloor$ (see \cite{DefantEngenMiller} and \cite[Chapter~8, Exercise~18]{Bona}). This result is natural if one thinks of the stack-sorting map as a genuine sorting procedure and views the number of right descents of a permutation in $S_n$ as a measure of how ``far away'' the permutation is from being sorted into the identity. In Section~\ref{SecDescents}, we prove that a similar result completely fails for arbitrary Coxeter stack-sorting operators on symmetric groups. Namely, we show that for each $n\geq 2$, there exist a permutation $w\in S_n$ and an essential semilattice congruence $\equiv$ on $(S_n,\leq_L)$ such that ${\bf S}_{\equiv}(w)$ has $n-2$ right descents. However, we will see that there \emph{is} a nontrivial upper bound on the number of right descents when we restrict our attention to \emph{lattice congruences}. We prove that if $\equiv$ is an essential lattice congruence on $(S_n,\leq_L)$, then every permutation in the image of ${\bf S}_{\equiv}$ has at most $\left\lfloor\frac{2(n-1)}{3}\right\rfloor$ right descents. Moreover, we will see that this bound is tight because it is attained when $\equiv$ is the descent congruence. In Section~\ref{SecTypeB}, we find that the maximum number of right descents that an element of $B_n$ in the image of $\stack_B$ can have is $\left\lfloor\frac{n}{2}\right\rfloor$. In Section~\ref{SecAffine}, we prove that the maximum number of right descents that an element of $\widetilde S_n$ in the image of $\widetilde\stack$ can have is also $\left\lfloor\frac{n}{2}\right\rfloor$.

\medskip

The number of preimages of a permutation $w\in S_n$ under the stack-sorting map is called the \dfn{fertility} of $w$. In \cite{DefantPolyurethane}, the author proved that the fertility of $w$ only depends on the sylvester class of $w$. An analogous result fails for arbitrary Coxeter stack-sorting operators, but we will see that it does hold for $\widetilde\stack$, the affine analogue of the stack-sorting map. Define the \dfn{fertility} of an element $w\in\widetilde S_n$ to be $\lvert \widetilde\stack^{-1}(w)\rvert $. We will see that the fertility of $w$ is finite and only depends on the affine sylvester class of $w$. (The affine sylvester congruence is the semilattice congruence used to define $\widetilde\stack$.) In fact, we will prove that for every $t\geq 0$, the number of preimages of $w$ under $\widetilde{\stack}^t$ is finite and only depends on the affine sylvester class of $w$.

In \cite{DefantEngenMiller}, Engen, Miller, and the present author defined a permutation to be \dfn{uniquely sorted} if its fertility is $1$. They proved the following characterization of these permutations. 
\begin{proposition}[\cite{DefantEngenMiller}]\label{PropUniquely}
Let $n\geq 1$. A permutation of size $n$ is uniquely sorted if and only if it is in the image of the stack-sorting map $\stack$ and has exactly $\frac{n-1}{2}$ descents. In particular, there are no uniquely sorted permutations of size $n$ if $n$ is even.
\end{proposition}
Quite surprisingly, uniquely sorted permutations are counted by a fascinating sequence called \emph{Lassalle's sequence} (sequence A180874 in \cite{OEIS}), which is the sequence of absolute values of the classical cumulants of the standard semicircular distribution \cite{DefantEngenMiller}. Uniquely sorted permutations also possess several other remarkable enumerative properties \cite{DefantEngenMiller, DefantCatalan, Mularczyk, Singhal}. 

Let us say an affine permutation $w\in\widetilde S_n$ is \dfn{uniquely sorted} if its fertility is $1$. We prove in Section~\ref{SecAffine} that an element of $\widetilde S_n$ is uniquely sorted if and only if it is in the image of $\widetilde\stack$ and has exactly $\frac{n}{2}$ right descents. In particular, there are no uniquely sorted affine permutations in $\widetilde S_n$ when $n$ is odd. When $n\geq 2$ is even, there are infinitely many uniquely sorted elements of $\widetilde S_n$. Because we know (by Theorem~\ref{ThmCoxx2}) that the fertility of an element of $\widetilde S_n$ only depends on its affine sylvester class, it is natural to define an affine sylvester class in $\widetilde S_n$ to be \dfn{uniquely sorted} if its elements are uniquely sorted. We prove that the number of uniquely sorted affine sylvester classes in $\widetilde S_{2k}$ is \[3\binom{4k}{k}-2\sum_{i=0}^k\binom{4k}{i}.\] This appears to be the first combinatorial interpretation for these numbers, which form sequence A107026 in \cite{OEIS}. 

Much of the author's work on the stack-sorting map has relied on a certain Decomposition Lemma, which provides a recursive method for computing the fertility of a permutation, and a certain Fertility Formula, which gives an explicit expression for the fertility of a permutation as a sum over combinatorial objects called \emph{valid hook configurations}. These tools have led to several new results about the stack-sorting map, including the aforementioned recurrence for counting $3$-stack-sortable permutations \cite{DefantCounting}, theorems about uniquely sorted permutations \cite{DefantEngenMiller, DefantCatalan, Mularczyk, Singhal}, a very surprising and useful connection with cumulants in noncommutative probability theory \cite{DefantTroupes}, and connections with certain polytopes called \emph{nestohedra} \cite{DefantFertilitopes}. In Section~\ref{SecAffine}, we introduce \emph{affine valid hook configurations}, and we prove analogues of the Decomposition Lemma and the Fertility Formula for affine permutations. These tools are what allow us to enumerate uniquely sorted affine sylvester classes.

\section{Coxeter Groups}\label{SecCoxeter}

We assume basic familiarity with the combinatorics of Coxeter groups; a standard reference for the subject is \cite{BjornerBrenti}. We will often refer to a Coxeter group $W$ with the understanding that we are really referring to a Coxeter system $(W,S)$ for some specific set $S$ of simple generators. We let $e$ denote the identity element of $W$. Thus, $W$ has a presentation of the form $\langle S:(ss')^{m(s,s')}=e\rangle$ such that $m(s,s)=1$ for all $s\in S$ and $m(s,s')=m(s',s)\in\{2,3,\ldots\}\cup\{\infty\}$ for all distinct $s,s'\in S$. The \dfn{length} of an element $w\in W$, denoted $\ell(w)$, is the length of a reduced word for $w$. The \dfn{left weak order} on $W$ is the partial order $\leq_L$ on $W$ defined by saying $x\leq_L y$ if $\ell(yx^{-1})=\ell(y)-\ell(x)$. The \dfn{right weak order} on $W$ is the partial order $\leq_R$ on $W$ defined by saying $x\leq_R y$ if $\ell(x^{-1}y)=\ell(y)-\ell(x)$. The map $W\to W$ given by $w\mapsto w^{-1}$ is an isomorphism from the left weak order to the right weak order. A \dfn{right descent} of an element $w\in W$ is a simple generator $s\in S$ such that $\ell(ws)<\ell(w)$; the collection of all right descents of $w$ is the \dfn{right descent set} of $w$, which we denote by $D_R(w)$. Equivalently, $D_R(w)=\{s\in S:s\leq_L w\}$. A \dfn{Coxeter element} of a finite Coxeter group $W$ is an element obtained by multiplying the simple generators together in some order; all Coxeter elements have the same order in $W$, which is called the \dfn{Coxeter number} of $W$. A Coxeter group is called \dfn{irreducible} if its Coxeter diagram is a connected graph.  

The Coxeter groups of type $A$ are the symmetric groups $S_n$. The elements of $S_n$ are permutations of $[n]$, which we write as words in one-line notation. The simple generators of $S_n$ are $s_1,\ldots,s_{n-1}$, where $s_i=(i\,\,i+1)$ is the transposition that swaps $i$ and $i+1$. The transposition $s_i$ is a right descent of a permutation $w\in S_n$ if and only if $w(i)>w(i+1)$. A \dfn{right inversion} (respectively, \dfn{left inversion}) of $w$ is a pair $(i,j)$ such that $1\leq i<j\leq n$ and $w(i)>w(j)$ (respectively, $w^{-1}(i)>w^{-1}(j)$). For $v,w\in S_n$, we have $v\leq_L w$ (respectively, $v\leq_R w$) if and only if every right (respectively, left) inversion of $v$ is a right (respectively, left) inversion of $w$. The number of right (equivalently, left) inversions of $w$ is $\ell(w)$. 

Recall that, in this article, a permutation of size $n$ is a bijection $w:[n]\to X$, where $X\subseteq\mathbb Z$ has cardinality $n$. The one-line notation of $w$ is the word $w(1)\cdots w(n)$. The \dfn{plot} of $w$ is the diagram showing the points $(i,w(i))\in\mathbb R^2$ for all $i\in[n]$. It is often convenient to consider plots modulo horizontal translation. In other words, moving all points in the plot of $w$ to the right or left by some fixed distance gives another diagram that we still call the \emph{plot} of $w$. We write $\Des(w)$ for the set of indices $i\in [n-1]$ such that $w(i)>w(i+1)$; such indices are called the \dfn{descents} of $w$. We have $i\in\Des(w)$ if and only if $s_i\in D_R(w)$. 

We will often make use of the automorphism $\alpha$ of the group $S_n$ defined by $\alpha(w)=w_0ww_0$, where $w_0=n(n-1)\cdots 321$ is the longest element of $S_n$. The one-line notation of $\alpha(w)$ is \[(n+1-w(n))(n+1-w(n-1))\cdots (n+1-w(1)).\] The plot of $\alpha(w)$ is obtained by rotating the plot of $w$ by $180^\circ$. It is well known that $\alpha$ is a lattice automorphism of the left (and also the right) weak order on $S_n$. Therefore, given a semilattice congruence $\equiv$ on $(S_n,\leq_L)$, we can consider the semilattice congruence $\alpha(\equiv)$ defined by saying $v\,\alpha(\equiv)\,w$ if and only if $\alpha(v)\equiv \alpha(w)$. For all $w\in S_n$, we have 
\begin{equation}\label{EqCoxx1}
\pi_\downarrow^{\alpha(\equiv)}(\alpha(w))=\alpha(\pi_\downarrow^\equiv(w))\quad\text{and}\quad{\bf S}_{\alpha(\equiv)}(\alpha(w))=\alpha({\bf S}_\equiv(w)). 
\end{equation}

\section{General Results}\label{SecGeneral}
In this section, we collect some general facts about arbitrary Coxeter stack-sorting operators. 

\medskip

We first address the problem of determining the maximum size of a forward orbit of a Coxeter stack-sorting operator ${\bf S}_\equiv:W\to W$ when $W$ is a finite irreducible Coxeter group. Recall from Remark~\ref{RemCox1} that our definition of a Coxeter stack-sorting operator guarantees that every forward orbit of ${\bf S}_\equiv$ contains $e$, which is a fixed point. Thus, we will be concerned with the maximum number of iterations of ${\bf S}_\equiv$ needed to send an element to the identity. 

Let $W$ be an arbitrary Coxeter group. In \cite{DefantCoxeterPop}, the author defined a map $f:W\to W$ to be \dfn{compulsive} if $f(w)\leq_R w$ and $f(w)\leq_R ws$ for all $w\in W$ and all $s\in D_R(w)$. He proved that if $W$ is finite and irreducible and $f$ is compulsive, then $\sup\limits_{w\in W}\lvert O_f(w)\rvert \leq h$, where $h$ is the Coxeter number of $W$. We will see that Coxeter stack-sorting operators are compulsive. First, we need a simple lemma that gives an alternative characterization of essential semilattice congruences. Recall that an equivalence relation $\equiv$ is said to \dfn{refine} an equivalence relation $\equiv'$ if every equivalence class of $\equiv$ is contained in an equivalence class of $\equiv'$.  

\begin{lemma}\label{LemEssential}
A semilattice congruence $\equiv$ on the left weak order of a Coxeter group $W$ is essential if and only if it refines the descent congruence $\equiv_{\des}$ on $W$.
\end{lemma}

\begin{proof}
Every refinement of the descent congruence is certainly essential because $e$ is the only element of $W$ with an empty right descent set. For the converse, suppose $\equiv$ is an essential semilattice congruence on the left weak order of $W$. Let $S$ be the set of simple generators of $W$. Consider $v,w\in W$ such that $v\equiv w$. Let $s\in S$. Because $\equiv$ is a semilattice congruence, we have $v\wedge s\equiv w\wedge s$. Both $v\wedge s$ and $w\wedge s$ are less than or equal to $s$ in the left weak order, so they must belong to $\{e,s\}$. Since $e\not\equiv s$, we deduce that $v\wedge s=w\wedge s$. Therefore, $v\geq_L s$ if and only if $w\geq_L s$. In other words, $s\in D_R(v)$ if and only if $s\in D_R(w)$. As this is true for all $s\in S$, we conclude that $D_R(v)=D_R(w)$, meaning $v\equiv_{\des} w$. 
\end{proof}

\begin{proposition}\label{PropCox1}
Let $\equiv$ be an essential semilattice congruence on the left weak order of a Coxeter group $W$. The Coxeter stack-sorting operator ${\bf S}_\equiv$ is compulsive. If $W$ is finite and irreducible and has Coxeter number $h$, then \[\max_{w\in W}\left\lvert O_{{\bf S}_\equiv}(w)\right\rvert \leq h.\]
\end{proposition} 

\begin{proof}
To prove that ${\bf S}_\equiv$ is compulsive, we must prove that ${\bf S}_\equiv(w)\leq_R wx$ for every $w\in W$ and $x\in D_R(w)\cup\{e\}$. Fix such elements $w$ and $x$. Let $\pi_\downarrow$ be the downward projection map associated to $\equiv$. Because $w\equiv\pi_\downarrow(w)$, it follows from Lemma~\ref{LemEssential} that $D_R(w)=D_R(\pi_\downarrow(w))$. Since $x\in D_R(w)\cup\{e\}=D_R(\pi_\downarrow(w))\cup\{e\}$, we have $\ell(wx)=\ell(w)-\ell(x)$ and $\ell(\pi_\downarrow(w)x)=\ell(\pi_\downarrow(w))-\ell(x)$. We know that $\pi_\downarrow(w)\leq_L w$ and ${\bf S}_\equiv(w)=w(\pi_\downarrow(w))^{-1}$, so $\ell({\bf S}_\equiv(w))+\ell(\pi_\downarrow(w))=\ell(w)$. Hence, \[\ell\left(({\bf S}_\equiv(w))^{-1}wx\right)=\ell(\pi_\downarrow(w)x)=\ell(\pi_\downarrow(w))-\ell(x)=\ell(w)-\ell(x)-\ell({\bf S_\equiv}(w))=\ell(wx)-\ell({\bf S_\equiv}(w)).\] This shows that ${\bf S}_\equiv(w)\leq_R wx$, completing the proof that ${\bf S}_\equiv$ is compulsive. The last statement in the theorem follows from Theorem~1.4 in \cite{DefantCoxeterPop}. 
\end{proof}

The Coxeter number of $S_n$ is $n$. Therefore, Proposition~\ref{PropCox1} states that if ${\bf S}_\equiv$ is a Coxeter stack-sorting operator on $S_n$, then ${\bf S}_\equiv^{n-1}(w)=e$ for all $w\in S_n$. 

\medskip

In \cite[Chapter~8, Exercise~23]{Bona}, B\'ona asks the reader to find the permutation in $S_n$ that has the maximum number of preimages under the stack-sorting map. As one might expect, the answer is the identity permutation $e=123\cdots n$. The number of preimages of $123\cdots n$ under the stack-sorting map is the $n^\text{th}$ Catalan number $C_n=\frac{1}{n+1}\binom{2n}{n}$, and every other permutation in $S_n$ has strictly fewer than $C_n$ preimages. This result is a special case of the following proposition because the number of congruence classes of the sylvester congruence on $S_n$ is the number of binary plane trees with $n$ vertices, which is $C_n$.

\begin{proposition}
Let $\equiv$ be an essential semilattice congruence on the left weak order of a Coxeter group $W$. Let $K$ be the number of congruence classes of $\equiv$, and assume $K$ is finite. Then $\lvert {\bf S}_\equiv^{-1}(e)\rvert =K$, and $\lvert {\bf S}_\equiv^{-1}(v)\rvert <K$ for every $v\in W\setminus\{e\}$.
\end{proposition}

\begin{proof}
The elements of ${\bf S}_\equiv^{-1}(e)$ are precisely the minimal elements of the congruence classes, so $\lvert {\bf S}_\equiv^{-1}(e)\rvert =K$. Now consider $v\in W\setminus\{e\}$. If $w\in{\bf S}_\equiv^{-1}(v)$, then $v^{-1}w=v^{-1}{\bf S}_\equiv(w)\pi_\downarrow(w)=\pi_\downarrow(w)\in{\bf S}_\equiv^{-1}(e)$, and we know that $\pi_\downarrow(w)\neq e$ because $e$ is in a singleton congruence class. It follows that the map $w\mapsto v^{-1}w$ is an injection from ${\bf S}_\equiv^{-1}(v)$ into ${\bf S}_\equiv^{-1}(e)\setminus\{e\}$, so $\lvert {\bf S}_\equiv^{-1}(v)\rvert \leq K-1$. 
\end{proof}

One of the main results of \cite{DefantMonotonicity} states that $\lvert \stack^{-1}(w)\rvert \leq\lvert\stack^{-1}(\stack(w))\rvert$ for all $w\in S_n$, with equality if and only if $w=e$. In other words, this theorem says that the action of the stack-sorting map causes fertilities of permutations to increase monotonically. The analogous result fails for an arbitrary Coxeter stack-sorting operator, even if we restrict our attention to $S_n$. For example, if ${\bf S}_{\equiv_{\des}}$ is the pop-stack-sorting map on $S_5$, then ${\bf S}_{\equiv_{\des}} (24135)=21435$, but $\lvert {\bf S}_{\equiv_{\des}}^{-1}(24135)\rvert =3>2=\lvert {\bf S}_{\equiv_{\des}}^{-1}(21435)\rvert $. Nonetheless, there is still a certain monotonicity result that holds for all Coxeter stack-sorting operators. 

\begin{theorem}\label{ThmWeakMono}
Let $W$ be a Coxeter group, and let ${\bf S}_\equiv:W\to W$ be a Coxeter stack-sorting operator. If $v,w\in W$ are such that $v\leq_L w$, then $\lvert {\bf S}_\equiv^{-1}(v)\rvert \geq\lvert {\bf S}_\equiv^{-1}(w)\rvert $. 
\end{theorem}

\begin{proof}
It suffices to prove the result when $\ell(w)=\ell(v)+1$. In this case, we can write $w=sv$ for some simple generator $s\in S$. Choose $u\in{\bf S}_\equiv^{-1}(w)$. We can write $u={\bf S}_\equiv(u)\pi_\downarrow(u)=w\pi_\downarrow(u)=sv\pi_\downarrow(u)$, where $\ell(u)=1+\ell(v)+\ell(\pi_\downarrow(u))$ because $\pi_\downarrow(u)\leq_L u$. Then $\pi_\downarrow(u)\leq_Lv\pi_\downarrow(u)=su\leq_L u$. Because $\equiv$ is a semilattice congruence and $u\equiv\pi_\downarrow(u)$, we have $su=u\wedge su\equiv\pi_\downarrow(u)\wedge su=\pi_\downarrow(u)$. Consequently, $\pi_\downarrow(su)=\pi_\downarrow(u)$. If follows that \[{\bf S}_\equiv(su)=su(\pi_\downarrow(su))^{-1}=su(\pi_\downarrow(u))^{-1}=s{\bf S}_\equiv(u)=sw=v.\] Hence, we have an injection ${\bf S}_\equiv^{-1}(w)\to{\bf S}_\equiv^{-1}(v)$ given by $u\mapsto su$.   
\end{proof}

\section{Descents After Coxeter Stack-Sorting in Symmetric Groups}\label{SecDescents}
Recall that we write $\stack$ for West's stack-sorting map, which coincides with the Coxeter stack-sorting operator ${\bf S}_{\equiv_{\syl}}$ defined using the sylvester congruence. 

In this section, we analyze Coxeter stack-sorting operators on symmetric groups. It is natural to use the number of right descents of a permutation $w\in S_n$ as a measure of how ``far away'' $w$ is from the identity permutation. Given that a Coxeter stack-sorting operator on $S_n$ is called a \emph{sorting operator}, we might hope that the permutations in its image would be somewhat ``close'' to the identity. For example, 
\begin{equation}\label{EqmaxDes}
\max\limits_{w\in S_n}\lvert D_R(\stack(w))\rvert =\left\lfloor\frac{n-1}{2}\right\rfloor
\end{equation} for all $n\geq 1$ (see \cite{DefantEngenMiller} and \cite[Chapter~8, Exercise~18]{Bona}). Is something similar true for other Coxeter stack-sorting operators? 

If we allow $\equiv$ to be an arbitrary essential semilattice congruence on the left weak order of $S_n$, then $\max\limits_{w\in S_n}\lvert D_R({\bf S}_\equiv(w))\rvert $ can be as large as $n-2$, as we will see below in Theorem~\ref{ThmCoxx6}. However, it turns out that if we require $\equiv$ to be a \emph{lattice congruence}, then the inequality $\max\limits_{w\in S_n}\lvert D_R({\bf S}_\equiv(w))\rvert \leq\left\lfloor\frac{2(n-1)}{3}\right\rfloor$ holds. We will also see that this bound is tight. 

Given a poset $Q$ and elements $u,v\in Q$ with $u\leq v$, we write $[u,v]$ for $\{q\in Q:u\leq q\leq v\}$, the closed interval in $Q$ between $u$ and $v$. A meet-semilattice is called \dfn{locally finite} if every one of its closed intervals is finite. 

\begin{lemma}\label{LemGeneratingCongruence}
Let $M$ be a locally finite meet-semilattice with a minimal element $\widehat 0$, and let $u,v\in M$ be such that $u\leq v$. There exists a semilattice congruence on $M$ that has $[u,v]$ as a congruence class. 
\end{lemma}

\begin{proof}
Let us say an equivalence relation $\sim$ on $M$ is \dfn{cheerful} if the following two conditions hold: 
\begin{itemize}
\item $[u,v]$ is an equivalence class of $\sim$;
\item if $w\not\leq v$, then $w$ is in a singleton equivalence class of $\sim$.
\end{itemize} We are going to construct a sequence $\sim_1,\sim_2,\ldots$ of cheerful equivalence relations on $W$. First, let $\sim_1$ be the equivalence relation whose equivalence classes are $[u,v]$ and the singleton sets $\{w\}$ for $w\in M\setminus [u,v]$; it is clear that $\sim_1$ is cheerful. 

Now let $j\geq 1$, and suppose we have already constructed the cheerful equivalence relation $\sim_j$. Define the relation $\sim_j'$ by saying $z_1\sim_j' z_2$ if and only if either $z_1\sim_j z_2$ or there exist $x_1,x_2,y_1,y_2\in M$ with $x_1\sim_j x_2$, $y_1\sim_j y_2$, $x_1\wedge y_1=z_1$, and $x_2\wedge y_2=z_2$. Let $\sim_{j+1}$ be the transitive closure of $\sim_j'$. To see that $\sim_{j+1}$ is cheerful, we first need to show that $[u,v]$ is an equivalence class of $\sim_{j+1}$. Notice that $[u,v]$ is certainly contained in an equivalence class of $\sim_{j+1}$ because it is an equivalence class of $\sim_j$. Now suppose, by way of contradiction, that $[u,v]$ is properly contained in an equivalence class of $\sim_{j+1}$. Then there exist $z_1\in M\setminus[u,v]$ and $z_2\in[u,v]$ such that $z_1\sim_j' z_2$. We know that $z_1\not\sim_j z_2$ because $\sim_j$ is cheerful, so there must exist $x_1,x_2,y_1,y_2\in M$ with $x_1\sim_j x_2$, $y_1\sim_j y_2$, $x_1\wedge y_1=z_1$, and $x_2\wedge y_2=z_2$. Observe that $x_2,y_2\geq z_2\geq u$. If $x_1\neq x_2$, then $x_2\leq v$ because $\sim_j$ is cheerful. In this case, $x_2\in[u,v]$, so $x_1\in[u,v]$, again because $\sim_j$ is cheerful. Similarly, if $y_1\neq y_2$, then $y_1,y_2\in [u,v]$. Since $x_1\wedge y_1=z_1\not\in[u,v]$, the elements $x_1$ and $y_1$ cannot both belong to $[u,v]$. Thus, either $x_1=x_2$ or $y_1=y_2$. Without loss of generality, assume $x_1=x_2$. We know that $y_1\neq y_2$ because $z_1\neq z_2$. Consequently, $y_1,y_2\in [u,v]$. We find that $x_1=x_2\geq u$ and $y_1\geq u$, so $z_1\geq u$. However, $z_1\leq y_1\leq v$, and this contradicts the assumption that $z_1\not\in [u,v]$. Hence, $[u,v]$ is an equivalence class of $\sim_{j+1}$. 

Consider $w\in M$ with $w\not\leq v$. Suppose, by way of contradiction, that $\{w\}$ is not an equivalence class of $\sim_{j+1}$. Then $w\sim_j' z$ for some $z\neq w$. We know that $w\not\sim_j z$ because $\sim_j$ is cheerful. Therefore, there exist $x_1',x_2',y_1',y_2'\in M$ with $x_1'\sim_j x_2'$, $y_1'\sim_j y_2'$, $x_1'\wedge y_1'=w$, and $x_2'\wedge y_2'=z$. Because $w\not\leq v$, we have $x_1'\not\leq v$ and $y_1'\not\leq v$. Since $\sim_j$ is cheerful, we must have $x_1'=x_2'$ and $y_1'=y_2'$, so $w=z$. This is a contradiction, so we deduce that $\sim_{j+1}$ is cheerful. 

By induction, all of the equivalence relations $\sim_1,\sim_2,\ldots$ are cheerful. The fact that $M$ is locally finite (in particular, $[\widehat 0,v]$ is finite) implies that there is some $N\geq 1$ such that $\sim_N$ and $\sim_{N+1}$ are identical. Thus, $\sim_N$ is the same as $\sim_N'$. It follows from our construction of $\sim_N'$ from $\sim_N$ that $\sim_N$ is a semilattice congruence on $M$. Because $\sim_N$ is cheerful, it has $[u,v]$ as a congruence class. 
\end{proof}

A subset $E$ of a poset is called \dfn{convex} if for all $x,y\in E$, we have $[x,y]\subseteq E$. A simple consequence of the definition of a semilattice congruence, which we will need in the proof of the next proposition, is that every congruence class of a semilattice congruence is convex. 

\begin{theorem}\label{ThmCoxx6}
For each $n\geq 2$, we have \[\max_\equiv\max_{w\in S_n}\lvert D_R({\bf S}_\equiv (w))\rvert =n-2,\] where the first maximum is over all essential semilattice congruences on the left weak order of $S_n$. 
\end{theorem}

\begin{proof}
It follows from Definition~\ref{DefCoxStackOp} that the decreasing permutation $n(n-1)\cdots 321$ is not in the image of any Coxeter stack-sorting operator on $S_n$. Therefore, $\max\limits_\equiv\max\limits_{w\in S_n}\lvert D_R({\bf S}_\equiv (w))\rvert \leq n-2$. 

Let $k=\left\lceil n/2\right\rceil$. Define $u\in S_n$ by $u(i)=(i+1)/2$ for $i$ odd and $u(i)=k+i/2$ for $i$ even. Define $v\in S_n$ by $v(i)=k-(i-1)/2$ for $i$ odd and $v(i)=n+1-i/2$ for $i$ even. For example, when $n=7$, we have $u=1526374$ and $v=4736251$. One can readily verify that every right inversion of $u$ is a right inversion of $v$, so $u\leq_L v$. Moreover, $\Des(u)=\Des(v)$, so $u\equiv_{\des} v$. The congruence classes of $\equiv_{\des}$ are convex, so $[u,v]$ is contained in one such congruence class. The left weak order on a Coxeter group is locally finite and has $e$ as its unique minimal element. According to Lemma~\ref{LemGeneratingCongruence}, there is a semilattice congruence $\equiv$ on $(S_n,\leq_L)$ that has $[u,v]$ as a congruence class. Let $\equiv'$ be the common refinement of $\equiv$ and $\equiv_{\des}$. In other words, we define $\equiv'$ by declaring $x\equiv' y$ if and only if $x\equiv y$ and $x\equiv_{\des} y$. Then $\equiv'$ is an essential semilattice congruence on $S_n$ that has $[u,v]$ as a congruence class. Finally, ${\bf S}_{\equiv'}(v)=vu^{-1}=k(k-1)\cdots 21n(n-1)\cdots(k+2)(k+1)$ is a permutation in the image of ${\bf S}_{\equiv'}$ with $n-2$ right descents. 
\end{proof}

Our next goal is to prove that permutations in the image of a Coxeter stack-sorting operator arising from a lattice congruence on $S_n$ must have at most $\left\lfloor\frac{2(n-1)}{3}\right\rfloor$ right descents. To do this, we will make use of a combinatorial description of the lattice congruences on $S_n$ due to Reading. These concepts are discussed in more detail in \cite{Hoang, ReadingNoncrossing} and \cite[Section 10-5]{ReadingFinite}, where they are formulated in terms of the right weak order on $S_n$; it is easy to translate to the left weak order by taking inverses of permutations. 

Fix $n\geq 2$. Suppose $v\lessdot_L w$ is a cover relation in the left weak order on $S_n$. Then $w=s_iv$ for some $i\in[n-1]$. If we let $v^{-1}(i)=a$ and $v^{-1}(i+1)=b$ and let $(a\,\,b)$ be the transposition that swaps $a$ and $b$, then $w=v\,(a\,\,b)$. Let $R=\{a+1,\ldots,b-1\}\cap v^{-1}([i-1])$ (equivalently, $R=\{a+1,\ldots,b-1\}\cap w^{-1}([i-1])$). We label the edge $(v,w)$ in the Hasse diagram of $(S_n,\leq_L)$ with the tuple $(a,b;R)$. The \dfn{fence} associated to the tuple $(a,b;R)$, denoted $\fen(a,b;R)$, is the set of all edges in the Hasse diagram of $(S_n,\leq_L)$ with label $(a,b;R)$. Let \[\Fen_n=\{\fen(a,b;R):1\leq a<b\leq n, R\subseteq\{a+1,\ldots,b-1\}\}\] be the set of all such fences. The \dfn{forcing order} is the partial order $\preceq_{\text{f}}$ on $\Fen_n$ given by declaring $\fen(a,b;R)\preceq_{\text{f}} \fen(a',b';R')$ whenever $a\leq a'<b'\leq b$ and $R'=R\cap\{a'+1,\ldots,b'-1\}$. Figure~\ref{FigCoxx11} shows the four fences of the left weak order on $S_3$ and the corresponding forcing order. 

\begin{figure}[ht]
  \begin{center}{\includegraphics[height=4.52cm]{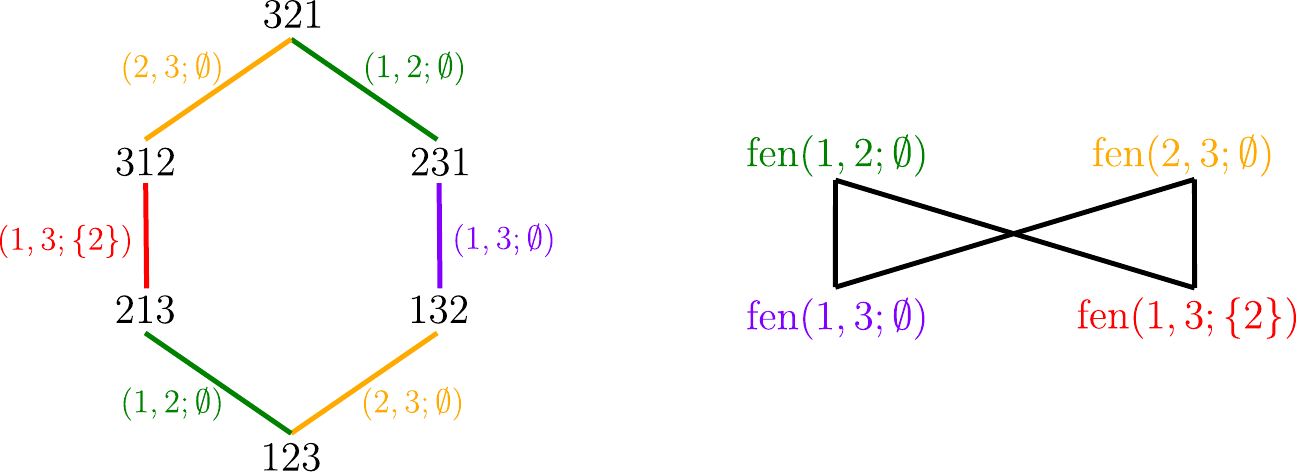}}
  \end{center}
  \caption{On the left is the left weak order on $S_3$ and its fences, which are represented by different colors. On the right is the forcing order on $\Fen_3$.}\label{FigCoxx11}
\end{figure}

For each lattice congruence $\equiv$ on $(S_n,\leq_L)$, let $\eta(\equiv)$ be the set of edges $(v,w)$ in the Hasse diagram of $(S_n,\leq_L)$ such that $v\equiv w$. Let $\Psi(\equiv)$ be the collection of fences in $\Fen_n$ that have a nonempty intersection with $\eta(\equiv)$. In other words, $\fen(a,b;R)$ is in $\Psi(\equiv)$ if and only if there is an edge $(v,w)$ with label $(a,b;R)$ such that $v\equiv w$. 

The following theorem appeared originally as \cite[Theorem~4.6]{ReadingNoncrossing}; although, in \cite[Theorem~4]{Hoang}, it was reformulated in a manner closer to what we have presented. 

\begin{theorem}[\cite{ReadingNoncrossing}]\label{ThmFences}
If $\equiv$ is a lattice congruence on $(S_n,\leq_L)$, then $\eta(\equiv)=\bigcup_{F\in\Psi(\equiv)}F$. Furthermore, the map $\Psi$ is a bijection from the set of lattice congruences on $(S_n,\leq_L)$ to the set of order ideals of $(\Fen_n,\preceq_{\mathrm{f}})$. 
\end{theorem}

The following lemma will handle most of the heavy lifting in our proof of Theorem~\ref{ThmCoxx1}.

\begin{lemma}\label{LemCoxx1}
Suppose $\equiv$ is a lattice congruence on $(S_n,\leq_L)$. For $w\in W$ and $v=\pi_\downarrow(w)$, let $\mathcal E_{\equiv}(w)$ denote the set of indices $i\in\Des({\bf S}_\equiv(w))$ such that there exists $r\in v^{-1}([i-1])$ satisfying $v^{-1}(i)<r<v^{-1}(i+1)$. Then $\lvert \mathcal E_\equiv(w)\rvert \leq n-1-\lvert \Des({\bf S}_\equiv(w))\rvert $. 
\end{lemma} 

\begin{proof}
For each $i\in\mathcal E_\equiv(w)$, let \[R_i=\{r\in v^{-1}([i-1]):v^{-1}(i)<r<v^{-1}(i+1)\}\quad\text{and}\quad R_i'=\{r'\in v^{-1}([i-1]):v^{-1}(i)<r'\};\] these sets are nonempty by the definition of $\mathcal E_\equiv(w)$. For $i\in\mathcal E_\equiv(w)$, let $\gamma(i)=\max\{v(r):r\in R_i\}$ and $\gamma'(i)=\max\{v(r'):r'\in R_i'\}$. We are going to show that $\gamma(i)=\gamma'(i)$. We will also prove that the map $\gamma$ is an injection from $\mathcal E_\equiv(w)$ into $[n-1]\setminus\Des({\bf S}_{\equiv}(w))$, which will complete the proof. 

Fix $i\in\mathcal E_\equiv(w)$. To ease notation, let $a=v^{-1}(i)$ and $b=v^{-1}(i+1)$. Since $i\in \Des({\bf S}_\equiv(w))$ and ${\bf S}_\equiv(w)=wv^{-1}$, we have $w(a)>w(b)$. The definition of $\mathcal E_\equiv(w)$ implies that $a<b$, so $(a,b)$ is a right inversion of $w$. The right inversions of $s_iv$ are the right inversions of $v$ together with $(a,b)$. Since $v\leq_L w$, the right inversions of $v$ are all right inversions of $w$. Consequently, $s_iv\leq_L w$. It follows that $s_iv\in[v,w]$, so the congruence class of $\equiv$ containing $v$ and $w$ must also contain $s_iv$ (congruence classes are convex). This means that $v\equiv s_iv$, so the edge $(v,s_iv)$ in the Hasse diagram of $(S_n,\leq_L)$ is in $\eta(\equiv)$. This edge has label $(a,b;R_i)$, so $\fen(a,b;R_i)\in\Psi(\equiv)$. 

Let $r_1=v^{-1}(\gamma'(i))$, and note that $r_1\in R_i'$ by the definition of $\gamma'(i)$. Let $r_0=v^{-1}(\gamma'(i)+1)$. It follows from the definition of $\gamma'(i)$ that $r_0\leq a<r_1$. Let $Q=\{q\in v^{-1}([\gamma'(i)-1]):r_0<q<r_1\}$. We want to prove that $\gamma'(i)=\gamma(i)$; this is equivalent to showing that $r_1\in R_i$, which is equivalent to showing that $r_1<b$. Suppose, by way of contradiction, that $r_1\geq b$. Then every element of $R_i$ is at most $\gamma'(i)-1$, so $Q\cap\{a+1,\ldots,b-1\}=R_i$. This implies that $\fen(r_0,r_1;Q)\preceq_{\text{f}}\fen(a,b;R_i)$. Theorem~\ref{ThmFences} tells us that $\Psi(\equiv)$ is an order ideal of $(\Fen_n,\preceq_{\text{f}})$, so $\fen(r_0,r_1;Q)\in\Psi(\equiv)$. Using Theorem~\ref{ThmFences} again, we find that $\fen(r_0,r_1;Q)\subseteq\eta(\equiv)$. Notice that $v(r_0\,\,r_1)$ is covered by $v$ in the left weak order because $v(r_1)=\gamma'(i)$ and $v(r_0)=\gamma'(i)+1$. The edge $(v(r_0\,\,r_1), v)$ in the Hasse diagram has label $(r_0,r_1;Q)$, so this edge is in the fence $\fen(r_0,r_1;Q)$. This implies that this edge is in $\eta(\equiv)$, so $v(r_0\,\,r_1)\equiv v$. However, this contradicts the fact that $v$ is the minimal element of its congruence class (because it is $\pi_\downarrow(w)$). We conclude from this contradiction that $\gamma(i)=\gamma'(i)$. In particular, $v^{-1}(\gamma(i)+1)=r_0\leq a<r_1=v^{-1}(\gamma(i))$. As $(r_0,r_1)$ is a right inversion of $v$, it must also be a right inversion of $w$ (since $v\leq_L w$). It follows that $w(v^{-1}(\gamma(i)))=w(r_1)<w(r_0)=w(v^{-1}(\gamma(i)+1))$, so $\gamma(i)\in[n-1]\setminus\Des(wv^{-1})=[n-1]\setminus\Des({\bf S}_\equiv(w))$. 

We have shown that $\gamma$ is a map from $\mathcal E_\equiv(w)$ to $[n-1]\setminus\Des({\bf S}_\equiv(w))$, so we are left with the task of proving its injectivity. Suppose  $i,j\in\mathcal E_\equiv(w)$ are such that $\gamma(i)=\gamma(j)$. Without loss of generality, we may assume $j\leq i$. By the definition of $\gamma$ and the assumption that $\gamma(i)=\gamma(j)$, we have $v^{-1}(\gamma(i))\in R_i\cap R_j$. Therefore, $v^{-1}(i)<v^{-1}(\gamma(i))<v^{-1}(j+1)$. This implies that $j+1\neq i$, so $j+1\in[i-1]\cup\{i+1\}$. We saw in the previous paragraph that $\gamma'(i)=\gamma(i)$. If $j+1\in[i-1]$, then $v^{-1}(j+1)\in R_i'$, so $j+1\leq\gamma'(i)=\gamma(i)$. However, this contradicts the fact that $v^{-1}(\gamma(i))\in R_j\subseteq v^{-1}([j-1])$. We deduce that $j+1=i+1$, so $j=i$.   
\end{proof}

In the proof of the following theorem, we use the automorphism $\alpha$ of $S_n$ given by $\alpha(w)=w_0ww_0$, where $w_0=n(n-1)\cdots 321$. Given a lattice congruence $\equiv$ on $(S_n,\leq_L)$, recall from Section~\ref{SecCoxeter} that we write $\alpha(\equiv)$ for the lattice congruence defined by saying $v\,\alpha(\equiv)\,w$ if and only if $\alpha(v)\equiv \alpha(w)$. As stated in \eqref{EqCoxx1}, we have the identities $\pi_\downarrow^{\alpha(\equiv)}(\alpha(w))=\alpha(\pi_\downarrow^\equiv(w))$ and ${\bf S}_{\alpha(\equiv)}(\alpha(w))=\alpha({\bf S}_\equiv(w))$. 

\begin{theorem}\label{ThmCoxx1}
For each $n\geq 1$, we have \[\max_\equiv\max_{w\in S_n}\lvert D_R({\bf S}_\equiv (w))\rvert =\left\lfloor\frac{2(n-1)}{3}\right\rfloor,\] where the first maximum is over all essential lattice congruences on the left weak order of $S_n$. 
\end{theorem}

\begin{proof}
Preserve the notation from Lemma~\ref{LemCoxx1}. Let $\equiv$ be an essential lattice congruence on the left weak order of $S_n$, and choose $w\in S_n$. Let $v=\pi_\downarrow^\equiv(w)$ so that ${\bf S}_\equiv(w)=wv^{-1}$. 

Because $\equiv$ is essential and $v\equiv w$, we have $\Des(v)=\Des(w)$ by Lemma~\ref{LemEssential}. Suppose $i\in \Des({\bf S}_\equiv(w))$. Then $w(v^{-1}(i))>w(v^{-1}(i+1))$. If we had $v^{-1}(i+1)<v^{-1}(i)$, then the pair $(v^{-1}(i+1),v^{-1}(i))$ would be a right inversion of $v$ but not a right inversion of $w$, contradicting the fact that $v\leq_L w$. Therefore, $v^{-1}(i)<v^{-1}(i+1)$. If we had $v^{-1}(i+1)=v^{-1}(i)+1$, then $v^{-1}(i)$ would be in $\Des(w)$ but not not in $\Des(v)$, contradicting the fact that $\Des(v)=\Des(w)$. Therefore, $v^{-1}(i+1)\geq v^{-1}(i)+2$. Choose an integer $r$ such that $v^{-1}(i)<r<v^{-1}(i+1)$. If $v(r)\leq i-1$, then $i\in\mathcal E_\equiv(w)$. Now suppose $v(r)\geq i$. We must have $v(r)\geq i+2$ by the choice of $r$. Observe that $n-i\in\Des(\alpha({\bf S}_\equiv(w)))=\Des({\bf S}_{\alpha(\equiv)}(\alpha(w)))$. If we let $u=\pi_\downarrow^{\alpha(\equiv)}(\alpha(w))=\alpha(v)$, then $u^{-1}(n-i)=n+1-v^{-1}(i+1)<n+1-r<n+1-v^{-1}(i)=u^{-1}(n-i+1)$. Furthermore, $u(n+1-r)=n+1-v(r)\leq n-i-1$. This shows that $n-i\in\mathcal E_{\alpha(\equiv)}(\alpha(w))$, so $i\in n-\mathcal E_{\alpha(\equiv)}(\alpha(w))$.

We have shown that \[\Des({\bf S}_\equiv(w))\subseteq\mathcal E_\equiv(w)\cup\left(n-\mathcal E_{\alpha(\equiv)}(\alpha(w))\right),\] so we can invoke Lemma~\ref{LemCoxx1} to find that \[\lvert \Des({\bf S}_\equiv(w))\rvert \leq \lvert \mathcal E_\equiv(w)\rvert +\lvert \mathcal E_{\alpha(\equiv)}(\alpha(w))\rvert \leq n-1-\lvert \Des({\bf S}_\equiv(w))\rvert +n-1-\lvert \Des({\bf S}_{\alpha(\equiv)}(\alpha(w)))\rvert \] \[=2n-2-\lvert \Des({\bf S}_\equiv(w))\rvert -\lvert \Des(\alpha({\bf S}_\equiv(w)))\rvert .\] The map $s_j\mapsto\alpha(s_j)$ is a bijection from $\Des({\bf S}_\equiv(w))$ to $\Des(\alpha({\bf S}_\equiv(w)))$, so \[\lvert \Des({\bf S}_\equiv(w))\rvert \leq 2n-2-2\lvert 
\Des({\bf S}_\equiv(w))\rvert .\] Upon rearranging, we obtain $\lvert D_R({\bf S}_\equiv(w))\rvert =\lvert \Des({\bf S}_\equiv(w))\rvert \leq\frac{2(n-1)}{3}$. 

To prove that equality holds in the theorem, it suffices to exhibit a permutation $\zeta_n\in S_n$ such that $\lvert \Des({\bf S}_{\equiv_{\des}}(\zeta_n))\rvert =\left\lfloor\frac{2(n-1)}{3}\right\rfloor$. We begin by recursively constructing the permutations $\zeta_3,\zeta_6,\zeta_9,\ldots$. Let $\zeta_3=231\in S_3$. For $k\geq 1$, we obtain $\zeta_{3k+3}$ from $\zeta_{3k}$ by replacing the entry $3k$ in $\zeta_{3k}$ with the entry $3k+2$ and then appending the entries $3k+1,3k+3,3k$ (in this order) to the end of the resulting permutation. For example, $\zeta_6=251463$ and $\zeta_9=251483796$. Now, suppose $n$ is not divisible by $3$, and let $m=\left\lceil n/3\right\rceil$. Let $\zeta_n$ be the standardization of the permutation formed by the first $n$ entries of $\zeta_{3m}$. For example, if $n=7$, then $m=3$, and the permutation formed by the first $7$ entries of $\zeta_{9}$ is $2514837$. The standardization of this permutation is $\zeta_7=2514736$. 

A \dfn{descending run} of a permutation $w$ is a maximal consecutive decreasing subsequence of $w$. As explained in \cite{DefantCoxeterPop}, the map ${\bf S}_{\equiv_{\des}}$ is the pop-stack-sorting map, which acts by reversing the descending runs of a permutation while keeping entries in different descending runs in the same relative order (this also follows easily from Proposition~\ref{PropCoxx1} below). For example, the descending runs of $\zeta_9$ are $2,51,4,83,7,96$, so ${\bf S}_{\equiv_{\des}}(\zeta_9)=215438769$. For each $k\geq 1$, let \[\xi_{3k-1}=21543876\cdots (3k-1)(3k-2)(3k-3).\] It is straightforward to check by hand that \[{\bf S}_{\equiv_{\des}}(\zeta_n)=\begin{cases} \xi_{3k-1}(3k) & \mbox{if } n=3k; \\ \xi_{3k-1}(3k+1)(3k) & \mbox{if } n=3k+1; \\ \xi_{3k-1}(3k+1)(3k)(3k+2) & \mbox{if } n=3k+2 \end{cases}\] (where juxtaposition represents concatenation). Hence, $\lvert \Des({\bf S}_{\equiv_{\des}}(\zeta_n))\rvert =\left\lfloor\frac{2(n-1)}{3}\right\rfloor$.  
\end{proof}

\section{Permutrees and Postorders}\label{SecPermutrees}

\subsection{Permutrees}
Pilaud and Pons \cite{Pilaud} introduced permutrees as generalizations of permutations, binary plane trees, Cambrian trees, and binary sequences. These objects provide a rich family of lattice congruences on $(S_n,\leq_L)$ that contains the equality congruence, the sylvester congruence, the Cambrian congruences, and the descent congruence. Permutrees also give rise to fascinating \emph{permutree lattices}, \emph{permutreehedra}, and \emph{permutree Hopf algebras}. We will be interested in permutrees in this article because they afford Coxeter stack-sorting operators that have useful combinatorial descriptions in terms of postorder readings. In Sections~\ref{SecTypeB} and~\ref{SecAffine}, we will define analogues of permutrees for Coxeter groups of type $B$ and type $\widetilde A$, respectively.   
 
The definition of a permutree makes use of the symbols $\noneCirc, \upCirc, \downCirc, \upDownCirc$. A \dfn{decoration} is a word over the alphabet $\{\noneCirc, \upCirc, \downCirc, \upDownCirc\}$. If $v$ is a vertex in a directed tree $T$, then a vertex $u$ is a \dfn{parent} (respectively, \dfn{child}) of $v$ if there is an arrow $v\to u$ (respectively, $u\to v$). We say $u$ is an \dfn{ancestor} (respectively, \dfn{descendant}) of $v$ if there is a directed path with at least one arrow from $v$ to $u$ (respectively, from $u$ to $v$). 
 
\begin{definition}\label{DefPermutrees}
Let $\delta=\delta_1\cdots\delta_n\in \{\noneCirc, \upCirc, \downCirc, \upDownCirc\}^n$ be a decoration. A \dfn{permutree} with decoration $\delta$, also called a \dfn{$\delta$-permutree}, is a directed plane tree $T$ with vertex set $V(T)$ endowed with a bijective labeling $p:V(T)\to[n]$ such that 
\begin{itemize}
\item each vertex $v$ such that $\delta_{p(v)}\in \{\noneCirc, \downCirc\}$ has at most $1$ parent;
\item each vertex $v$ such that $\delta_{p(v)}\in \{\noneCirc,\upCirc\}$ has at most $1$ child; 
\item each vertex $v$ such that $\delta_{p(v)}\in\{\upCirc,\upDownCirc\}$ has at most one parent $u$ satisfying $p(u)<p(v)$; if such a vertex $u$ exists, then every ancestor $u'$ of $u$ also satisfies $p(u')<p(v)$;
\item each vertex $v$ such that $\delta_{p(v)}\in\{\upCirc,\upDownCirc\}$ has at most one parent $u$ satisfying $p(u)>p(v)$; if such a vertex $u$ exists, then every ancestor $u'$ of $u$ also satisfies $p(u')>p(v)$;
\item each vertex $v$ such that $\delta_{p(v)}\in\{\downCirc,\upDownCirc\}$ has at most one child $u$ satisfying $p(u)<p(v)$; if such a vertex $u$ exists, then every descendant $u'$ of $u$ also satisfies $p(u')<p(v)$;
\item each vertex $v$ such that $\delta_{p(v)}\in\{\downCirc,\upDownCirc\}$ has at most one child $u$ satisfying $p(u)>p(v)$; if such a vertex $u$ exists, then every descendant $u'$ of $u$ also satisfies $p(u')>p(v)$.
\end{itemize}
A parent (respectively, child) $u$ of a vertex $v$ satisfying $p(u)<p(v)$ is called a \dfn{left parent} (respectively, \dfn{left child}) of $v$. A parent (respectively, child) $u'$ of a vertex $v$ satisfying $p(u')>p(v)$ is called a \dfn{right parent} (respectively, \dfn{right child}) of $v$.
\end{definition} 

\begin{figure}[ht]
  \begin{center}{\includegraphics[height=4.226cm]{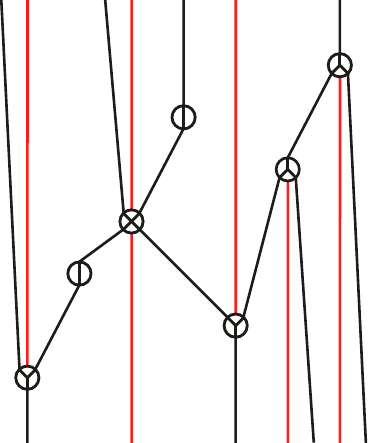}}
  \end{center}
  \caption{A permutree with decoration $\upCirc\noneCirc\upDownCirc\noneCirc\upCirc\downCirc\downCirc$.}\label{FigCoxx1}
\end{figure}

Figure~\ref{FigCoxx1} shows an example of a permutree with $7$ vertices that has decoration $\upCirc\noneCirc\upDownCirc\noneCirc\upCirc\downCirc\downCirc$. Following \cite{Pilaud}, we use the following conventions when drawing a $\delta$-permutree: 

\begin{itemize}
\item Each edge is directed upward, so we do not need to draw the orientations separately. 
\item The vertices are drawn from left to right in the order $p^{-1}(1),\ldots,p^{-1}(n)$, so the labeling $p$ is encoded in the drawing.
\item Each vertex $p^{-1}(i)$ is represented by the symbol $\delta_i\in \{\noneCirc, \upCirc, \downCirc, \upDownCirc\}$, so the decoration is encoded in the drawing. 
\item We draw a vertical red wall emanating up from each vertex $v$ such that $\delta_{p(v)}\in \{\upCirc,\upDownCirc\}$. We draw a vertical red wall emanating down from each vertex $v$ such that $\delta_{p(v)}\in \{\downCirc,\upDownCirc\}$.  
\item We draw extra strands that extend up or down infinitely when a vertex does not have as many parents or children as its symbol permits it to have. 
\end{itemize}

When $v$ has a red wall emanating up or down from it, we say it \dfn{emits} the red wall. The last four bulleted items in Definition~\ref{DefPermutrees} imply that edges in a permutree cannot cross through red walls. 

A \dfn{linear extension} of an $n$-element poset $(Q,\leq_Q)$ is a bijection $\sigma: Q\to[n]$ such that $\sigma(x)\leq\sigma(y)$ whenever $x\leq_Q y$. Associated to a permutree $T$ is a natural partial order $\preceq$ on the vertex set $V(T)$ obtained by declaring $u\preceq v$ whenever there is a directed path (possibly with no edges) from $u$ to $v$. A \dfn{decreasing permutree} is a pair $(T,\sigma)$, where $T$ is a permutree with $n$ vertices and $\sigma:V(T)\to[n]$ is a linear extension. We think of $\sigma$ as a labeling of the vertices of $T$, and we often write $\mathcal T$ for the pair $(T,\sigma)$. When we draw a decreasing permutree, we place the element with label $j$ at height $j$. 

Consider the following \dfn{insertion algorithm}. Fix a decoration $\delta=\delta_1\cdots\delta_n\in \{\noneCirc, \upCirc, \downCirc, \upDownCirc\}^n$. Consider a permutation $w\in S_n$. Draw the plot of $w$ with each point $(i,w(i))$ represented by the symbol $\delta_i$. Draw a vertical red wall emanating up from each point represented by $\upCirc$ or $\upDownCirc$. Draw a vertical red wall emanating down from each point represented by $\downCirc$ or $\upDownCirc$. Between any two downward red walls, and on the far right and the far left, draw an infinite incoming strand. Now move up the points in the plot from bottom to top. At each step, a point represented by $\noneCirc$ or $\upCirc$ attaches to the only incoming strand that it sees, while a point represented by $\downCirc$ or $\upDownCirc$ attaches to both of the incoming strands that it sees. Furthermore, a point represented by $\noneCirc$ or $\downCirc$ expels $1$ upward strand, while a point represented by $\upCirc$ or $\upDownCirc$ expels two upward strands. This procedure finishes with an infinite upward strand between any two consecutive upward red walls and on the far right and far left.

\begin{figure}[ht]
  \begin{center}{\includegraphics[height=8.985cm]{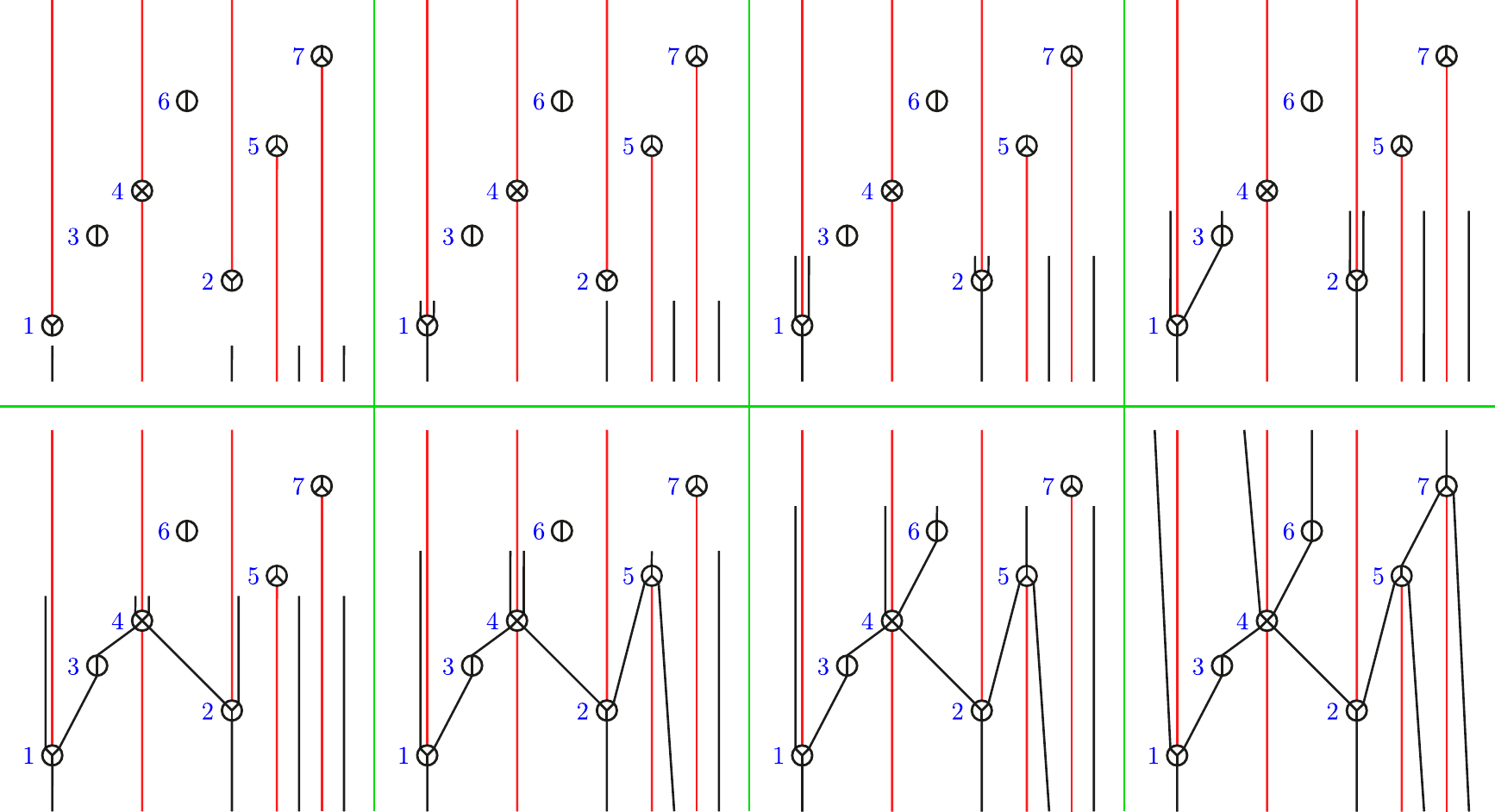}}
  \end{center}
  \caption{The insertion algorithm.}\label{FigCoxx2}
\end{figure}

The end result of the insertion algorithm is a decreasing $\delta$-permutree such that the vertex $p^{-1}(i)$ is the point $(i,w(i))$, which has label $w(i)$. Pilaud and Pons \cite{Pilaud} proved that this algorithm defines a bijection from $S_n$ to the set of decreasing $\delta$-permutrees. The inverse of this bijection is easier to describe. The permutation $w$ is obtained from the decreasing permutree $\mathcal T$ by reading the labels of the vertices $p^{-1}(1),\ldots,p^{-1}(n)$ in this order; this is called the \dfn{in-order reading} of $\mathcal T$. We let $\mathcal I_\delta(\mathcal T)$ denote the in-order reading of a decreasing $\delta$-permutree $\mathcal T$. In other words, if $\mathcal T=(T,\sigma)$, then $\mathcal I_{\delta}(\mathcal T)=\sigma\circ p^{-1}=w$. 

\begin{example}
Suppose $w=1346257$ and $\delta=\upCirc\noneCirc\upDownCirc\noneCirc\upCirc\downCirc\downCirc$. The first step of the insertion algorithm is shown in the top left corner of Figure~\ref{FigCoxx2}. The final result of the algorithm is the decreasing $\delta$-permutree $\mathcal I_\delta^{-1}(w)=(T,\sigma)$ shown in the bottom right of the figure. The label $\sigma(v)$ of each vertex $v$ (which is also its height) is written next to $v$. 
\end{example}

Define the \dfn{skeleton} of a decreasing permutree $\mathcal T=(T,\sigma)$, denoted $\skel(\mathcal T)$, to be the underlying permutree $T$. Using the in-order reading bijection $\mathcal I_\delta$, we can define the \dfn{$\delta$-skeleton} of a permutation $w$, denoted $\skel_\delta(w)$, to be the $\delta$-permutree $\skel(\mathcal I_\delta^{-1}(w))$. The \dfn{$\delta$-permutree congruence} is the equivalence relation $\equiv_\delta$ on $S_n$ defined by saying $v\equiv_\delta w$ if and only if $\skel_\delta(v)=\skel_\delta(w)$. Pilaud and Pons \cite{Pilaud} proved that the $\delta$-permutree congruence is a lattice congruence on the left weak order of $S_n$. Therefore, we can consider the Coxeter stack-sorting operator ${\bf S}_{\equiv_\delta}:S_n\to S_n$ defined by ${\bf S}_{\equiv_\delta}(w)=w\left(\pi_\downarrow^{\equiv_\delta}(w)\right)^{-1}$. We call such a map a \dfn{permutree stack-sorting operator}. 

\begin{example}\label{ExamCoxx1}
Consider the following special choices of the decoration $\delta$: 
\begin{enumerate}[(I)]
\item If $\delta=\noneCirc^n$, then the map $\skel_\delta$ is a bijection from $S_n$ to the set of $\delta$-permutrees. In this case, the $\delta$-permutree congruence is the equality congruence given by $v\equiv_\delta w$ if and only if $v=w$. The Coxeter stack-sorting operator ${\bf S}_{\equiv_\delta}$ is the constant map that sends every permutation in $S_n$ to the identity permutation $e$. 

\medskip

\item If $\delta=\downCirc^n$, then $\delta$-permutrees are the same as binary plane trees. In this case, the $\delta$-permutree congruence is the \dfn{sylvester congruence} $\equiv_{\syl}$. It follows from \cite[Corollary~16]{DefantPolyurethane} that the Coxeter stack-sorting operator ${\bf S}_{\equiv_\delta}$ is the same as West's stack-sorting map $\stack$.

\medskip

\item If $\delta\in \{\upCirc,\downCirc\}^n$, then $\delta$-permutrees are called \dfn{Cambrian trees} \cite{ChatelPilaud}. The $\delta$-permutree congruences for $\delta\in\{\upCirc,\downCirc\}^n$ are precisely the Cambrian congruences on $S_n$, which were introduced by Reading \cite{ReadingCambrian}. The combinatorial properties of a Coxeter stack-sorting operator associated to a Cambrian congruence $\equiv_\delta$ depends on the particular decoration $\delta$. 

\medskip

\item If $\delta=\upDownCirc^n$, then $\delta$-permutrees are essentially the same as binary sequences. In this case, the $\delta$-permutree congruence is the descent congruence on $S_n$, and the Coxeter stack-sorting operator ${\bf S}_{\equiv_\delta}$ is the pop-stack-sorting map (see \cite{DefantCoxeterPop, Pilaud}). \qedhere
\end{enumerate}
\end{example}

We end this subsection with some terminology that will help when we deal with postorder readings of decreasing permutrees later. Let $\delta$ be a decoration, and let $\mathcal T=(T,\sigma)$ be a decreasing $\delta$-permutree with associated linear ordering $p:V(T)\to [n]$. Consider two vertices $x,y\in V(T)$ with $p(x)<p(y)$. If $z$ is a vertex represented by either $\upCirc$ or $\upDownCirc$, then it emits an upward red wall. We say this wall \dfn{separates} $x$ and $y$ if $p(x)<p(z)<p(y)$ and $\sigma(z)<\min\{\sigma(x),\sigma(y)\}$. Similarly, if $z'$ is a vertex represented by either $\downCirc$ or $\upDownCirc$, then it emits a downward red wall. We say this wall separates $x$ and $y$ if $p(x)<p(z')<p(y)$ and $\sigma(z')>\max\{\sigma(x),\sigma(y)\}$. The vertices $x$ and $y$ are comparable in the poset $(V(T),\preceq)$ if and only if they are not separated by any red walls. 

\subsection{Postorders}

Let $(Q,\leq_Q)$ be an $n$-element poset whose elements are endowed with a linear ordering via a bijective function $p:Q\to [n]$. Define the \dfn{postorder} of $Q$ to be the unique linear extension $\sigma_{\post}:Q\to [n]$ such that $\sigma_{\post}(x)<\sigma_{\post}(y)$ whenever $x,y\in Q$ are incomparable (with respect to $\leq_Q$) and satisfy $p(x)<p(y)$. To visualize this definition, imagine that we draw the Hasse diagram of $Q$ in the plane so that $p(v)$ is the horizontal coordinate of $v$ for every $v\in Q$. Then $\sigma_{\post}$ is obtained by listing the elements of $Q$ in a greedy fashion so that at each step, we write down the leftmost element that has not yet been listed and that is not above (in the Hasse diagram) any other elements that have not been listed. For example, consider the poset $Q$ in Figure~4. We have named the elements $v_1,\ldots,v_7$ so that $p(v_i)=i$ for all $i$. The list $\sigma_{\post}^{-1}(1),\ldots,\sigma_{\post}^{-1}(7)$ of the elements in postorder is $v_1,v_2,v_5,v_3,v_4,v_6,v_7$.  

\begin{figure}[ht]
  \centering
  {\includegraphics[height=3.231cm]{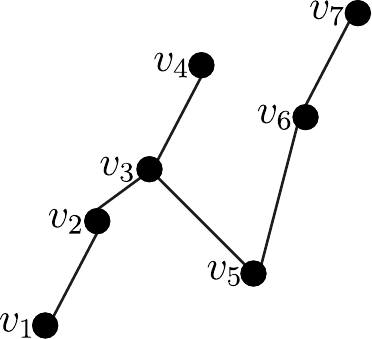}}
  \caption{A poset $Q$ together with a labeling $p$ given by $p(v_i)=i$.}\label{Fig4}
\end{figure}

If $\sigma:Q\to [n]$ is an arbitrary linear extension of $Q$, then we define the \dfn{postorder reading} of $\sigma$ to be the permutation $\mathcal P(\sigma)=\sigma\circ \sigma_{\post}^{-1}\in S_n$. In other words, if we think of $\sigma$ as giving a labeling of the elements of $Q$, then the postorder reading of $\sigma$ is obtained by reading the labels in the order specified by $\sigma_{\post}$. 

Let $\delta\in\{\noneCirc,\upCirc,\downCirc,\upDownCirc\}^n$ be a decoration. As mentioned above, associated to each $\delta$-permutree $T$ is a poset $(V(T),\preceq)$ endowed with a linear ordering of its elements given by a bijection $p:V(T)\to[n]$. Therefore, it makes sense to consider the postorder $\sigma_{\post}$ of $V(T)$. If $\mathcal T=(T,\sigma)$ is a decreasing $\delta$-permutree, then it makes sense to consider the postorder reading $\mathcal P(\mathcal T)$, which is simply the postorder reading $\mathcal P(\sigma)$ as defined above. For a concrete example, let $\mathcal T$ be the decreasing permutree shown in the bottom right corner of Figure~\ref{FigCoxx2}. The poset $(V(T),\preceq)$ and its labeling $p$ are the same as the poset and the labeling in Figure~4. The postorder reading is $\mathcal P(\mathcal T)=1324657$.  

We now know how to define the postorder reading of a decreasing permutree. If $\delta=\downCirc^n$, then a decreasing $\delta$-permutree is the same thing as a decreasing binary plane tree, and the postorder reading defined here agrees with the standard postorder reading (see \cite{Bona, DefantPolyurethane, DefantTroupes}). One definition of the stack-sorting map, which is responsible for much of the structure underlying it (such as its connection with free probability theory \cite{DefantTroupes} and nestohedra \cite{DefantFertilitopes}), combines the in-order reading with the postorder reading. More precisely, the stack-sorting map is given by $\stack=\mathcal P\circ\mathcal I_{\downCirc^n}^{-1}$ (see \cite{Bona, DefantPolyurethane, DefantTroupes}). The following proposition shows that we can generalize this fact, providing a useful combinatorial model for working with permutree stack-sorting operators. 

\begin{proposition}\label{PropCoxx1}
Let $\delta\in\{\noneCirc,\upCirc,\downCirc,\upDownCirc\}^n$ be a decoration, and let $\equiv_\delta$ be the corresponding permutree congruence on $S_n$. The permutree stack-sorting operator ${\bf S}_{\equiv_\delta}$ satisfies the identity \[{\bf S}_{\equiv_\delta}=\mathcal P\circ\mathcal I_\delta^{-1}.\]
\end{proposition}

\begin{proof}
Choose $w\in S_n$, and let $v=\pi_\downarrow^{\equiv_\delta}(w)$. Let $\mathcal I_{\delta}^{-1}(w)=\mathcal T=(T,\sigma)$. Then $T$ is a $\delta$-permutree, so it has an associated linear ordering of its vertices given by a bijection $p:V(T)\to[n]$. Let $\sigma_{\post}:V(T)\to[n]$ be the postorder of $V(T)$. The set of permutations of the form $\sigma'\circ p^{-1}$ for $\sigma'$ a linear extension of $V(T)$ is precisely the congruence class of $\equiv_\delta$ containing $w$. In particular, there are linear extensions $\sigma_w:V(T)\to[n]$ and $\sigma_v:V(T)\to[n]$ such that $w=\sigma_w\circ p^{-1}$ and $v=\sigma_v\circ p^{-1}$. 

We claim that $\sigma_v=\sigma_{\post}$. Suppose otherwise. Let $\Theta$ be the set of pairs $(x,y)\in V(T)\times V(T)$ such that $x$ and $y$ are incomparable in $(V(T),\preceq)$, $p(x)<p(y)$, and $\sigma_v(x)>\sigma_v(y)$. The assumption that $\sigma_v\neq\sigma_{\post}$ says that $\Theta\neq\emptyset$. If $(x,y)\in\Theta$, then $x$ and $y$ are separated by a red wall. Either $x$ and $\sigma_v^{-1}(\sigma_v(x)-1)$ are separated by the same red wall or $\sigma_v^{-1}(\sigma_v(x)-1)$ and $y$ are separated by the same red wall. It follows that one of the pairs $(x,\sigma_v^{-1}(\sigma_v(x)-1))$ or $(\sigma_v^{-1}(\sigma_v(x)-1),y)$ is in $\Theta$. By repeating this argument, we eventually find that there exists a pair $(x_0,y_0)\in\Theta$ with $\sigma_v(x_0)-\sigma_v(y_0)=1$. Letting $i=\sigma_v(y_0)$, we find that $s_i\circ \sigma_v$ is a linear extension of $V(T)$, so the permutation $s_iv=s_i\circ\sigma_v\circ p^{-1}$ is in the same $\equiv_\delta$-congruence class as $w$. However, $s_iv<_Lv$, and this contradicts the definition of $v$. 

We have shown that $v=\sigma_{\post}\circ p^{-1}$. Consequently, \[{\bf S}_{\equiv_\delta}(w)=w\circ v^{-1}=(\sigma_w\circ p^{-1})\circ(\sigma_{\post}\circ p^{-1})^{-1}=\sigma_w\circ\sigma_{\post}^{-1}=\mathcal P(\sigma_w)=\mathcal P(\mathcal I_\delta^{-1}(w)).\qedhere\]  
\end{proof}

\begin{figure}[ht]
  \begin{center}{\includegraphics[height=4.226cm]{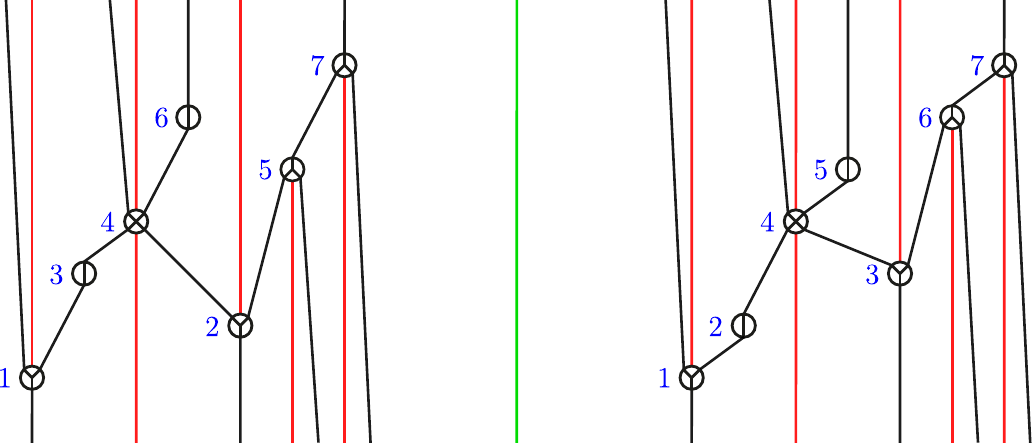}}
  \end{center}
  \caption{Decreasing $\delta$-permutrees with in-order readings $w=1346257$ (left) and $v=\pi_\downarrow^{\equiv_\delta}(w)=1245367$ (right).}\label{FigCoxx3}
\end{figure}

\begin{example}
Let $w=1346257$ and $\delta=\upCirc\noneCirc\upDownCirc\noneCirc\upCirc\downCirc\downCirc$. The decreasing $\delta$-permutree $\mathcal I_\delta^{-1}(w)$ is shown on the left in Figure~\ref{FigCoxx3}. On the right side of Figure~\ref{FigCoxx3} is $\mathcal I_\delta^{-1}(v)$, where $v=\pi_\downarrow^{\equiv_\delta}(w)=1245367$. Notice that for each vertex $x$, the label $\sigma_v(x)$ in the tree on the right is the same as $\sigma_{\post}(x)$. The postorder reading of the decreasing $\delta$-permutree on the left is $1324657$, which is equal to the permutation $wv^{-1}={\bf S}_{\equiv_\delta}(w)$.  
\end{example}

\section{Coxeter Stack-Sorting in Type $B$}\label{SecTypeB}

The Coxeter groups of type $B$ are the \dfn{hyperoctahedral groups} $B_n$, which are defined as follows. Consider the automorphism $\alpha$ of $S_{2n}$ given by $\alpha(w)=w_0ww_0$, where $w_0=(2n)(2n-1)\cdots 321$. This automorphism has the effect of rotating the plot of $w$ by $180^\circ$. The subgroup of $S_{2n}$ consisting of the permutations fixed by this automorphism is $B_n$. If we let $s_i=(i\,\,i+1)$ be the $i^\text{th}$ simple generator of $S_{2n}$, then the simple generators of $B_n$ are $s_1^B,\ldots,s_n^B$, where $s_i^B=s_is_{2n-i}$ for $i\in[n-1]$ and $s_n^B=s_n$. The simple generator $s_i^B$ is a right descent of an element $w\in B_n$ if and only if $w(i)>w(i+1)$ (equivalently, $w(2n-i)>w(2n-i+1)$). The automorphism $\alpha$ is a lattice automorphism of the left weak order on $S_{2n}$, so $(B_n,\leq_L)$ is a sublattice of $(S_{2n},\leq_L)$. Similarly, $(B_n,\leq_R)$ is a sublattice of $(S_{2n},\leq_R)$. 

Define the \dfn{complement} of a symbol $\Omega\in\{\noneCirc,\upCirc,\downCirc,\upDownCirc\}$ to be the symbol $\mho$ obtained by turning $\Omega$ upside-down. More precisely, $\mho=\Omega$ if $\Omega\in\{\noneCirc,\upDownCirc\}$, and $\mho$ is the unique element of $\{\upCirc,\downCirc\}\setminus\{\Omega\}$ if $\Omega\in\{\upCirc,\downCirc\}$. Let us say a decoration $\delta=\delta_1\cdots\delta_{2n}\in\{\noneCirc,\upCirc,\downCirc,\upDownCirc\}^{2n}$ is \dfn{antisymmetric} if for every $i\in[2n]$, the symbol $\delta_{2n+1-i}$ is the complement of $\delta_i$. For example, $\upCirc\downCirc\noneCirc\noneCirc\upCirc\downCirc$ is antisymmetric. We define a \dfn{centrally symmetric permutree} to be a permutree (with an even number of vertices) whose decoration is antisymmetric. A \dfn{decreasing centrally symmetric permutree} is a decreasing $\delta$-permutree $\mathcal T$ such that $\delta$ is antisymmetric and $\mathcal I_\delta(\mathcal T)\in B_n$ for some $n$. Given an antisymmetric decoration $\delta\in \{\noneCirc,\upCirc,\downCirc,\upDownCirc\}^{2n}$, define the \dfn{centrally symmetric $\delta$-permutree congruence} on $(B_n,\leq_L)$, denoted $\equiv_\delta^B$, to be the restriction of the $\delta$-permutree congruence $\equiv_\delta$ on $S_{2n}$ to $B_n$. Since $(B_n,\leq_L)$ is a sublattice of $(S_{2n},\leq_L)$, the equivalence relation $\equiv_\delta^B$ is a genuine lattice congruence on the left weak order of $B_n$. 

Observe that if $\delta$ is antisymmetric, then, in the notation of \eqref{EqCoxx1}, $\alpha(\equiv_\delta)$ is equal to $\equiv_\delta$. Therefore, it follows from \eqref{EqCoxx1} that $\pi_\downarrow^{\equiv_\delta}\circ\alpha=\alpha\circ\pi_\downarrow^{\equiv_\delta}$. This implies that if $w\in B_n$, then $\pi_\downarrow^{\equiv_\delta}(w)\in B_n$, so $\pi_\downarrow^{\equiv_\delta}(w)=\pi_\downarrow^{\equiv_\delta^B}(w)$. In other words, the minimal element of the $\delta$-permutree congruence class in $S_{2n}$ containing $w$ is the same as the minimal element of the centrally symmetric $\delta$-permutree congruence class in $B_n$ containing $w$. Hence, 
\begin{equation}\label{EqCoxx3}
{\bf S}_{\equiv_\delta}(w)={\bf S}_{\equiv_\delta^B}(w).
\end{equation} 
This fact is useful because it means we will be able to use Proposition~\ref{PropCoxx1} to compute ${\bf S}_{\equiv_\delta^B}(w)$ as a postorder reading. In symbols, we have 
\begin{equation}\label{EqCoxx2}
{\bf S}_{\equiv_\delta^B}(w)=\mathcal P\circ\mathcal I_\delta^{-1}(w)
\end{equation} for all $w\in B_n$.  

\begin{remark}
As mentioned in \cite{Pilaud}, the Cambrian congruences on $S_n$ are the $\delta$-permutree congruences with $\delta\in\{\upCirc,\downCirc\}^n$. One can show that the type-$B$ Cambrian congruences on $B_n$ discussed in \cite{ReadingCambrian} are precisely the centrally symmetric permutree congruences $\equiv_\delta^B$ with $\delta\in\{\upCirc,\downCirc\}^{2n}$ (and with $\delta$ antisymmetric). 
\end{remark} 

The quotient of the left weak order on $B_n$ by the lattice congruence $\equiv_{\upCirc^n\downCirc^n}^B$ is one of the Cambrian lattices that Reading called a \emph{type-$B$ Tamari lattice} due to its similarities with the classical $n^\text{th}$ Tamari lattice \cite{ReadingCambrian}. Since the $n^\text{th}$ Tamari lattice is the quotient of $(S_n,\leq_L)$ by the sylvester congruence, it is natural to call $\equiv_{\upCirc^n\downCirc^n}^B$ the \dfn{type-$B$ sylvester congruence} on $B_n$. We devote the remainder of this section to studying the Coxeter stack-sorting operator $\stack_B={\bf S}_{\equiv_{\upCirc^n\downCirc^n}^B}$ associated to the type-$B$ sylvester congruence, which one can view as a canonical type-$B$ analogue of West's stack-sorting map. 

\medskip

Our first theorem in this section regards the forward orbits of the dynamical system $\stack_B:B_n\to B_n$. Proposition~\ref{PropCox1} tells us that the size of every such forward orbit is at most the Coxeter number of $B_n$, which is $2n$. We will improve this result by showing that these forward orbits actually all have size at most $n+1$, and we will demonstrate that this bound is tight.

\begin{theorem}\label{ThmCoxx7}
For every $n\geq 1$, we have \[\max_{w\in B_n}\left\lvert O_{\stack_B}(w)\right\rvert =n+1.\]
\end{theorem}

\begin{proof}
Fix $w\in B_n$. To ease notation, let $v_t=\stack_B^t(w)$. Let $Z_t=\{m\in[n]:v_t^{-1}(m)\in[n]\}$. Since $v_t\in B_n$, we can describe $Z_t$ equivalently as the set of entries $m\in[n]$ such that $(m,2n+1-m)$ is not a left inversion of $v_t$ (viewing $v_t$ as an element of $S_{2n}$). Let $Z_t'$ be the set of integers $m\in [n]$ such that every entry appearing to the left of $m$ in $v_t$ is smaller than $m$. Equivalently, an entry $m\in [n]$ is in $Z_t'$ if and only if $v_t$ has no left inversions of the form $(m,b)$ with $m<b$; this implies that $Z_t'\subseteq Z_t$. It is immediate from Definition~\ref{DefCoxStackOp} (or alternatively, Proposition~\ref{PropCox1}) that $v_t\geq_R v_{t+1}$. This means that every left inversion of $v_{t+1}$ is a left inversion of $v_t$, so $Z_t\subseteq Z_{t+1}$ and $Z_t'\subseteq Z_{t+1}'$. 

Suppose $v_t\neq e$. Let $\delta=\upCirc^n\downCirc^n$, and recall that $v_{t+1}=\stack_B(v_t)={\bf S}_{\equiv_\delta^B}(v_t)=\mathcal P(\mathcal I_\delta^{-1}(w))$, where the last equality comes from \eqref{EqCoxx2}. We are going to prove that $Z_{t+1}'\setminus Z_t'\neq\emptyset$. We consider two cases. 

\medskip

\noindent {\bf Case 1:} Suppose $Z_t'\neq Z_t$. Let $a=\min(Z_t\setminus Z_t')$. We want to show that $a\in Z_{t+1}'$. Assume this is not the case. Then there exists an entry $b>a$ such that $(a,b)$ is a left inversion of $v_{t+1}$. Since $v_t\geq_R v_{t+1}$, the pair $(a,b)$ must also be a left inversion of $v_t$. Because $a$ is the smallest element of $Z_t\setminus Z_t'$, there cannot be an entry $a'<a$ appearing between $b$ and $a$ in $v_t$. This implies that there is no upward red wall separating the vertices with labels $b$ and $a$ in the decreasing permutree $\mathcal I_\delta^{-1}(v_t)$. There is also no downward red wall separating these two vertices because all vertices appearing horizontally between them are represented by the symbol $\upCirc$. It follows that the vertex with label $a$ is less than the vertex with label $b$ in the partial order $\preceq$, so $(a,b)$ is not a left inversion of $\mathcal P(\mathcal I_\delta^{-1}(w))=v_{t+1}$. This is a contradiction, so $a\in Z_{t+1}'\setminus Z_t'$. 

\medskip

\noindent {\bf Case 2:} Suppose $Z_t'=Z_t$. This means that every entry $m\in[n]$ such that $v_t^{-1}(m)\in[n]$ is greater than every entry to its left in $v_t$. In particular, this implies that if $m\in Z_t$, then every entry to the left of $m$ in $v_t$ is in $Z_t$. Hence, $Z_t=\{v_t(1),\ldots,v_t(k)\}$ for some $k$. Furthermore, $v_t(1)<\cdots<v_t(k)$. The assumption that $v_t\neq e$ forces $k<n$, so $v_t(n)\not\in Z_t$. This means that $v_t(n)\geq n+1$. Let $a=v_t(n+1)$. Because $v_t\in B_n$, we have $a=v_t(n+1)=2n+1-v_t(n)\leq n$. It follows that $s_n^B$ is a right descent of $v_t$. The map $\stack_B$ is compulsive by Proposition~\ref{PropCox1}, so we must have $v_{t+1}\leq_R v_ts_n^B$. Suppose, by way of contradiction, that $a\not\in Z_{t+1}'$. Then there exists a left inversion $(a,b)$ of $v_{t+1}$. Since $v_{t+1}\leq_R v_t$, the pair $(a,b)$ must also be a left inversion of $v_t$. Suppose there is an entry $a'<a$ appearing between $b$ and $a$ in $v_t$. Then $a'<a\leq n$ and $v_t^{-1}(a')<v_t^{-1}(a)=n+1$, so $a'\in Z_t=\{v_t(1),\ldots,v_t(k)\}$. Consequently, $b\in Z_t$. However, this is impossible because $v_1(1)<\cdots<v_t(k)$ and $b>a'$. This shows that no such entry $a'$ exists, so there is no upward red wall separating the vertices with labels $b$ and $a$ in the decreasing permutree $\mathcal I_\delta^{-1}(v_t)$. As in Case 1, there is also no downward red wall separating these two vertices, so we conclude in the same way as before that $(a,b)$ is not a left inversion of $\mathcal P(\mathcal I_\delta^{-1}(w))=v_{t+1}$. This is a contradiction, so $a\in Z_{t+1}'\setminus Z_t'$. 

\medskip 

We have demonstrated that if $t\geq 0$ is such that $v_t\neq e$, then $Z_t'\subsetneq Z_{t+1}'$. Let $j$ be the size of the forward orbit of $w$ under $\stack_B$. Then $v_0,v_1,\ldots,v_{j-2}$ are all not equal to $e$, so $Z_0'\subsetneq Z_1'\subsetneq\cdots\subsetneq Z_{j-1}'$. Hence, $j-1\leq \lvert Z_{j-1}'\rvert \leq n$. 

To complete the proof, we must exhibit a permutation $u\in B_n$ such that $\stack_B^{n-1}(u)\neq e$. Let \[u=(2n)23\cdots (2n-2)(2n-1)1.\] For $1\leq t\leq n-1$, it is straightforward to compute that $\stack_B^t(u)$ is obtained from $\stack_B^{t-1}(u)$ by shifting the entry $2n$ one space to the right and shifting the entry $1$ one space to the left. For example, if $n=4$, then $u=82345671$, $\stack_B(u)=28345617$, $\stack_B^2(u)=23845167$, and $\stack_B^3(u)=23481567$. For general $n$, we have $\stack_B^{n-1}(u)=23\cdots n(2n)1(n+1)(n+2)\cdots(2n-1)\neq e$. 
\end{proof}

Recall from \eqref{EqmaxDes} that if $w\in S_n$ is in the image of West's stack-sorting map $\stack$, then it has at most $\left\lfloor\frac{n-1}{2}\right\rfloor$ right descents. Our next goal is to understand the maximum number of right descents that an element of the image of $\stack_B$ can have. Recall from \eqref{EqCoxx3} that if $\delta=\upCirc^n\downCirc^n$, then $\stack_B:B_n\to B_n$, which is defined to be the Coxeter stack-sorting operator on $B_n$ associated to $\equiv_\delta^B$, is the restriction of the permutree stack-sorting operator ${\bf S}_{\equiv_\delta}:S_{2n}\to S_{2n}$ to $B_n$. Theorem~\ref{ThmCoxx1} tells us that if $w\in S_{2n}$, then ${\bf S}_{\equiv_\delta}(w)$ has at most $\left\lfloor\frac{2(2n-1)}{3}\right\rfloor$ right descents in $S_{2n}$. A permutation in $B_n$ with $k$ right descents in $B_n$ has either $2k$ or $2k-1$ right descents in $S_{2n}$. Therefore, it follows from Theorem~\ref{ThmCoxx1} that if $w\in B_n$, then $\stack_B(w)$ has at most $\left\lfloor\frac{2n}{3}\right\rfloor$ right descents in $B_n$. The next theorem improves upon this bound. In the remainder of this section, for $w\in B_n$, we write $D_R(w)$ for the right descent set of $w$ in $B_n$ and $\Des^A(w)$ for the set of indices $i\in[2n-1]$ such that $s_i$ is a right descent of $w$ in $S_{2n}$. 

\begin{theorem}\label{ThmCoxx2}
For each $n\geq 1$, we have \[\max_{w\in B_n}\left\lvert D_R\left(\stack_B(w)\right)\right\rvert =\left\lfloor\frac{n}{2}\right\rfloor.\]  
\end{theorem}

\begin{proof}
Let $\delta=\upCirc^n\downCirc^n$. Fix $w\in B_n$, and let $v=\stack_B(w)$. Our first goal is to prove that $\lvert D_R(v)\rvert \leq\left\lfloor\frac{n}{2}\right\rfloor$. Let $T$ be the skeleton of $\mathcal I_\delta^{-1}(w)$. This means that there is an associated linear ordering $p:V(T)\to[n]$ and a labeling $\sigma:V(T)\to[n]$ such that $w=\sigma\circ p^{-1}$. Let $\sigma_{\post}:V(T)\to[n]$ be the postorder of $V(T)$ so that, by \eqref{EqCoxx2}, we have $v=\mathcal P(\mathcal I_\delta^{-1}(w))=\sigma\circ\sigma_{\post}^{-1}$. We denote each edge in $T$ by $(x,y)$, where $x$ is the lower endpoint of the edge and $y$ is the upper endpoint. Observe that if $(x,y)$ is a right edge in $T$, then $\sigma_{\post}(y)=\sigma_{\post}(x)+1$. Since $\sigma$ is a linear extension, we have $v(\sigma_{\post}(x))=\sigma(x)<\sigma(y)=v(\sigma_{\post}(x)+1)$, so $\sigma_{\post}(x)\in[2n-1]\setminus\Des^A(v)$. Therefore, $\lvert \mathcal R(T)\rvert \leq 2n-1-\lvert \Des^A(v)\rvert $, where $\mathcal R(T)$ is the set of right edges in $T$. We are going to show that 
\begin{equation}\label{EqCoxx4}
\lvert \Des^A(v)\rvert \leq\lvert \mathcal R(T)\rvert .
\end{equation} 
Once we do this, it will follow that $\lvert \Des^A(v)\rvert \leq 2n-1-\lvert \Des^A(v)\rvert $, so $2\lvert \Des^A(v)\rvert \leq 2n-1$. This will then imply that $\lvert \Des^A(v)\rvert \leq\left\lfloor\frac{2n-1}{2}\right\rfloor=n-1$, so $\lvert D_R(v)\rvert =\left\lfloor\frac{\lvert \Des^A(v)\rvert +1}{2}\right\rfloor\leq\left\lfloor\frac{n}{2}\right\rfloor$, as desired.  

For each $i\in\Des^A(v)$, let $x_i$ and $y_i$ be the vertices of $T$ such that $\sigma(x_i)=v(i)$ and $\sigma(y_i)=v(i+1)$. The pair $(v(i+1),v(i))$ is a left inversion of $v$, so it must also be a left inversion of $w$ (since $v\leq_R w$). This implies that $p(x_i)<p(y_i)$ and $\sigma(x_i)>\sigma(y_i)$. We have \[\sigma_{\post}(x_i)=v^{-1}(\sigma(x_i))=i<i+1=v^{-1}(\sigma(y_i))=\sigma_{\post}(y_i).\] Because $\sigma$ and $\sigma_{\post}$ are both linear extensions of $V(T)$, this implies that $x_i$ and $y_i$ are incomparable in the poset $(V(T),\preceq)$. Hence, for every $i\in\Des^A(v)$, the vertices $x_i$ and $y_i$ are separated by a red wall. Let $G$ be the set of indices $i\in\Des^A(v)$ such that $x_i$ and $y_i$ are separated by an upward red wall. 

We are going to define an injection $\omega:G\to\mathcal R(T)$. Suppose $i\in G$, and let $z_i$ be the leftmost vertex (i.e., the vertex that minimizes $p(z_i)$) that emits an upward red wall that separates $x_i$ from $y_i$. It follows from the definition of a permutree that there is a right edge $(z_i,q_i)\in\mathcal R(T)$ for some vertex $q_i$. Let $\omega(i)=(z_i,q_i)$. To see that $\omega:G\to\mathcal R(T)$ is injective, suppose there is some $j\in G$ with $\omega(j)=\omega(i)$. Then $z_i=z_j$. The vertices $x_j$ and $y_i$ are separated by the upward red wall emitting by $z_i$, so we must have $\sigma_{\post}(x_j)<\sigma_{\post}(y_i)$. However, $\sigma_{\post}(x_j)=v^{-1}(\sigma(x_j))=j$ and $\sigma_{\post}(y_i)=v^{-1}(\sigma(y_i))=i+1$, so $j<i+1$. Thus, $j\leq i$. Reversing the roles of $i$ and $j$ shows that $i\leq j$, so $i=j$.  

We now want to construct an injection $\omega':\Des^A(v)\setminus G\to \mathcal R(T)$. Suppose $i\in \Des^A(v)\setminus G$. Then $x_i$ and $y_i$ are separated by a downward red wall. Let $q_i'$ be the rightmost vertex (i.e., the vertex that maximizes $p(q_i')$) that emits a downward red wall that separates $x_i$ from $y_i$. It follows from the definition of a permutree that there is a right edge $(z_i',q_i')\in\mathcal R(T)$ for some vertex $z_i'$. Let $\omega'(i)=(z_i',q_i')$. The proof that $\omega'$ is injective is virtually identical to the proof of the injectivity of $\omega$ that we gave in the previous paragraph. 

To prove \eqref{EqCoxx4}, it suffices to show that the image of $\omega$ is disjoint from the image of $\omega'$. This is where we use the fact that $T$ is a $\delta$-permutree, where $\delta=\upCirc^n\downCirc^n$. Suppose $(z_i,q_i)$ is in the image of $\omega$. Then $z_i$ emits an upward red wall, so it must be represented by the symbol $\upCirc$. Since $p(q_i)<p(z_i)$, the vertex $q_i$ is also represented by $\upCirc$. However, this implies that $q_i$ cannot emit a downward red wall, so $(z_i,q_i)$ is not in the image of $\omega'$. 

We have proven that \[\max_{w\in B_n}\left\lvert D_R\left(\stack_B(w)\right)\right\rvert \leq\left\lfloor\frac{n}{2}\right\rfloor.\] To see that this bound is tight, consider the permutation $v=13254\cdots(2n-1)(2n-2)(2n)\in B_n$. Let $w=vs_1s_3s_5\cdots s_{2n-1}=3152\cdots (2n-1)(2n-4)(2n)(2n-2)$. For example, if $n=5$, then $v=1\,3\,2\,5\,4\,7\,6\,9\,8\,10$ and $w=3\,1\,5\,2\,7\,4\,9\,6\,10\,8$. Using \eqref{EqCoxx2}, one can readily check that $\stack_B(w)=v$ (see Figure~\ref{FigCoxx4}). Furthermore, $v$ has $\left\lfloor\frac{n}{2}\right\rfloor$ right descents in $B_n$.   
\end{proof}

\begin{figure}[ht]
  \begin{center}{\includegraphics[height=4.226cm]{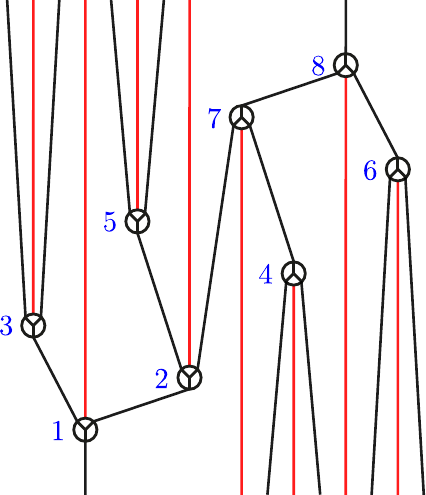}}
  \end{center}
  \caption{The decreasing permutree $\mathcal I_\delta^{-1}(w)$, where $w=31527486\in B_4$ and $\delta=\upCirc^4\downCirc^4$. The postorder reading of this tree is $v=13254768$. Note that $v$ has $2$ right descents in $B_4$; namely, $s_2^B$ and $s_4^B$.}\label{FigCoxx4}
\end{figure}

\section{Coxeter Stack-Sorting in Type $\widetilde A$}\label{SecAffine}

An \dfn{affine permutation} of size $n$ is a bijection $w:\mathbb Z\to\mathbb Z$ such that 
\[w(i+n)=w(i)+n\quad\text{for all }i\in\mathbb Z\] and 
\[\sum_{i=1}^nw(i)=\binom{n+1}{2}.\]
The set $\widetilde S_n$ of affine permutations of size $n$ forms a group under composition called the $n^\text{th}$ \dfn{affine symmetric group}; it is a Coxeter group of type~$\widetilde A_{n-1}$. The simple generators are $\widetilde s_1,\ldots,\widetilde s_n$, where $\widetilde s_i$ is the affine permutation that swaps $i+mn$ and $i+mn+1$ for all $m\in\mathbb Z$ and fixes all other integers. The simple generator $\widetilde s_i$ is a right descent of an affine permutation $w$ if and only if $w(i)>w(i+1)$. The \dfn{plot} of $w$ is the collection of points $(i,w(i))\in\mathbb R^2$ for all $i\in\mathbb Z$. We associate $w$ with its one-line notation, which is the infinite word $\cdots w(0)w(1)w(2)\cdots$. The \dfn{window notation} of $w$, which we write in brackets, is $[w(1),\ldots,w(n)]$. Note that $w$ is determined by its window notation. 

Define an \dfn{$n$-periodic decoration} to be an $n$-periodic bi-infinite word $\delta$ over $\{\noneCirc,\upCirc,\downCirc,\upDownCirc\}$. We think of the letters in this word as being indexed by the integers, so it is perhaps more helpful to think of $\delta$ as a function $\mathbb Z\to\{\noneCirc,\upCirc,\downCirc,\upDownCirc\}$, denoted by $i\mapsto\delta_i$, such that $\delta_i=\delta_{n+i}$ for all integers $i$. We define \dfn{affine permutrees} in exactly the same way that we defined  permutrees in Definition~\ref{DefPermutrees}; the only difference is that the set of vertices $V(T)$ is now countably infinite and the decoration $\delta$ is $n$-periodic instead of finite (when we use the word \emph{affine}, we tacitly assume there is a specific integer $n$ in the background). The linear ordering of the vertices of $T$ is now a bijection $p:V(T)\to\mathbb Z$. Associated to an affine permutree $T$ is a poset $(V(T),\preceq)$ obtained by saying $u\preceq v$ whenever there is a directed path (possibly with no edges) from $u$ to $v$. As in the symmetric group setting, two vertices $x$ and $y$ are incomparable in $(V(T),\preceq)$ if and only if they are separated by a red wall. We say a map $\sigma:V(T)\to\mathbb Z$ is an \dfn{affine linear extension} if $\sigma\circ p^{-1}$ is an affine permutation in $\widetilde S_n$ and $\sigma(u)\leq \sigma(v)$ whenever $u\preceq v$. A \dfn{decreasing affine permutree} is a pair $\mathcal T=(T,\sigma)$ such that $T$ is an affine permutree and $\sigma:V(T)\to[n]$ is an affine linear extension. We view $\sigma$ as a labeling of the vertices of $T$. When we draw a decreasing affine permutree, we will place the element with label $j$ at height $j$. We will also omit the infinite strands in our drawings to avoid clutter. See the left side of Figure~\ref{FigCoxx5}.

Define the \dfn{in-order reading} of the decreasing affine $\delta$-permutree $\mathcal T=(T,\sigma)$, denoted $\mathcal I_\delta(\mathcal T)$, to be the affine permutation $\sigma\circ p^{-1}$. In Section~\ref{SecPermutrees}, we described an insertion algorithm that creates a decreasing $\delta$-permtree from a permutation, thereby constructing the inverse of the in-order reading bijection. The same exact insertion algorithm works \emph{mutatis mutandis} in the affine setting. The main difference occurs when all of the symbols in the decoration are $\noneCirc$ or $\upCirc$. In this case, the algorithm begins at an arbitrary point $(i,w(i))$ in the plot of the affine permutation, which expels $1$ (if it is represented by $\noneCirc$) or $2$ (if it is represented by $\upCirc$) upward stands. Then the algorithm proceeds as before, and we can construct the part of the decreasing affine permutree below the point $(i,w(i))$ by $n$-periodically extending the part above $(i,w(i))$. Hence, the in-order reading defines a bijection from the set of decreasing affine permutrees with $n$-periodic decorations to $\widetilde S_n$ (the proof of this fact is identical to the proof Pilaud and Pons gave in the symmetric group setting \cite{Pilaud}, so we omit it). If $\mathcal I_\delta(\mathcal T)=w$, then the affine permutree is called the \dfn{$\delta$-skeleton} of $w$ and is denoted by $\skel_\delta(w)$. Define the \dfn{affine $\delta$-permutree congruence} $\equiv_\delta$ on $\widetilde S_n$ by saying $v\equiv_\delta w$ if and only if $\skel_\delta(v)=\skel_\delta(w)$. 

The affine $\delta$-permutree congruence is a semilattice congruence on the left weak order of $\widetilde S_n$. One can prove this in a manner virtually identical to how Pilaud and Pons proved that permutree congruences are lattice congruences on symmetric groups (they invoked results of Reading, but those results can also be adapted immediately to the affine setting). 

Suppose $T$ is an affine permutree with an $n$-periodic decoration $\delta$. The \dfn{postorder} of $V(T)$ is the unique affine linear extension $\sigma_{\post}:V(T)\to\mathbb Z$ such that $\sigma_{\post}(x)<\sigma_{\post}(y)$ whenever $x,y\in V(T)$ are incomparable and satisfy $p(x)<p(y)$. The \dfn{postorder reading} of an affine linear extension $\sigma:V(T)\to\mathbb Z$ is $\mathcal P(\sigma)=\sigma\circ\sigma_{\post}^{-1}$, which is an element of $\widetilde S_n$. Just as in the symmetric group setting, we define the postorder reading $\mathcal P(\mathcal T)$ of the decreasing affine permutree $\mathcal T=(T,\sigma)$ to be $\mathcal P(\sigma)$. The affine permutree congruence $\equiv_\delta$ gives rise to the \dfn{affine permutree stack-sorting operator} ${\bf S}_{\equiv_\delta}:\widetilde S_n\to\widetilde S_n$. The same argument used to prove Proposition~\ref{PropCoxx1} yields the fact that 
\begin{equation}\label{EqAffinePostorder}
{\bf S}_{\equiv_\delta}=\mathcal P\circ\mathcal I_\delta^{-1}.
\end{equation}

For the rest of this section, we will focus our attention on one specific affine permutree stack-sorting operator that is very similar to West's stack-sorting map. Let $\downCirc^{\mathbb Z}$ denote the $n$-periodic decoration whose symbols are all $\downCirc$. We call affine permutrees with decoration $\downCirc^{\mathbb Z}$ \dfn{affine binary plane trees}. We call $\equiv_{\downCirc^{\mathbb Z}}$ the \dfn{affine sylvester congruence} and denote it by $\equiv_{\widetilde{\syl}}$. The \dfn{affine stack-sorting map}, denoted $\widetilde\stack$, is the Coxeter stack-sorting operator ${\bf S}_{\equiv_{\widetilde{\syl}}}$. We say an affine permutation $w\in\widetilde S_n$ is \dfn{$t$-stack-sortable} if $\widetilde\stack^t(w)=e$. The only decoration symbol that will appear throughout the rest of this section is $\downCirc$. Therefore, we will write $\mathcal I(\mathcal T)$ with the understanding that this is $\mathcal I_{\downCirc^{\mathbb Z}}(\mathcal T)$ when $\mathcal T$ has infinitely many vertices and $\mathcal I_{\downCirc^k}(\mathcal T)$ when $\mathcal T$ has $k$ vertices. 

\begin{remark}
The affine sylvester congruence is not new, though to the best of our knowledge, it has not been given this name before. Motivated by an attempt to model certain geometric and combinatorial objects arising from cluster algebras, Reading and Speyer \cite{ReadingSpeyerCambrian2} considered the orientation $-\Omega$ of the type-$\widetilde A_{n-1}$ Coxeter diagram given by $\widetilde s_1\to\widetilde s_2\to\cdots\to\widetilde s_n\to\widetilde s_1$ (they call the reverse orientation $\Omega$). From this orientation, they constructed a map $\pi_\downarrow^{-\Omega}:\widetilde S_n\to\widetilde S_n$. Using the results in their paper, one can show that $\pi_\downarrow^{-\Omega}$ is the same as our downward projection map $\pi_\downarrow^{\equiv_{\widetilde{\syl}}}$. It follows that the fibers of the map $\pi_\downarrow^{-\Omega}$ are precisely the affine sylvester congruence classes. 
\end{remark}

Let $w$ be a permutation or an affine permutation. We say entries $b,c,a$ in $w$ form a \dfn{231-pattern} if they appear in the order $b,c,a$ in $w$ and satisfy $a<b<c$. We say $w$ is \dfn{231-avoiding} if no three entries form a $231$ pattern in $w$. Knuth \cite{Knuth} showed that a permutation is $1$-stack-sortable if and only if it is $231$-avoiding, and he proved that the number of $231$-avoiding permutations in $S_n$ is the Catalan number $\frac{1}{n+1}\binom{2n}{n}$. Crites investigated $231$-avoiding affine permutations (as well as other pattern-avoiding affine permutations) in \cite{Crites}, showing, in particular, that the number of $231$-avoiding elements of $\widetilde S_n$ is $\binom{2n-1}{n}$. 

Suppose $w\in\widetilde S_n$, and let $\mathcal I^{-1}(w)=\mathcal T=(T,\sigma)$. Let $p:V(T)\to\mathbb Z$ be the associated linear ordering of the vertices of the affine binary plane tree $T$. Fix integers $a<b$. Saying an integer $c$ is such that $b,c,a$ form a $231$-pattern in $w$ is equivalent to saying that $p(\sigma^{-1}(b))<p(\sigma^{-1}(a))$ and that $\sigma^{-1}(c)$ emits a downward red wall that separates $\sigma^{-1}(b)$ and $\sigma^{-1}(a)$. Hence, such an entry $c$ exists if and only if $\sigma_{\post}(\sigma^{-1}(b))<\sigma_{\post}(\sigma^{-1}(a))$. In other words, such an entry $c$ exists if and only if $b$ appears before $a$ in $\mathcal P(\mathcal T)$. Using \eqref{EqAffinePostorder}, we deduce the following simple proposition. This is an affine analogue of Knuth's characterization of $1$-stack-sortable permutations. We stress that the enumeration of $231$-avoiding affine permutations in the next proposition is due to Crites \cite{Crites}.

\begin{proposition}\label{PropCoxx3}
An affine permutation $w\in\widetilde S_n$ is $1$-stack-sortable if and only if it is $231$-avoiding. The number of $1$-stack-sortable elements of $\widetilde S_n$ is $\binom{2n-1}{n}$. 
\end{proposition}

We are going to prove an affine analogue of the fact that permutations of size $n$ in the image of the stack-sorting map can have at most $\left\lfloor\frac{n-1}{2}\right\rfloor$ right descents. The proof of this upper bound is very similar to the proof of the analogous result for symmetric groups, so we merely sketch it. 

\begin{theorem}\label{PropCoxx4}
For each $n\geq 1$, we have \[\max_{w\in\widetilde S_n}\lvert D_R(\widetilde\stack(w))\rvert =\left\lfloor\frac{n}{2}\right\rfloor.\]
\end{theorem} 

\begin{proof}
Choose $w\in\widetilde S_n$, and let $v=\widetilde\stack(w)$. Let $\Des(v)$ be the set of indices $i\in[n]$ such that $\widetilde s_i\in D_R(v)$. Let $\mathcal I^{-1}(w)=\mathcal T=(T,\sigma)$. Let $\sigma_{\post}:V(T)\to\mathbb Z$ be the postorder of $V(T)$. Suppose $\widetilde s_i$ is a right descent of $v$. Because $v=\mathcal P(\mathcal T)=\sigma\circ\sigma_{\post}^{-1}$, this means that $\sigma(\sigma_{\post}^{-1}(i))>\sigma(\sigma_{\post}^{-1}(i+1))$. Since $\sigma$ and $\sigma_{\post}$ are both affine linear extensions of $(V(T),\preceq)$, this forces $\sigma_{\post}^{-1}(i)$ and $\sigma_{\post}^{-1}(i+1)$ to be incomparable in $(V(T),\preceq)$. It follows that there is a vertex $x_i$ with two children in $T$ such that $\sigma_{\post}^{-1}(i)$ is the vertex in the left subtree of $x_i$ appearing last in the postorder. In other words, the maximum of $\sigma_{\post}(y)$ as $y$ ranges over the vertices in the left subtree of $x_i$ is $i$. Let $j_i=\sigma_{\post}(x_i')$, where $x_i'$ is the right child of $x_i$. Then $v(j_i)=\sigma(x_i')<\sigma(x_i)=v(j_i+1)$. Let $\overline j_i$ be the unique element of $[n]$ that is congruent to $j_i$ modulo $n$. The map $i\mapsto \overline j_i$ is an injection from $\Des(v)$ into $[n]\setminus\Des(v)$, so $\lvert \Des(v)\rvert \leq\left\lfloor\frac{n}{2}\right\rfloor$. This proves that $\max\limits_{w\in\widetilde S_n}\lvert D_R(\widetilde\stack(w))\rvert \leq\left\lfloor\frac{n}{2}\right\rfloor$. 

One could prove that this upper bound is tight by writing down explicit affine permutations, but we prefer a more theoretical approach. First, suppose $n$ is odd. Observe that there is a natural injection $\iota:S_n\to\widetilde S_n$ sending a permutation $u$ to the unique affine permutation with window notation $[u(1),\ldots,u(n)]$. The number of right descents of $u$ in $S_n$ is the same as the number of right descents of $\iota(u)$ in $\widetilde S_n$. For $u,u'\in S_n$, we have $u\equiv_{\syl}u'$ if and only if $\iota(u)\equiv_{\widetilde{\syl}}\iota(u')$. It follows that $\widetilde\stack\circ\iota=\iota\circ\stack$. We know already from \cite{DefantEngenMiller} that there exists a permutation $u\in S_n$ that lies in the image of $\stack$ and that has $\frac{n-1}{2}$ right descents (e.g., the permutation $214365\cdots(n-1)(n-2)n$). Then $\iota(u)$ is in the image of $\widetilde s$ and has $\frac{n-1}{2}$ right descents. 

Now suppose $n$ is even. Because $\widetilde S_2$ is infinite, it follows from Proposition~\ref{PropCoxx3} that there exists $w\in\widetilde S_2$ such that $\widetilde\stack(w)$ has $1$ right descent. There is an affine permutation $\widehat w\in\widetilde S_n$ that is equal to $w$ as a function from $\mathbb Z$ to $\mathbb Z$. The decreasing affine binary plane trees $\mathcal I^{-1}(w)$ and $\mathcal I^{-1}(\widehat w)$ are identical, so their postorder readings are identical. Thus, $\widetilde\stack(\widehat w)$ is equal to $\widetilde\stack(w)$ as a function. Since $\widetilde\stack(w)$ has $1$ right descent in $\widetilde S_2$, $\widetilde\stack(\widehat w)$ must have $\frac{n}{2}$ right descents as an element of $\widetilde S_n$.  
\end{proof}

Let $K_n$ be the maximum length of a $1$-stack-sortable affine permutation in $\widetilde S_n$; this is finite by Proposition~\ref{PropCoxx3}. Notice that $w\in\widetilde S_n$ is $1$-stack-sortable if and only if it is in the image of the downward projection map $\pi_\downarrow^{\equiv_{\widetilde{\syl}}}$ associated to the affine sylvester congruence. It follows that for every $v\in\widetilde S_n$, we have \[\ell({\bf S}_{\equiv_{\syl}}(v))=\ell\left(v\left(\pi_\downarrow^{\equiv_{\widetilde{\syl}}}(v)\right)^{-1}\right)=\ell(v)-\ell\left(\pi_\downarrow^{\equiv_{\widetilde{\syl}}}(v)\right)\geq\ell(v)-K_n.\] Because there are only finitely many elements of $\widetilde S_n$ of each length, this implies that each affine permutation has only finitely many preimages under $\widetilde\stack$. Hence, it makes sense to define the \dfn{fertility} of an affine permutation $w\in\widetilde S_n$ to be $\lvert \widetilde\stack^{-1}(w)\rvert $. 

The next theorem states that the fertility of an affine permutation only depends on its affine sylvester class. In fact, it says much more. Define the \dfn{affine stack-sorting tree} on $\widetilde S_n$, denoted $\widetilde {\mathscr S}_n$, to be the infinite rooted tree with vertex set $\widetilde S_n$ and with root vertex $e$ such that the parent of each non-root vertex $w$ is $\widetilde\stack(w)$. When we refer to the \emph{subtree of $\widetilde{\mathscr S}_n$ with root $x$}, we mean the rooted subtree of $\widetilde{\mathscr S}_n$ with root vertex $x$ and vertex set $\bigcup_{t\geq 0}\widetilde\stack^{-t}(x)$. 

\begin{theorem}\label{ThmCoxx3}
Suppose $v,w\in\widetilde S_n$ are such that $v\equiv_{\widetilde{\syl}} w$. Then the subtree of $\widetilde{\mathscr S}_n$ with root $v$ is isomorphic as a rooted tree to the subtree of $\widetilde{\mathscr S}_n$ with root $w$. In particular, $v$ and $w$ have the same fertility. 
\end{theorem}

\begin{proof}
Each affine sylvester class is convex (as defined in Section~\ref{SecDescents}) and contains a unique minimal element, so it suffices to prove the theorem when $\ell(w)=\ell(v)+1$. We are going to prove that if $z\in\widetilde S_n$ and $i\in[n]$ are such that $z\equiv_{\widetilde{\syl}}\widetilde s_iz$, then the map $u\mapsto \widetilde s_i u$ is a bijection from $\widetilde\stack^{-1}(z)$ to $\widetilde\stack^{-1}(\widetilde s_iz)$ such that $u\equiv_{\widetilde{\syl}}\widetilde s_iu$ for all $u\in\widetilde\stack^{-1}(z)$. It will then follow from repeated applications of this fact that the map $u\mapsto \widetilde s_iu$ is an isomorphism from the subtree of $\widetilde{\mathscr S}_n$ with root $z$ to the subtree of $\widetilde{\mathscr S}_n$ with root $\widetilde s_iz$. As this is true for all $i\in[n]$, this will prove the theorem in the case when $\ell(w)=\ell(v)+1$. 

Suppose $z\in\widetilde S_n$ and $i\in[n]$ are such that $z\equiv_{\widetilde{\syl}}\widetilde s_iz$. By switching the roles of $z$ and $\widetilde s_iz$ if necessary, we may assume $z\leq_L \widetilde s_iz$. By specializing the proof of Theorem~\ref{ThmWeakMono} to the setting of the affine stack-sorting map, we find that the map $u\mapsto \widetilde s_iu$ is an injection from $\widetilde\stack^{-1}(\widetilde s_iz)$ into $\widetilde\stack^{-1}(z)$. To prove that it is also surjective, choose $\widehat z\in\widetilde\stack^{-1}(z)$. Let $\mathcal I^{-1}(z)=\mathcal T=(T,\sigma)$. Let $\mathcal I^{-1}(\widehat z)=\widehat{\mathcal T}=(\widehat{T},\widehat\sigma)$. Then $z=\mathcal P(\widehat{\mathcal T})$ by \eqref{EqAffinePostorder}. Because $\ell(\widetilde s_iz)>\ell(z)$, the entry $i$ must appear before $i+1$ in $z$. Since $\widetilde s_iz\equiv_{\widetilde{\syl}}z$, the vertices $\sigma^{-1}(i)$ and $\sigma^{-1}(i+1)$ must be incomparable in the poset $(V(T),\preceq)$. This means that there is an integer $j>i$ such that $\sigma^{-1}(i)$ and $\sigma^{-1}(i+1)$ are separated by a downward red wall emitted by $\sigma^{-1}(j)$ (see Figure~\ref{FigCoxx10}). Now, $j$ appears between $i$ and $i+1$ in $z$, so $(i+1,j)$ is a left inversion of $z$. As $\widetilde\stack$ is a Coxeter stack-sorting operator, we have $\widehat z\geq_R\widetilde\stack(\widehat z)=z$. It follows that $(i+1,j)$ is a left inversion of $\widehat z$. Since $j$ appears before $i+1$ in $z$, the vertices $\widehat\sigma^{-1}(i+1)$ and $\widehat\sigma^{-1}(j)$ are incomparable in $(V(\widehat T),\preceq)$. Consequently, there is an integer $k>j$ such that $\widehat \sigma^{-1}(j)$ and $\widehat \sigma^{-1}(i+1)$ are separated by the downward red wall emitted by $\widehat \sigma^{-1}(k)$ (see Figure~\ref{FigCoxx10}). If $i$ appeared to the right of $k$ in $\widehat z$, then $\widehat \sigma^{-1}(j)$ and $\widehat \sigma^{-1}(i)$ would be separated by the downward red wall emitted by $\widehat \sigma^{-1}(k)$. However, this would force $j$ to appear before $i$ in $z$, which would be a contradiction. Hence, $i$ must appear to the left of $k$ in $\widehat z$. This means that $\widehat \sigma^{-1}(i)$ and $\widehat \sigma^{-1}(i+1)$ are separated by the downward red wall emitting by $\widehat \sigma^{-1}(k)$ in $\widehat T$. Since $\widehat{\mathcal T}$ is a decreasing affine binary plane tree, it follows that for every integer $m$, the vertices $\widehat \sigma^{-1}(i+mn)$ and $\widehat \sigma^{-1}(i+mn+1)$ are incomparable in $(V(\widehat T),\preceq)$. Hence, $\widetilde s_i \widehat z\equiv_{\widetilde{\syl}} \widehat z$. We deduce that $\pi_\downarrow^{\equiv_{\widetilde{\syl}}}(\widetilde s_i\widehat z)=\pi_\downarrow^{\equiv_{\widetilde{\syl}}}(\widehat z)$, so \[\widetilde\stack(\widetilde s_i\widehat z)=\widetilde s_i \widehat z\left(\pi_\downarrow^{\equiv_{\widetilde{\syl}}}(\widetilde s_i\widehat z)\right)^{-1}=\widetilde s_i \widehat z\left(\pi_\downarrow^{\equiv_{\widetilde{\syl}}}(\widehat z)\right)^{-1}=\widetilde s_i\widetilde\stack(\widehat z)=\widetilde s_i z.\] This proves that the map $u\mapsto \widetilde s_iu$ is a bijection from $\widetilde\stack^{-1}(\widetilde s_iz)$ to $\widetilde\stack^{-1}(z)$. Because $\widetilde s_i$ is an involution, the map $u\mapsto \widetilde s_iu$ is a bijection from $\widetilde\stack^{-1}(z)$ to $\widetilde\stack^{-1}(\widetilde s_iz)$.   
\end{proof}

\begin{figure}[ht]
\begin{center}
\includegraphics[height=4.274cm]{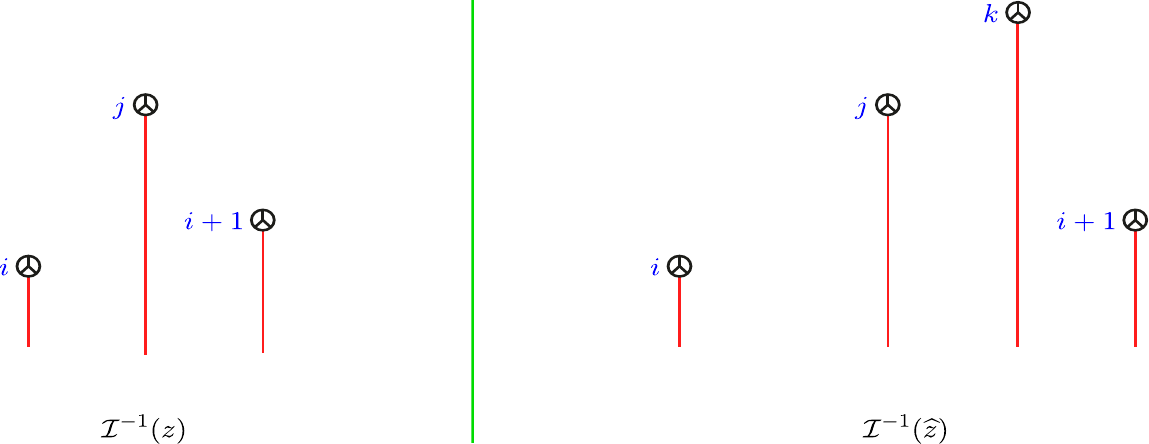}
\caption{A schematic illustration of the proof of Theorem~\ref{ThmCoxx3}. On the left are some of the vertices in $\mathcal I^{-1}(z)$. On the right are some of the vertices in $\mathcal I^{-1}(\widehat z)$.} 
\label{FigCoxx10}
\end{center}  
\end{figure}

The vertices of depth at most $1$ in $\widetilde{\mathscr S}_n$ are just the $1$-stack-sortable affine permutations (equivalently, the $231$-avoiding affine permutations) in $\widetilde S_n$. Since each affine sylvester class contains a unique $1$-stack-sortable affine permutation (i.e., its minimal element in the left weak order), we have the following corollary of Theorem~\ref{ThmCoxx3}. 

\begin{corollary}
Every subtree of $\widetilde{\mathscr S}_n$ is isomorphic as a rooted tree to a subtree of $\widetilde{\mathscr S}_n$ whose root has depth at most $1$. 
\end{corollary}

\begin{remark}
The natural analogue of Theorem~\ref{ThmCoxx3} is false for arbitrary Coxeter stack-sorting operators. For example, in $S_3$, we have $132\equiv_{\des} 231$, but $\left\lvert {\bf S}_{\equiv_{\des}}^{-1}(132)\right\rvert =1>0=\left\lvert {\bf S}_{\equiv_{\des}}^{-1}(231)\right\rvert $.   
\end{remark}

Our next goal in this section is to develop a method for computing the fertility of an arbitrary affine permutation. To do this, we first need to recall some basic facts about valid hook configurations, as discussed, for example, in \cite{DefantEngenMiller, DefantTroupes}. It will be convenient to consider permutations of arbitrary finite sets of integers. Thus, as discussed in Sections~\ref{SecStackBack} and~\ref{SecCoxeter}, a permutation of an $n$-element set $X$ is a bijection $w:[n]\to X$, which we represent in one-line notation as the word $w(1)\cdots w(n)$. Recall that we defined the stack-sorting map $\stack$ at the beginning of Section~\ref{SecStackBack} so that it operates on permutations in this more general setting. This means that it makes sense to consider the fertility $\lvert \stack^{-1}(w)\rvert $ of a permutation $w:[n]\to X$. It is immediate from the definition of $\stack$ in Section~\ref{SecStackBack} that permutations with the same standardization have the same fertility. Furthermore, the insertion algorithm described in Section~\ref{SecPermutrees} works just as well for permutations of $X$. Thus, for each permutation $w:[n]\to X$, there is a decreasing binary plane tree $\mathcal I^{-1}(w)=(T,\sigma)$, where now $\sigma$ is a bijection from $V(T)$ to $X$. We can also define the postorder reading of $\mathcal I^{-1}(w)$ to be the permutation $\sigma\circ\sigma_{\post}^{-1}:[n]\to X$, where $\sigma_{\post}:V(T)\to[n]$ is the postorder of $V(T)$. As before, $\mathcal P(\mathcal I^{-1}(w))=\stack(w)$ (this approach with non-standardized permutations is discussed in more detail in \cite{DefantTroupes}). 

Let $v$ be a permutation. When $i$ is a descent of $v$, we call the point $(i,v(i))$ a \dfn{descent top} of the plot of $v$ and call $(i+1,v(i+1))$ a \dfn{descent bottom} of the plot of $v$. A \dfn{hook} of $v$ is a rotated $\mathsf{L}$ shape connecting two points $(i,v(i))$ and $(j,v(j))$ with $i<j$ and $v(i)<v(j)$ (see Figure~\ref{Fig28}). The point $(i,v(i))$ is the \dfn{southwest endpoint} of the hook, and $(j,v(j))$ is the \dfn{northeast endpoint} of the hook. 

\begin{definition}\label{Def5}
Let $v$ be a permutation with descents $d_1<\cdots <d_k$. A \emph{valid hook configuration} of $v$ is a tuple $\mathcal H=(H_1,\ldots,H_k)$ of hooks of $v$ that satisfy the following properties: 
\begin{enumerate}
\item \label{Item1} For each $i\in[k]$, the southwest endpoint of $H_i$ is $(d_i,v(d_i))$. 
\item \label{Item2} No point in the plot of $v$ lies directly above a hook in $\mathcal H$. 
\item \label{Item3} No two hooks intersect or overlap each other unless the northeast endpoint of one is the southwest endpoint of the other. 
\end{enumerate}

Let $\VHC(v)$ denote the set of valid hook configurations of $v$.  
\end{definition} 

Suppose $\mathcal H=(H_1,\ldots,H_k)$ is a valid hook configuration of a permutation $v$ of size $n$. Draw a blue sky over the diagram depicting this valid hook configuration, and assign arbitrary distinct non-blue colors to all of the hooks. We are going to color all of the points in the plot of $v$ that are not northeast endpoints of hooks of $\mathcal H$. Imagine that each point that is not a northeast endpoint looks upward. If the point sees a hook, then it receives the same color as that hook. If it does not see a hook, then it must see the sky, so it receives the color blue. Note that if the point is the southwest endpoint of a hook, then it must look around (on the left side of) that hook when looking upward. Let $q_i$ be the number of points given the same color as $H_i$, and let $q_0$ be the number of points colored blue. Let ${\bf q}^{\mathcal H}=(q_0,\ldots,q_k)$. The plot of $v$ has $n$ points, $k$ of which did not get colored because they are northeast endpoints of hooks. Thus, ${\bf q}^{\mathcal H}$ is a composition of the integer $n-k$ into $k+1$ parts. 

\begin{figure}[ht]
\begin{center}
\includegraphics[height=7.408cm]{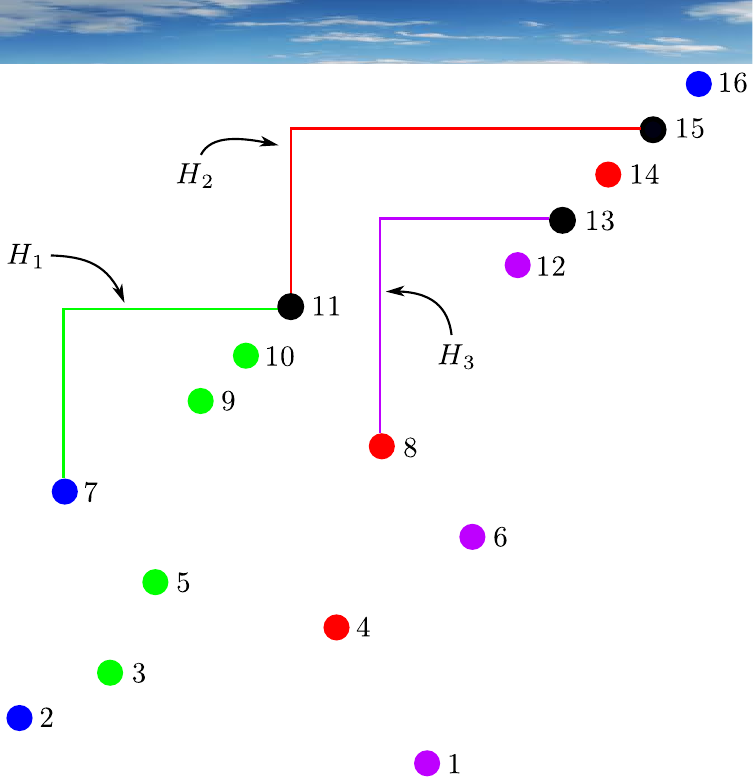}
\caption{A valid hook configuration and its induced coloring.} 
\label{Fig28}
\end{center}  
\end{figure} 

\begin{example}
Figure~\ref{Fig28} shows a valid hook configuration $\mathcal H=(H_1,H_2,H_3)$ of a permutation in $S_{16}$, along with its induced coloring. In this example, there are $3$ blue points, $4$ green points, $3$ red points, and $3$ purple points, so ${\bf q}^{\mathcal H}=(3,4,3,3)$. 
\end{example}

The following theorem provides a very useful tool for computing fertilities of permutations. It was originally proved in \cite{DefantPostorder}, though a clearer and more conceptual proof (which also applies in a much more general setting) appears in \cite{DefantTroupes}. Let $C_r$ denote the $r^\text{th}$ Catalan number $\frac{1}{r+1}\binom{2r}{r}$. Given an integer composition ${\bf q}=(q_0,\ldots,q_k)$, we write $C_{\bf q}$ for the product $\prod_{i=0}^kC_{q_i}$. 

\begin{theorem}[Fertility Formula]\label{FF}
The fertility of a permutation $v$ is given by \[\lvert \stack^{-1}(v)\rvert =\sum_{\mathcal H\in\VHC(v)}C_{{\bf q}^{\mathcal H}}.\] 
\end{theorem}

A simple consequence of the Fertility Formula that will be helpful to keep in mind is that a uniquely sorted permutation (i.e., a permutation with fertility $1$) must have a unique valid hook configuration. 

Let us now turn our attention back to affine permutations. Let us say a set $Y\subseteq\mathbb Z$ is \dfn{$n$-periodic} if for every integer $j$, we have $j\in Y$ if and only if $j+n\in Y$. Suppose $w\in\widetilde S_n$, and consider the decreasing affine binary plane tree $\mathcal I^{-1}(w)=\mathcal T=(T,\sigma)$. A \dfn{left-to-right maximum} of $w$ is an integer that is larger than every integer appearing to its left in $w$. Let $\cdots<m_0<m_1<m_2<\cdots$ be the left-to-right maxima of $w$. Note that the set of left-to-right maxima of $w$ is $n$-periodic. It follows directly from the insertion algorithm that constructs $\mathcal T$ from $w$ that $\sigma^{-1}(m_i)$ is a left child of $\sigma^{-1}(m_{i+1})$ for every integer $i$. Define the \dfn{infinite left branch} of $\mathcal T$ to be the collection of all the vertices $\sigma^{-1}(m_i)$ together with all the left edges between these vertices. Now let $v=\widetilde\stack(w)$. By \eqref{EqAffinePostorder}, we have $v=\mathcal P(\mathcal T)$. Every left-to-right maximum of $w$ is a left-to-right maximum of $v$. Indeed, the numbers appearing before $m_i$ in $v$ are the labels (given by $\sigma$) of the vertices in the left and right subtrees of $\sigma^{-1}(m_i)$ in $\mathcal T$, and all of these labels are necessarily smaller than $m_i$. Furthermore, if $\mathcal T_R(m_{i+1})$ denotes the (possibly empty) right subtree of $\sigma^{-1}(m_{i+1})$, then $\mathcal P(\mathcal T_R(m_{i+1}))$ is the (finite and possibly empty) permutation $z_i$ appearing between $m_i$ and $m_{i+1}$ in the affine permutation $\mathcal P(\mathcal T)=v$. In other words, if $u_i=\mathcal I(\mathcal T_R(m_{i+1}))$, then $z_i=\stack(u_i)$.  

The previous paragraph hints at how to construct each preimage $w$ of an affine permutation $v\in\widetilde S_n$ under $\widetilde\stack$. Start with $v$. Choose a nonempty $n$-periodic subset $\mathscr M=\{\cdots m_0<m_1<m_2<\cdots\}$ of the set of left-to-right maxima of $v$. This means that there is an integer $r$ such that $m_{j+r}=m_j+n$ for all $j$. For each index $i$, draw a hook $H_i$ connecting the points $(v^{-1}(m_i),m_i)$ and $(v^{-1}(m_{i+1}),m_{i+1})$ in the plot of $v$. This produces an infinite periodic chain of hooks on the plot of $v$ that we call a \dfn{skyline}. Now turn the skyline into an infinite left branch by ``unbending'' each hook. More precisely, form an infinite tree $\mathscr T$ whose edges are all left edges, and label the vertices of $\mathscr T$ with the elements of $\mathscr M$ so that the vertex with label $m_i$ is the left child of the vertex with label $m_{i+1}$ for all $i$. Let $z_i$ be the (finite) permutation whose plot lies underneath $H_i$ (not including the endpoints of $H_i$). Because $v\in\widetilde S_n$, the permutations $z_j$ and $z_{j+r}$ must have the same relative order (i.e., their standardizations are equal) for all $j$. More precisely, $z_{j+r}$ is obtained by increasing each of the entries in $z_j$ by $n$. This implies that we can choose permutations $u_i\in\stack^{-1}(z_i)$ for all $i$ in such a way that, for each integer $j$, the permutation $u_{j+r}$ is obtained by increasing each of the entries in $u_j$ by $n$. Finally, for each $i$, construct the decreasing binary plane tree $\mathcal I^{-1}(u_i)$, and attach it as the right subtree of the vertex of $\mathscr T$ with label $m_{i+1}$. After doing this for all $i$, we obtain a decreasing affine binary plane tree $\mathcal T$. Let $w=\mathcal I(\mathcal T)$. Our construction guarantees that $\widetilde\stack(w)=\mathcal P(\mathcal T)=v$. It follows from the preceding paragraph that every preimage of $v$ under $\widetilde\stack$ arises uniquely in this way.

\begin{figure}[ht]
\begin{center}
\includegraphics[height=10.144cm]{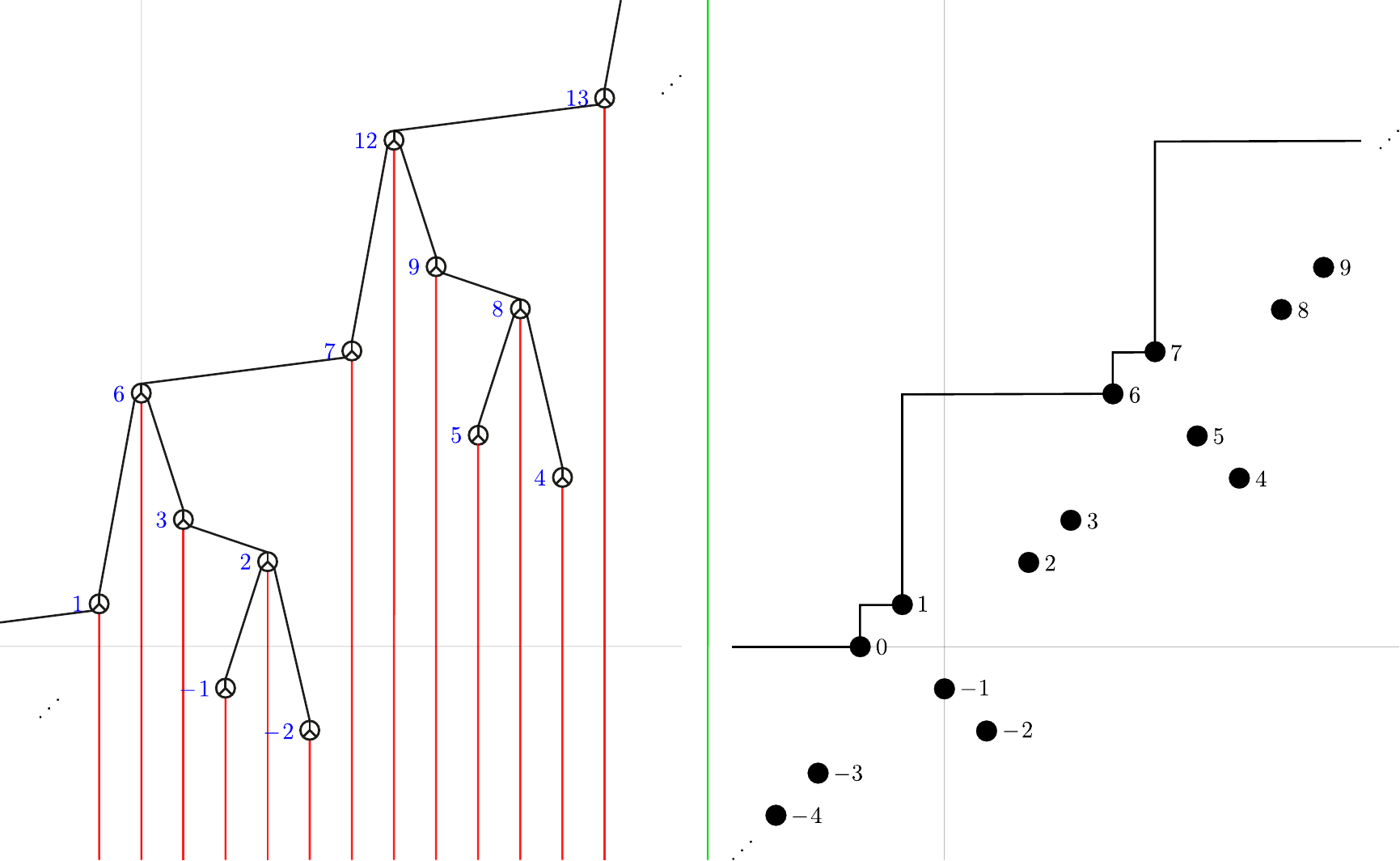}
\caption{On the left is the decreasing affine binary plane tree $\mathcal I^{-1}(w)$, where $w\in\widetilde S_6$ has window notation $[3,-1,2,-2,7,12]$. On the right is the plot of the affine permutation $v=\widetilde\stack(w)$, which has window notation $[-2,2,3,6,7,5]$, along with the skyline obtained by bending the edges in the infinite left branch of $\mathcal I^{-1}(w)$.} 
\label{FigCoxx5}
\end{center}  
\end{figure} 

\begin{example}\label{ExamCoxx2}
Let $v$ be the affine permutation in $\widetilde S_6$ with window notation $[-2,2,3,6,7,5]$. The plot of $v$ is shown on the right in Figure~\ref{FigCoxx5}. Let $\mathscr M=\{\cdots<m_0<m_1<m_2<\cdots\}$ be the $6$-periodic subset of the set of left-to-right maxima of $v$ given by $m_1=1$, $m_2=6$, and $m_{j+2}=m_j+6$ for all $j\in\mathbb Z$. We have drawn the corresponding hooks $H_i$ on the right side of Figure~\ref{FigCoxx5}. For example, $H_1$ has southwest endpoint $(-1,1)$ and northeast endpoint $(4,6)$, and $H_2$ has southwest endpoint $(4,6)$ and northeast endpoint $(5,7)$. Then $z_1=(-1)(-2)23$, and $z_2$ is the empty permutation. For each integer $j$, the permutation $z_{j+2}$ is obtained by increasing each of the entries in $z_j$ by $6$. Let us choose $u_1$ to be the permutation $3(-1)2(-2)\in\stack^{-1}(z_1)$. We are forced to choose $u_2$ to be empty. Our choices of $u_i$ for all $i\in\mathbb Z$ are then determined by the condition that $u_{j+2}$ is obtained by increasing each of the entries in $u_j$ by $6$. The decreasing affine binary plane tree $\mathcal T$ constructed from these choices is shown on the left in Figure~\ref{FigCoxx5}. The affine permutation $w=\mathcal I(\mathcal T)\in\widetilde S_6$ has window notation $[3,-1,2,-2,7,12]$. Finally, $\widetilde\stack(w)=\mathcal P(\mathcal T)=v$. 
\end{example} 

The discussion above yields an affine analogue of the Decomposition Lemma from \cite{DefantCounting, DefantTroupes}. In order to state it precisely, we fix some notation. Suppose $v\in\widetilde S_n$. Let $\Sky(v)$ denote the collection of all nonempty $n$-periodic subsets of the set of left-to-right maxima of $v$. The notation comes from the fact that each set $\mathscr M=\{\cdots m_0<m_1<m_2<\cdots\}$ in $\Sky(v)$ corresponds to a skyline of $v$; this is the infinite connected chain of hooks $\ldots,H_0,H_1,H_2,\ldots$, where $H_i$ has southwest endpoint $(v^{-1}(m_i),m_i)$ and northeast endpoint $(v^{-1}(m_{i+1}),m_{i+1})$. Let us make the convention that the indices of the elements of $\mathscr M$ are chosen so that $v^{-1}(m_1)\leq 1<v^{-1}(m_2)$. For such a set $\mathscr M$, we let $r(\mathscr M)$ be the index such that $m_{j+r(\mathscr M)}=m_j+n$ for all integers $j$. Also, let $z_i^{\mathscr M}$ be the permutation whose plot lies underneath $H_i$ (not including the endpoints of $H_i$). Note that $z_i^{\mathscr M}$ could be the empty permutation $\varepsilon$, in which case $\stack^{-1}(\varepsilon)=\{\varepsilon\}$. 

\begin{theorem}[Affine Decomposition Lemma]\label{ADL}
For every $v\in\widetilde S_n$, we have \[\lvert \widetilde\stack^{-1}(v)\rvert =\sum_{\mathscr M\in\Sky(v)}\prod_{i=1}^{r(\mathscr M)}\left\lvert \stack^{-1}\left(z_i^{\mathscr M}\right)\right\rvert .\]
\end{theorem}

If we want to use the Affine Decomposition Lemma to compute the fertility of an affine permutation $v\in\widetilde S_n$, then we need to be able to compute the fertilities of the finite permutations $z_i^{\mathscr M}$ appearing in the formula. We can do this using valid hook configurations via the Fertility Formula (Theorem~\ref{FF}). We first choose a valid hook configuration $\mathcal H_i$ for each $z_i^{\mathscr M}$ with $1\leq i\leq r(\mathscr M)$. Since each permutation $z_j^{\mathscr M}$ always has the same relative order as $z_{j+r(\mathscr M)}^\mathscr M$, we can copy the hooks from $\mathcal H_i$ onto the permutations $z_{i+\beta r(\mathscr M)}^{\mathscr M}$ for all $\beta\in\mathbb Z$ and all $1\leq i\leq r(\mathscr M)$ to produce a periodic configuration of hooks. Together with the hooks in the skyline, this produces a diagram $\mathcal H$ that we call an \dfn{affine valid hook configuration}. When we refer to the \dfn{skyline} of $\mathcal H$, we mean the skyline of $v$ used to construct $\mathcal H$. Let ${\bf q}^{\mathcal H}$ be the concatenation of the compositions ${\bf q}^{\mathcal H_1},\ldots,{\bf q}^{\mathcal H_{r(\mathscr M)}}$. Let $\AVHC(v)$ denote the set of affine valid hook configurations of $v$. By combining the Fertility Formula (Theorem~\ref{FF}) with the Affine Decomposition Lemma (Theorem~\ref{ADL}), we obtain the following Affine Fertility Formula. 

\begin{theorem}[Affine Fertility Formula]\label{AFF}
For every $v\in\widetilde S_n$, we have \[\lvert \widetilde\stack^{-1}(v)\rvert =\sum_{\mathcal H\in\AVHC(v)}C_{{\bf q}^{\mathcal H}}.\] 
\end{theorem}

\begin{figure}[ht]
\begin{center}
\includegraphics[height=8.616cm]{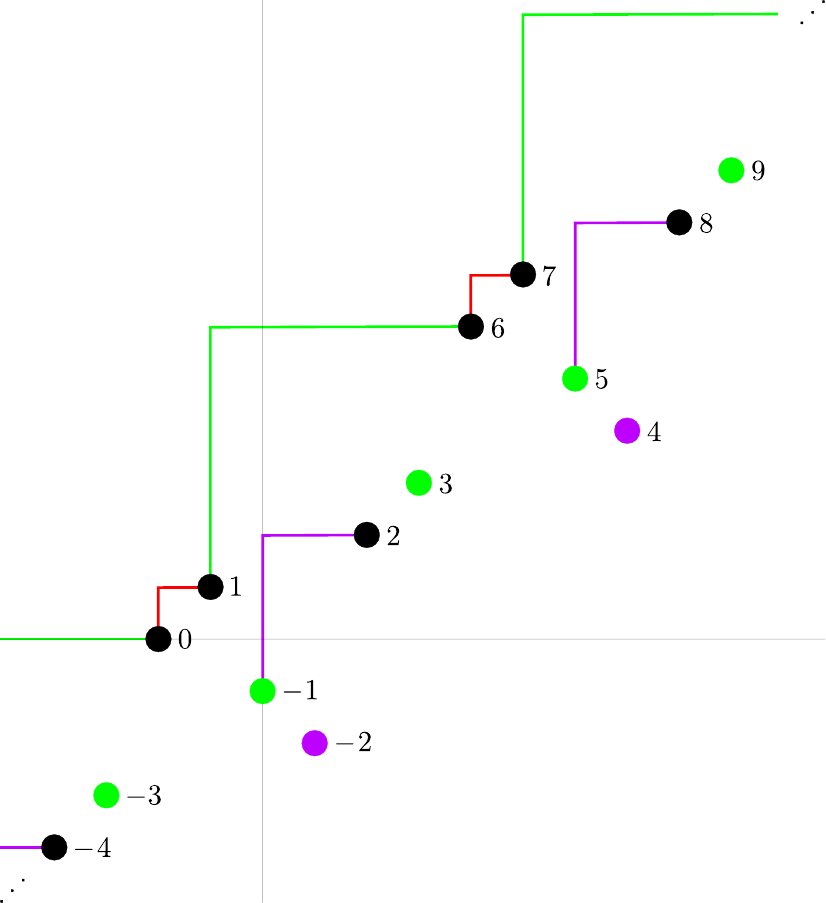}
\caption{An affine valid hook configuration and its induced coloring.} 
\label{FigCoxx6}
\end{center}  
\end{figure} 

\begin{example}
Preserve the notation from Example~\ref{ExamCoxx2}. We have $z_1=(-1)(-2)23$ and $z_2=\varepsilon$. Consider the valid hook configuration $\mathcal H_1$ of $z_1$ that consists of a single hook whose southwest endpoint has height $-1$ and whose northeast endpoint has height $2$. If we add this valid hook configuration to the image shown on the right of Figure~\ref{FigCoxx5}, and then repeat it $6$-periodically, we obtain the affine valid hook configuration of $v$ shown in Figure~\ref{FigCoxx6}. Note that $\mathcal H_2$ is the empty valid hook configuration of the empty permutation. The composition induced by the valid hook configuration ${\bf q}^{\mathcal H_1}$ is $(2,1)$. The composition induced by $\mathcal H_2$ is the empty composition. The concatenation of ${\bf q}^{\mathcal H_1}=(2,1)$ and the empty composition ${\bf q}^{\mathcal H_2}$ is ${\bf q}^{\mathcal H}=(2,1)$. This corresponds to the fact that, among the points $(1,v(1)),\ldots,(6,v(6))$, there are $2$ colored green and $1$ colored purple in Figure~\ref{FigCoxx6}. Hence, the affine valid hook configuration of $v$ shown in that figure contributes the term $C_{(2,1)}=C_2C_1=2$ to the sum in the Affine Fertility Formula. 
\end{example}

Before we proceed, let us pause to collect some observations. Consider an affine valid hook configuration $\mathcal H$ of an affine permutation $v$. The hooks in the skyline of $\mathcal H$ are precisely the hooks that do not lie underneath any other hooks. Also, just as for ordinary valid hook configurations, no hook in $\mathcal H$ passes underneath a point in the plot of $v$. No two hooks in $\mathcal H$ intersect or overlap each other, except when the southwest endpoint of one hook coincides with the northeast endpoint of the other. Every descent top of the plot of $v$ must be a southwest endpoint of a hook in $\mathcal H$. The main difference between ordinary and affine valid hook configurations is that, in the affine setting, not all southwest endpoints of hooks need to be descent tops. Namely, it is possible to have a southwest endpoint of a hook in the skyline of $\mathcal H$ that is not a descent top of the plot of $v$. However, southwest endpoints of hooks that are \emph{not} in the skyline \emph{are} required to be descent tops. For our final observation, let $\mathscr M\in\Sky(v)$ be the set corresponding to the skyline of $\mathcal H$. Suppose that each of the permutations $z_i^{\mathscr M}$ defined above is uniquely sorted. By the Fertility Formula (Theorem~\ref{FF}), this implies that each of the compositions ${\bf q}^{\mathcal H_i}$ has all its parts equal to $1$, so the composition ${\bf q}^{\mathcal H}$ has all its parts equal to $1$.

We now apply the Affine Decomposition Lemma and Affine Fertility Formula to characterize (for even $n$) which affine permutations attain the maximum in Theorem~\ref{PropCoxx4}. Let us say an affine permutation in $\widetilde S_n$ is \dfn{uniquely sorted} if its fertility is $1$. 

\begin{theorem}\label{ThmCoxx4}
An affine permutation in $\widetilde S_n$ is uniquely sorted if and only if it is in the image of $\widetilde\stack$ and has exactly $\frac{n}{2}$ descents. In particular, there are no uniquely sorted affine permutations in $\widetilde S_n$ when $n$ is odd. 
\end{theorem}

\begin{proof}
First, suppose $v\in\widetilde S_n$ is uniquely sorted. It follows from the Affine Fertility Formula (Theorem~\ref{AFF}) that $v$ has a unique affine valid hook configuration $\mathcal H$. Let $\mathscr M=\{\cdots<m_0<m_1<m_2<\cdots\}\in\Sky(v)$ be the set of left-to-right maxima of $v$ corresponding to the skyline of $\mathcal H$. Let $H_i$ be the hook of $\mathcal H$ with endpoints $(v^{-1}(m_i),m_i)$ and $(v^{-1}(m_{i+1}),m_{i+1})$, and let $z_i=z_i^{\mathscr M}$ be the permutation whose plot lies underneath $H_i$. Let $r=r(\mathscr M)$ be the integer such that $m_{j+r}=m_j+n$ for all integers $j$. Suppose, by way of contradiction, that there is an integer $i$ such that $z_i$ is empty. By periodicity, $z_{i+\beta r}$ is empty for all integers $\beta$. For each integer $\beta$, delete the hooks $H_{i+\beta r-1}$ and $H_{i+\beta r}$, and replace them with a single hook that has southwest endpoint $(v^{-1}(m_{i+\beta r-1}),m_{i+\beta r-1})$ (the former southwest endpoint of $H_{i+\beta r-1}$) and northeast endpoint $(v^{-1}(m_{i+\beta r+1}),m_{i+\beta r+1})$ (the former northeast endpoint of $H_{i+\beta r+1}$). This results in a new affine valid hook configuration of $v$, contradicting the uniqueness of $\mathcal H$. Hence, none of the permutations $z_i$ are empty. It follows that each of the points $(v^{-1}(m_i),m_i)$ in the skyline of $\mathcal H$ is a descent top of the plot of $v$. Now let $n_i$ be the number of entries of $z_i$. The sum $\sum_{i=1}^r(n_i+1)$ counts the entries in $z_1,\ldots,z_r$ and the entries $m_1,\ldots,m_r$, so it is equal to $n$. It follows from the Affine Decomposition Lemma (Theorem~\ref{ADL}) that \[1=\lvert \widetilde\stack^{-1}(v)\rvert =\prod_{i=1}^r\left\lvert \stack^{-1}(z_i)\right\rvert ,\] so each permutation $z_i$ is uniquely sorted. We have shown that each $z_i$ is nonempty, so each $n_i$ is positive. According to Proposition~\ref{PropUniquely}, each permutation $z_i$ has exactly $\frac{n_i-1}{2}$ descents (in particular, $n_i$ is odd). Therefore, the total number of right descents of $v$ is $|\{v^{-1}(m_1),\ldots,v^{-1}(m_r)\}|+\sum_{i=1}^r\frac{n_i-1}{2}=r+\sum_{i=1}^r\frac{n_i-1}{2}=\frac{1}{2}\sum_{i=1}^r(n_i+1)=\frac{n}{2}$. 

\medskip

To prove the converse, assume that $v\in\widetilde S_n$ is in the image of $\widetilde\stack$ and has exactly $\frac{n}{2}$ right descents. Because $v$ is in the image of $\widetilde\stack$, the Affine Fertility Formula (Theorem~\ref{AFF}) tells us that it has an affine valid hook configuration $\mathcal H$. Let $\mathscr M=\{\cdots<m_0<m_1<m_2<\cdots\}\in\Sky(v)$, $H_i$, $z_i$, $r$, and $n_i$ have the same meanings as in the previous paragraph. Let $Y$ be the set of indices $i\in[r]$ such that the southwest endpoint of $H_i$ is a descent top. Let $Y'$ be the set of indices $i\in[r]$ such that $z_i$ is nonempty. Note that $Y\subseteq Y'$. The hooks in $\mathcal H$ that lie underneath $H_i$ form a valid hook configuration of $z_i$, so it follows from the Fertility Formula (Theorem~\ref{FF}) that $z_i$ is in the image of $\stack$. If $i\in Y'$, then since $z_i$ is a nonempty permutation of size $n_i$ in the image of $\stack$, we have $\lvert \Des(z_i)\rvert \leq\frac{n_i-1}{2}$ by \eqref{EqmaxDes}. The number of right descents of $v$ is 
\begin{equation}\label{EqCoxx5}
\lvert Y\rvert +\sum_{i\in Y'}\lvert \Des(z_i)\rvert \leq r+\sum_{i\in Y'}\frac{n_i-1}{2}\leq r+\sum_{i=1}^r\frac{n_i-1}{2}=\frac{1}{2}\sum_{i=1}^r(n_i+1)=\frac{n}{2}.
\end{equation} However, our assumption is that $v$ has exactly $\frac{n}{2}$ right descents, so the inequalities in \eqref{EqCoxx5} must be equalities. It follows that $Y=Y'=[r]$ and $\lvert \Des(z_i)\rvert =\frac{n_i-1}{2}$ for all $i\in[r]$. By periodicity, we conclude that every southwest endpoint of a hook in $\mathcal H$ is a descent top of the plot of $v$ (this is always true of southwest endpoints not in the skyline) and that every permutation $z_i$ for $i\in\mathbb Z$ has $\frac{n_i-1}{2}$ descents. According to Proposition~\ref{PropUniquely}, each of the permutations $z_i$ for $i\in\mathbb Z$ is uniquely sorted. It follows from the last observation stated immediately before this theorem that the composition ${\bf q}^{\mathcal H}$ has all its parts equal to $1$. Therefore, if we can prove that $\mathcal H$ is the only affine valid hook configuration of $v$, then it will follow from the Affine Fertility Formula that $\lvert \widetilde\stack^{-1}(v)\rvert =C_{{\bf q}^{\mathcal H}}=1$, as desired.

Let $G$ be the set of points $(i,v(i))$ such that $i\in[n]$. Let $\DB$ denote the set of descent bottoms of the plot of $v$. Let $\mathfrak N$ be the set of northeast endpoints of hooks in $\mathcal H$. If there were a point $(i,v(i))\in\DB\cap\,\mathfrak N$, then the hook with northeast endpoint $(i,v(i))$ would pass underneath the point $(i-1,v(i-1))$, contradicting the properties of an affine valid hook configuration. Therefore, $\DB$ and $\mathfrak N$ are disjoint. The assumption that $v$ has $\frac{n}{2}$ right descents implies that $\lvert \DB\cap\,G\rvert =\frac{n}{2}$. Furthermore, every descent top of the plot of $v$ is a southwest endpoint of a hook in $\mathcal H$, so we must have $\lvert \mathfrak N\cap G\rvert \geq\frac{n}{2}$. Since $(\DB\cap\,G)\cup(\mathfrak N\cap G)\subseteq G$ and $\lvert G\rvert =n$, we must have $\lvert \mathfrak N\cap G\rvert =\frac{n}{2}$. Thus, $\DB\cap \,G$ and $\mathfrak N\cap G$ form a partition of the set $G$. By periodicity, we find that $\DB$ and $\mathfrak N$ form a partition of the set of all points in the plot of $v$. 

Recall that our goal is to prove that $\mathcal H$ is the unique affine valid hook configuration of $v$. Suppose instead that there is some other $\mathcal H'\in\AVHC(v)$. Let $\mathfrak N'$ be the set of northeast endpoints of hooks in $\mathcal H'$. We have seen that the southwest endpoints of hooks in $\mathcal H$ are precisely the descent tops of the plot of $v$ and that $\mathfrak N=\{(i,v(i)):i\in\mathbb Z\}\setminus\DB$. Applying the same arguments to $\mathcal H'$ instead of $\mathcal H$ shows that the set of southwest endpoints of hooks in $\mathcal H'$ is also equal to the set of descent tops of the plot of $v$ and that $\mathfrak N'=\{(i,v(i)):i\in\mathbb Z\}\setminus\DB=\mathfrak N$. Since $\mathcal H\neq\mathcal H'$, there must be a descent top $(d,v(d))$ such that the hook $H$ of $\mathcal H$ with southwest endpoint $(d,v(d))$ is not equal to the hook $H'$ of $\mathcal H'$ with southwest endpoint $(d,v(d))$. By switching the roles of $\mathcal H$ and $\mathcal H'$ if necessary, we may assume that $H$ lies underneath $H'$ (i.e., the northeast endpoint of $H'$ is higher than that of $H$). Let $(j,v(j))$ be the northeast endpoint of $H$. Let $Z$ be the set of points lying underneath $H$ (not including the endpoints of $H$), and let $Z^*=Z\cup\{(j,v(j))\}$. Observe that the leftmost point in $Z^*$ is $(d+1,v(d+1))$, which is immediately to the right of the southwest endpoint of $H'$. This implies that every hook of $\mathcal H'$ whose northeast endpoint lies in $Z^*$ has its southwest endpoint in $Z$. Therefore, $\lvert \mathfrak N'\cap Z^*\rvert $ is at most the number of descent tops in $Z$. The number of descent tops in $Z$ is the number of hooks of $\mathcal H$ whose endpoints are in $Z$, which is also equal to $\lvert \mathfrak N\cap Z\rvert $. Since $\mathfrak N=\mathfrak N'$, this implies that $\lvert \mathfrak N\cap Z^*\rvert =\lvert \mathfrak N'\cap Z^*\rvert \leq \lvert \mathfrak N\cap Z\rvert $. This is our desired contradiction because $\mathfrak N\cap Z^*=(\mathfrak N\cap Z)\cup\{(j,v(j))\}$.    
\end{proof}

It is not difficult to prove that every affine sylvester class in $\widetilde S_n$ is infinite when $n\geq 2$. Therefore, it follows from Theorems~\ref{ThmCoxx3} and~\ref{ThmCoxx4} that the number of uniquely sorted affine permutations in $\widetilde S_n$ is either $0$ or $\infty$, depending on the parity of $n$. This means that the na\"ive affine analogue of the fact that uniquely sorted permutations are enumerated by Lassalle's sequence \cite{DefantEngenMiller} is not so interesting from an enumerative point of view. However, we can still obtain an interesting enumerative result by considering affine sylvester classes as objects in their own right. Theorem~\ref{ThmCoxx3} tells us that if an affine sylvester class contains a uniquely sorted affine permutation, then \emph{all} of the affine permutations in that class are uniquely sorted. Hence, we define an affine sylvester class to be \dfn{uniquely sorted} if its elements are uniquely sorted. 

Every affine sylvester class contains a unique $231$-avoiding affine permutation, so enumerating uniquely sorted affine sylvester classes is the same as enumerating uniquely sorted $231$-avoiding affine permutations. In the symmetric group setting, $231$-avoiding uniquely sorted permutations were enumerated in \cite{DefantCatalan}, where it was shown that they are in bijection with intervals in Tamari lattices. In order to obtain an analogous result in the affine setting, we first need to understand the structure of $231$-avoiding affine permutations.  

Given $v\in\widetilde S_n$, define $\shift(v)$ to be the affine permutation in $\widetilde S_n$ given by $\shift(v)(i)=v(i+1)-1$ for all integers $i$. We are going to make use of the injection $\iota:S_n\to\widetilde S_n$ sending each permutation $u$ to the unique affine permutation in $\widetilde S_n$ that agrees with $u$ on $[n]$ (i.e., the affine permutation with window notation $[u(1),\ldots,u(n)]$). The \dfn{direct sum} of permutations $u\in S_m$ and $v\in S_n$ is the permutation $u\oplus v\in S_{m+n}$ defined by \[(u\oplus v)(i)=\begin{cases} u(i) & \mbox{if } 1\leq i\leq m; \\ v(i-m)+m & \mbox{if } m+1\leq i\leq m+n. \end{cases}\]

Suppose $\kappa\geq 1$ is an integer. It follows from the definition of a valid hook configuration that if $x\in S_{\kappa+1}$ is a permutation that has a valid hook configuration, then $x(\kappa+1)=\kappa+1$. Let $\Upsilon_{\kappa+1}$ be the set of $231$-avoiding uniquely sorted permutations $x\in S_{\kappa+1}$ such that the unique valid hook configuration of $x$ has a hook with southwest endpoint $(1,x(1))$ and northeast endpoint $(\kappa+1,\kappa+1)$. Let $\Upsilon_\kappa^*$ be the set of permutations in $S_\kappa$ that can be obtained from a permutation $x\in\Upsilon_{\kappa+1}$ by deleting the last entry $\kappa+1$. It will be convenient to define $\Upsilon_0^*$ to be the empty set.  

Now suppose $v$ is a $231$-avoiding affine permutation in $\widetilde S_n$. Let $\mathcal H$ be the unique affine valid hook configuration of $v$, and let $\mathscr M=\{\cdots<m_0<m_1<m_2<\cdots\}\in\Sky(v)$ be the set of left-to-right maxima of $v$ corresponding to the skyline of $\mathcal H$. Let $H_i$, $z_i=z_i^{\mathscr M}$, and $r=r(\mathscr M)$ be as described above. Assume the indices are chosen so that $m_1\leq 1<m_2$. Let $y_i$ be the permutation whose plot consists of the points lying underneath $H_i$ together with the southwest endpoint of $H_i$. Let $x_i$ be the permutation whose plot consists of the points lying underneath $H_i$ together with both endpoints of $H_i$. Let $\widehat z_i$, $\widehat y_i$, and $\widehat x_i$ be the standardizations of $z_i$, $y_i$, and $x_i$, respectively. Let $\kappa_i$ be the size of $y_i$ (so $\kappa_i-1$ is the size of $z_i$, and $\kappa_i+1$ is the size of $x_i$). We saw in the proof of Theorem~\ref{ThmCoxx4} that $z_i$ is nonempty and uniquely sorted and that the southwest endpoint of $H_i$ is a descent top. It follows from Proposition~\ref{PropUniquely} that $z_i$ has $\frac{\kappa_i-2}{2}$ descents. Therefore, $x_i$ has $\frac{\kappa_i}{2}$ descents. The hooks in $\mathcal H$ that lie on points in the plot of $x_i$ form a valid hook configuration of $x_i$, so $x_i$ is in the image of $\stack$ by the Fertility Formula (Theorem~\ref{FF}). It follows from Proposition~\ref{PropUniquely} that $x_i$ is uniquely sorted. Consequently, $\widehat x_i\in\Upsilon_{\kappa_i+1}$. Also, $\widehat y_i$ is obtained by deleting the entry $\kappa_i+1$ from $\widehat x_i$, so $\widehat y_i\in\Upsilon_{\kappa_i}^*$. Because $H_i$ is in the skyline of $\mathcal H$, its southwest endpoint is higher up than all points to its left. Since $v$ is $231$-avoiding, every point in the plot of $v$ appearing to the left of the points in the plot of $y_i$ must also be below all of the points in the plot of $y_i$. Similarly, every point in the plot of $v$ appearing to the right of the points in the plot of $y_i$ must also be above all of the points in the plot of $y_i$. We deduce that $v$ can be written as $\shift^{\beta}(\iota(u))$, where $u=\widehat y_1\oplus\cdots\oplus \widehat y_r\in S_n$ and each $\widehat y_i$ is in $\Upsilon_{\kappa_i}^*$. Furthermore, there is a unique choice of $\beta$ such that $0\leq\beta\leq\kappa_1-1$. 

\begin{figure}[ht]
\begin{center}
\includegraphics[height=6.504cm]{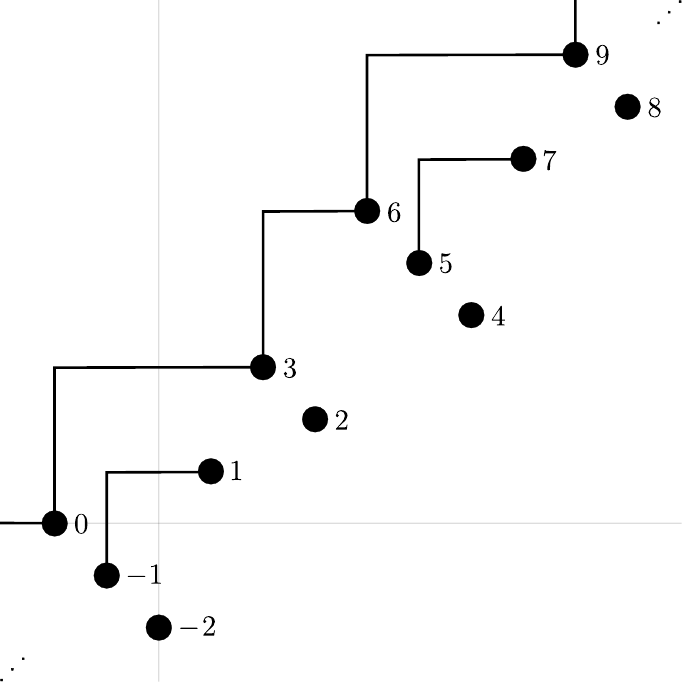}
\caption{The unique affine valid hook configuration of a $231$-avoiding uniquely sorted affine permutation in $\widetilde S_6$.} 
\label{FigCoxx7}
\end{center}  
\end{figure} 

\begin{example}
Consider the $231$-avoiding uniquely sorted affine permutation $v\in\widetilde S_6$ with window notation $[1,3,2,6,5,4]$. The unique affine valid hook configuration $\mathcal H$ of $v$ appears in Figure~\ref{FigCoxx7}. The hook $H_1$ has endpoints $(-2,0)$ and $(2,3)$. The hook $H_2$ has endpoints $(2,3)$ and $(4,6)$. Also, $r=2$. We have $z_1=(-1)(-2)1$, $y_1=0(-1)(-2)1$, and $x_1=0(-1)(-2)13$, so $\widehat z_1=213$, $\widehat y_1=3214$, and $\widehat x_1=32145$. Also, $\kappa_1=4$. Similarly, we have $z_2=2$, $y_2=32$, $x_2=326$, $\widehat z_2=1$, $\widehat y_2=21$, $\widehat x_2=213$, and $\kappa_2=2$. Then $\widehat x_1\in\Upsilon_5$ and $\widehat x_2\in\Upsilon_3$, so $\widehat y_1\in\Upsilon_4^*$ and $\widehat y_2\in\Upsilon_2^*$. The decomposition described above takes the form $v=\shift^3(\iota(u))$, where $u=\widehat y_1\oplus\widehat y_2=3214\oplus 21=321465$. 
\end{example}

The decomposition described above is unique. On the other hand, if we are given positive integers $\kappa_1,\ldots,\kappa_r$ summing to $n$, an integer $\beta$ satisfying $0\leq\beta\leq\kappa_1-1$, and permutations $\widehat y_1,\ldots,\widehat y_r$ with $\widehat y_i\in\Upsilon_{\kappa_i}$ for all $i$, then $\shift^\beta(\iota(\widehat y_1\oplus\cdots\oplus\widehat y_r))$ is a uniquely sorted $231$-avoiding affine permutation in $\widetilde S_n$. Indeed, each $\widehat y_i$ is obtained by deleting the largest entry from a permutation $\widehat x_i\in\Upsilon_{\kappa_i+1}$. The unique affine valid hook configuration of $\shift^\beta(\iota(\widehat y_1\oplus\cdots\oplus\widehat y_r))$ is obtained by gluing together the valid hook configurations of $x_1,\ldots,x_r$, which are essentially the same as the valid hook configurations of $\widehat x_1,\ldots,\widehat x_r$, and repeating modulo $r$. It then follows from the Affine Fertility Formula (Theorem~\ref{AFF}) that the resulting affine permutation $\shift^\beta(\iota(\widehat y_1\oplus\cdots\oplus\widehat y_r))$ is indeed uniquely sorted. Thus, we have the following lemma. 

\begin{lemma}\label{LemCoxx2}
Every $231$-avoiding uniquely sorted affine permutation in $\widetilde S_n$ can be written uniquely as $\shift^\beta(\widehat y_1\oplus\cdots\oplus\widehat y_r)$, where $\widehat y_i\in\Upsilon_{\kappa_i}^*$ for all $i$, $\kappa_1+\cdots+\kappa_r=n$, and $0\leq \beta\leq \kappa_1-1$. 
\end{lemma} 

An argument similar to the one used to prove Lemma~\ref{LemCoxx2} yields the next lemma. 

\begin{lemma}\label{LemCoxx3}
Every $231$-avoiding uniquely sorted permutation in $S_n$ can be written uniquely as $\widehat y_1\oplus\cdots\oplus\widehat y_r\oplus 1$, where $\widehat y_i\in\Upsilon_{\kappa_i}^*$ for all $i$ and $\kappa_1+\cdots+\kappa_r+1=n$. 
\end{lemma}

\begin{figure}[ht]
\begin{center}
\includegraphics[height=3.23cm]{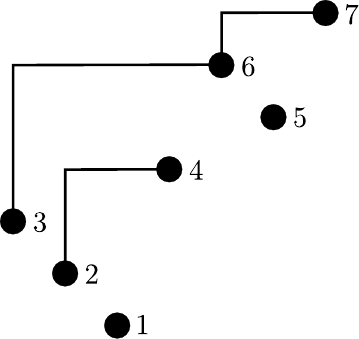}
\caption{The unique valid hook configuration of a $231$-avoiding uniquely sorted permutation in $S_7$.} 
\label{FigCoxx8}
\end{center}  
\end{figure} 

\begin{example}
The $231$-avoiding uniquely sorted permutation $3214657$, whose unique valid hook configuration is presented in Figure~\ref{FigCoxx8}, can be written as $3214\oplus21\oplus 1$. We have $3214\in\Upsilon_4^*$ and $21\in\Upsilon_2^*$. 
\end{example}

\begin{theorem}\label{ThmCoxx5}
For each $k\geq 1$, the number of uniquely sorted affine sylvester classes of $\widetilde S_{2k}$, which is also the number of $231$-avoiding uniquely sorted affine permutations in $\widetilde S_{2k}$, is \[3\binom{4k}{k}-2\sum_{i=0}^k\binom{4k}{i}.\]
\end{theorem}

\begin{proof}
Let $\mathcal U_n(231)$ denote the set of $231$-avoiding uniquely sorted permutations in $S_n$. Let $\widetilde{\mathcal U}_n(231)$ denote the set of $231$-avoiding uniquely sorted affine permutations in $\widetilde S_n$. Each uniquely sorted affine sylvester class of $\widetilde S_{2k}$ contains a unique element of $\widetilde{\mathcal U}_{2k}(231)$, so we just need to count the elements of $\widetilde{\mathcal U}_{2k}(231)$. Let \[F(q)=\sum_{\kappa\geq 1}\lvert \Upsilon_{\kappa}^*\rvert q^\kappa,\quad J(q)=\sum_{\kappa\geq 1}\kappa\lvert \Upsilon_\kappa^*\rvert q^\kappa,\quad G(q)=\sum_{n\geq 1}\lvert \mathcal U_n(231)\rvert q^n,\quad \widetilde G(q)=\sum_{n\geq 1}\lvert \widetilde{\mathcal U}_n(231)\rvert q^n.\] Observe that $J(q)=qF'(q)$. 
It follows immediately from Lemma~\ref{LemCoxx3} that 
$\displaystyle G(q)=\frac{q}{1-F(q)}$. This means that $F(q)=1-\dfrac{q}{G(q)}$. Similarly, it follows from Lemma~\ref{LemCoxx2} that $\displaystyle \widetilde G(q)=\frac{J(q)}{1-F(q)}$. Combining these facts, we find that 
\begin{equation}\label{EqCoxx6}
\widetilde G(q)=\frac{qF'(q)}{1-F(q)}=G(q)F'(q)=G(q)\frac{d}{dq}\left(1-\frac{q}{G(q)}\right)=G(q)\frac{qG'(q)-G(q)}{G(q)^2}=q\frac{G'(q)}{G(q)}-1.
\end{equation} We know by Proposition~\ref{PropUniquely} that $\lvert \mathcal U_n(231)\rvert =0$ when $n$ is even. On the other hand, \cite[Corollary~6.1]{DefantCatalan} states that \[\lvert \mathcal U_{2k+1}(231)\rvert =\frac{2}{(3k+1)(3k+2)}\binom{4k+1}{k+1}.\] Let $E(q)=\sum_{k\geq 0}\frac{2}{(3k+1)(3k+2)}\binom{4k+1}{k+1}q^k=G(\sqrt q)/{\sqrt q}$. The coefficients of $E(q)$ appear as OEIS sequence A000260 \cite{OEIS}; it is stated there that $E(q)=(2-g(q))g(q)^2$, where $g(q)$ satisfies the equation $g(q)=1+qg(q)^4$. Thus, $G(q)=q(2-g(q^2))g(q^2)^2$. Using Mathematica, one can solve this equation explicitly for $G(q)$ and then use that explicit expression to verify that \[\left(2q\frac{G'(q)}{G(q)}\right)^4q^2-\left(q\frac{G'(q)}{G(q)}+1\right)^3\left(q\frac{G'(q)}{G(q)}-1\right)=0.\] Therefore, we can use \eqref{EqCoxx6} to see that \[(2(\widetilde G(q)+1))^4q^2-(\widetilde G(q)+2)^3\widetilde G(q)=0.\] Now let \[Q(q)=\sum_{k\geq 1}\left(3\binom{4k}{k}-2\sum_{i=0}^k\binom{4k}{i}\right)q^{2k}.\] According to OEIS entry A107026 \cite{OEIS}, the generating function $Q(q)$ satisfies the equation \[(2(Q(q)+1))^4q^2-(Q(q)+2)^3Q(q)=0.\] In fact, $Q(q)$ is uniquely determined by the fact that it satisfies this equation and satisfies $Q(q)=2q^2+O(q^3)$. Since $\widetilde G(q)$ is another series satisfying this equation and satisfying $\widetilde G(q)=2q^2+O(q^3)$, we must have $Q(q)=\widetilde G(q)$. This completes the proof. 
\end{proof}

\begin{remark}
Although the numbers in Theorem~\ref{ThmCoxx5} already appeared as sequence A107026 in \cite{OEIS}, it seems that there was no known combinatorial interpretation for them until now. 
\end{remark}

As mentioned in Section~\ref{SecStackBack}, much of the early work on West's stack-sorting map developed around the notion of a $2$-stack-sortable permutation and arose from Zeilberger's \cite{Zeilberger} theorem that the number of $2$-stack-sortable permutations in $S_n$ is $\frac{2}{(n+1)(2n+1)}\binom{3n}{n}$. Our final theorem provides an affine analogue of Zeilberger's result. 

In the next lemma, recall that a \emph{left-to-right maximum} of a permutation $z$ is an entry in $z$ that is larger than all entries appearing to its left in $z$.

\begin{lemma}\label{LemCoxx4}
Let $mz$ be a permutation of a set $\{a,a+1,\ldots,b\}$ of consecutive integers obtained by concatenating a number $m$ with a $231$-avoiding permutation $z$. Then $mz$ is $231$-avoiding if and only if $m=a$ or $m-1$ is a left-to-right maximum of $z$.   
\end{lemma}

\begin{proof}
If $m=a$, then certainly $mz$ is $231$-avoiding. Suppose $m-1$ is a left-to-right maximum of $z$. If there were entries $c$ and $d$ such that $m,d,c$ formed a $231$-pattern in $mz$, then $m-1,d,c$ would form a $231$-pattern in $z$. This cannot happen, so $mz$ is $231$-avoiding. 

Conversely, suppose $m\neq a$ and $m-1$ is not a left-to-right maximum of $z$. There is an entry $c>m-1$ that appears to the left of $m-1$ in $z$, so $m,c,m-1$ form a $231$-pattern in $mz$. 
\end{proof}

Consider $w\in\widetilde S_n$, and let $v=\widetilde\stack(w)$. Suppose $w$ is $2$-stack-sortable. Equivalently, $v$ is $1$-stack-sortable. Equivalently (by Proposition~\ref{PropCoxx3}), $v$ is $231$-avoiding. Let $\mathscr M=\{\cdots<m_0<m_1<m_2<\cdots\}$ be the left-to-right maxima of $w$. Then $\mathscr M\in\Sky(v)$. Let us assume the indices are chosen so that $w^{-1}(m_1)\leq 1<w^{-1}(m_2)$. Let $\mathcal I^{-1}(w)=\mathcal T=(T,\sigma)$. Then $v=\mathcal P(\mathcal T)$ by \eqref{EqAffinePostorder}. Let $\mathcal T_R(m_{i+1})$ be the (possibly empty) right subtree of $\sigma^{-1}(m_{i+1})$ in $\mathcal T$, and let $z_i=\mathcal P(\mathcal T_R(m_{i+1}))$. The one-line notation of $v$ is $\cdots m_0z_0m_1z_1m_2\cdots$. Let $u_i=\mathcal I(\mathcal T_R(m_{i+1}))$. Then each permutation $z_i=\stack(u_i)$ is $231$-avoiding, so each permutation $u_i$ is $2$-stack-sortable (recall that Knuth showed that $1$-stack-sortable permutations are the same as $231$-avoiding permutations). All of the entries in $m_iz_i$ are smaller than $m_{i+1}$ because they are labels of vertices lying below $\sigma^{-1}(m_{i+1})$ in $\mathcal T$. Furthermore, every entry $a$ in $m_iz_i$ is smaller than every entry $b$ in $z_{i+1}$ since, otherwise, $a,m_{i+1},b$ would form a $231$-pattern in $v$. It follows from these observations that for every integer $i$, the entries in $m_iz_i$ form a set of consecutive integers. By Lemma~\ref{LemCoxx4}, the entry $m_i$ is either the smallest entry in $m_iz_i$ or is $1$ more than a left-to-right maximum of $z_i$. Let $r$ be the unique integer such that $m_{j+r}=m_j+n$ for all $j\in\mathbb Z$, and let $\kappa_i$ be the size of $m_iz_i$. Finally, let $\beta=-w^{-1}(m_1)+1$. Since $\kappa_1=w^{-1}(m_2)-w^{-1}(m_1)$, the assumption that $w^{-1}(m_1)\leq 1<w^{-1}(m_2)$ guarantees that $0\leq\beta\leq\kappa_1-1$. See Example~\ref{ExamCoxx3}.   

The decomposition of a $2$-stack-sortable affine permutation described in the previous paragraph is reversible. That is to say, $w$ is uniquely determined once we choose the integers $\kappa_1,\ldots,\kappa_r$ that sum to $n$, the standardizations of the permutations $m_iu_i$ for $1\leq i\leq r$, and the integer $\beta$ satisfying $0\leq \beta\leq \kappa_1-1$. To describe the restrictions on $m_iu_i$, let $\widehat m_i\widehat u_i$ be its standardization. Let $\widehat z_i=\stack(\widehat u_i)$. Then $\widehat u_i$ must be $2$-stack-sortable, and either $\widehat m_i=1$ or $\widehat m_i-1$ is a left-to-right maximum of $\widehat z_i$. If we write $\tau_i$ for the standardization of $u_i$, then $\tau_i$ is a $2$-stack-sortable permutation in $S_{\kappa_i-1}$. 

\begin{figure}[ht]
  \begin{center}{\includegraphics[height=9.203cm]{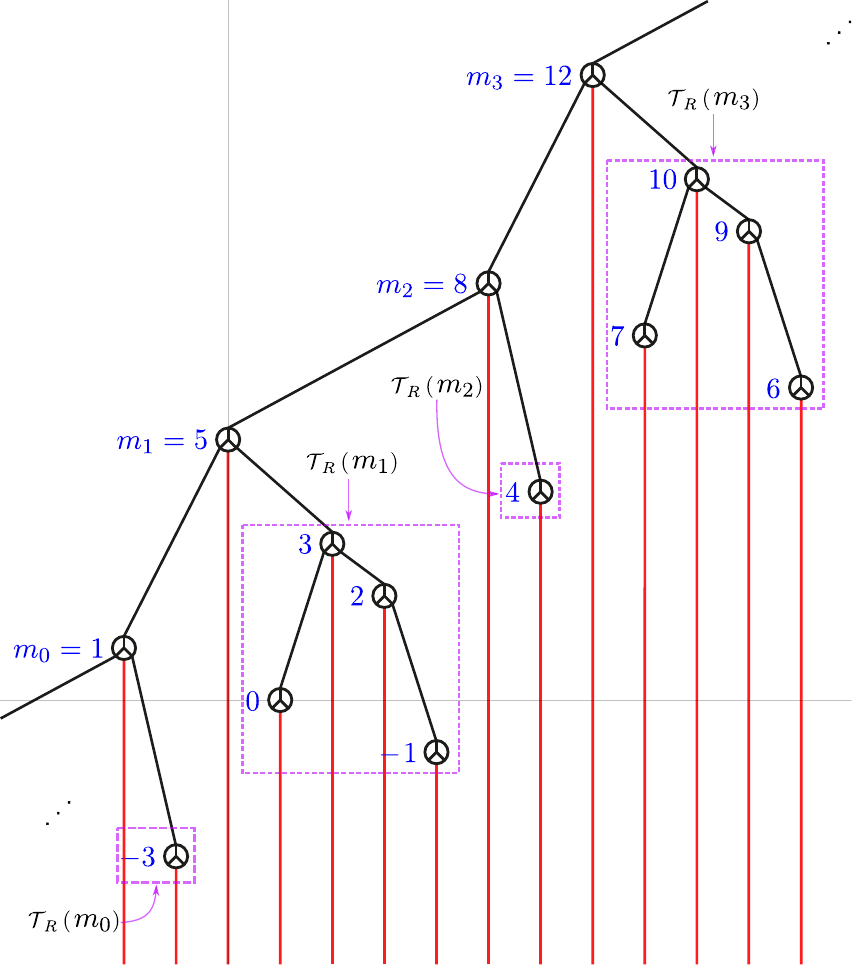}}
  \end{center}
  \caption{The decomposition of $\mathcal I^{-1}(w)$, where $w\in\widetilde S_7$ is the $2$-stack-sortable affine permutation with window notation $[0,3,2,-1,8,4,12]$. }\label{FigCoxx9}
\end{figure}

\begin{example}\label{ExamCoxx3}
Let $n=7$, and let $w\in\widetilde S_7$ be the $2$-stack-sortable affine permutation with window notation $[0,3,2,-1,8,4,12]$. The decreasing affine binary plane tree $\mathcal T=\mathcal I^{-1}(w)$ is depicted in Figure~\ref{FigCoxx9}. In the above notation, we have $m_1=5$, $m_2=8$, and $r=2$. The subtrees $\mathcal T_R(m_{i+1})$ are marked in the figure. We have $u_1=4$ and $u_2=7\,10\,9\,6$. Also, $z_1=\stack(u_1)=\mathcal P(\mathcal T_R(m_2))=4$ and $z_2=\stack(u_2)=\mathcal P(\mathcal T_R(m_3))=7\,6\,9\,10$. Notice that for each $i$, the entries in $m_iz_i$ form a set of consecutive integers. We have $m_1z_1=54$ and $m_2z_2=8\,7\,6\,9\,10$. Also, $\widehat m_1\widehat u_1=21$, $\widehat m_2\widehat u_2=32541$, $\widehat z_1=\stack(\widehat u_1)=\stack(1)=1$ and $\widehat z_2=\stack(\widehat u_2)=\stack(2541)=2145$. Observe that $\widehat m_1-1=2-1=1$ is a left-to-right maximum of $\widehat z_1$, while $\widehat m_2-1=3-1=2$ is a left-to-right maximum of $\widehat z_2$. The permutations $\widehat u_1=1$ and $\widehat u_2=2541$ are $2$-stack-sortable, so their standardizations $\tau_1=1$ and $\tau_2=2431$ are also $2$-stack-sortable. Finally, $\beta=-w^{-1}(5)+1=1$. 

Suppose we were given that $n=7$ and $r=2$ and that we were handed the standardizations $\widehat m_1\widehat u_1=21$ and $\widehat m_2\widehat u_2=32541$ along with the integer $\beta=1$. The standardizations tell us the relative order of $m_iu_i$ for all $i\in\mathbb Z$, and we can begin to reconstruct the set of entries of $m_iu_i$ by using the fact that each such set of entries is a set of consecutive integers (and $\cdots<m_0<m_1<m_2<\cdots$). In fact, since $w=\cdots m_1u_0m_2u_1m_3u_2\cdots$ and $w^{-1}(m_1)=1-\beta=0$, this determines the entire one-line notation of $w$ up to a global shift of the entries. However, there is a unique choice of this global shift that makes the identity $\sum_{i=1}^nw(i)=\binom{n+1}{2}$ hold. Hence, we can reconstruct $w$.   
\end{example}

In what follows, we let ${\mathcal W}_2(n)$ and $\widetilde{\mathcal W}_2(n)$ denote, respectively, the set of $2$-stack-sortable permutations in $S_n$ and the set of $2$-stack-sortable affine permutations in $\widetilde S_n$. For each permutation $u$, let $\slmax(u)$ denote the number of left-to-right maxima of $\stack(u)$. In the above notation, $\slmax(\tau_i)+1=\slmax(\widehat u_i)+1$ is the number of ways to choose the value of $\widehat m_i$ relative to the values of the entries of $\widehat u_i$ so that either $\widehat m_i=1$ or $\widehat m_i-1$ is a left-to-right maximum of $\widehat z_i$. The preceding discussion yields the following lemma. 

\begin{lemma}\label{LemCoxx5}
For $n\geq 1$, we have \[\lvert \widetilde{\mathcal W}_2(n)\rvert =\sum_{\substack{\kappa_1,\ldots,\kappa_r\geq 1 \\ \kappa_1+\cdots+\kappa_r=n}}\kappa_1\prod_{i=1}^r\left(\sum_{\tau_i\in\mathcal W_2(\kappa_i-1)}(\slmax(\tau_i)+1)\right).\]
\end{lemma}

We can now enumerate $2$-stack-sortable affine permutations. 

\begin{theorem}
Let \[I(q)=\sum_{n\geq 0}\frac{2}{(n+1)(2n+1)}\binom{3n}{n}q^n\quad\text{and}\quad\widetilde I(q)=\sum_{n\geq 1}\lvert \widetilde {\mathcal W}_2(n)\rvert q^n.\] We have \[\widetilde I(q)=\frac{qI'(q)}{I(q)(I(q)-1)}-1.\]  
\end{theorem}

\begin{proof}
It follows from Lemma~\ref{LemCoxx5} that \[\widetilde I(q)=\left(\sum_{\kappa_1\geq 1}\kappa_1\sum_{\tau_1\in\mathcal W_2(\kappa_1-1)}(\slmax(\tau_1)+1)q^{\kappa_1}\right)\sum_{j\geq 0}\left(\sum_{\kappa\geq 1}\sum_{\tau\in\mathcal W_2(\kappa-1)}(\slmax(\tau)+1)q^\kappa\right)^j\] 
\begin{equation}\label{EqCoxx9}
=\frac{qF'(q)}{1-F(q)},
\end{equation} where $F(q)=\sum_{\kappa\geq 1}\sum_{\tau\in\mathcal W_2(\kappa-1)}(\slmax(\tau)+1)q^\kappa$. If we let $G(q,x)=\sum_{\kappa\geq 1}\sum_{\tau\in\mathcal W_2(\kappa)}x^{\slmax(\tau)}q^\kappa$, then 
\begin{equation}\label{EqCoxx8}
F(q)=q+q\left[\frac{\partial}{\partial x}\left(xG(q,x)\right)\right]_{x=1}.
\end{equation} As we mentioned before, $2$-stack-sortable permutations have been studied extensively, and luckily for us, Fang \cite{Fang} has already investigated the generating function $G(q,x)$. In fact, \cite[Equation~(5)]{Fang} provides an equation for a refinement\footnote{In Fang's notation, $G(q,x)$ is $T(q,x,1,1;1,1,1)$.} of $G(q,x)$ that keeps track of several additional statistics. By setting most of the variables in that equation equal to $1$, we arrive at the equation 
\begin{equation}\label{EqCoxx7}
(x-1)G(q,x)=qx(x-1)(1+G(q,x))^2+qxG(q,x)G(q,x)-qxG(q,x)G(q,1). 
\end{equation}
Let us write $G_x(q,x)$ for $\frac{\partial}{\partial x}G(q,x)$. Differentiating each side of \eqref{EqCoxx7} with respect to $x$ yields \[(x-1)G_x(q,x)+G(q,x)=2qx(x-1)(1+G(q,x))G_x(q,x)+q(2x-1)(1+G(q,x))^2\] \[+2qxG(q,x)G_x(q,x)+qG(q,x)^2-qG(q,1)(xG_x(q,x)+G(q,x)).\] Now put $Q(q)=\left[G_x(q,x)\right]_{x=1}$, and set $x=1$ in the previous equation to obtain \[G(q,1)=q(1+G(q,1))^2+2qG(q,1)Q(q)+qG(q,1)^2-qG(q,1)(Q(q)+G(q,1))\] \[=q(1+G(q,1))^2+qG(q,1)Q(q).\] Hence, \[Q(q)=\frac{G(q,1)-q(1+G(q,1))^2}{qG(q,1)}.\] By \eqref{EqCoxx8}, we have \[F(q)=q(1+G(q,1)+Q(q))=q\left(1+G(q,1)+\frac{G(q,1)-q(1+G(q,1))^2}{qG(q,1)}\right)\] \[=1+q\left(1+G(q,1)-\frac{(1+G(q,1))^2}{G(q,1)}\right).\] Zeilberger's enumeration of $2$-stack-sortable permutations \cite{Zeilberger} tells us that \[G(q,1)=\sum_{n\geq 1}\lvert \mathcal W_2(n)\rvert q^n=\sum_{n\geq 1}\frac{2}{(n+1)(2n+1)}\binom{3n}{n}q^n=I(q)-1,\] so \[F(q)=1+q\left(I(q)-\frac{I(q)^2}{I(q)-1}\right)=1+q\frac{I(q)}{1-I(q)}.\] Finally, by \eqref{EqCoxx9}, we have \[\widetilde I(q)=\frac{qF'(q)}{1-F(q)}=\frac{q\frac{qI'(q)+I(q)(1-I(q))}{(1-I(q))^2}}{-q\frac{I(q)}{1-I(q)}}=\frac{qI'(q)}{I(q)(I(q)-1)}-1. \qedhere\]
\end{proof}

\section{Future Directions}\label{SecConclusion}

We have initiated the combinatorial and dynamical analysis of Coxeter stack-sorting operators, but it seems we have only grazed the surface. Here, we collect suggestions for future work. 

\subsection{Other Special Lattice Congruences}
There are many special lattice congruences on symmetric groups whose associated Coxeter stack-sorting operators have not been explored. For example, it could be fruitful to investigate more thoroughly the Coxeter stack-sorting operators on $S_n$ arising from Reading's Cambrian congruences \cite{ReadingCambrian}. Recall that these are the same as the $\delta$-permutree congruences $\equiv_\delta$ for $\delta\in\{\upCirc,\downCirc\}^n$. One particularly interesting Cambrian congruence from \cite{ReadingCambrian} is the \dfn{bipartite Cambrian congruence}, which is obtained by setting $\delta=\delta_1\cdots\delta_n$, where $\delta_i=\upCirc$ when $i$ is odd and $\delta_i=\downCirc$ when $i$ is even.

Some other lattice congruences on $S_n$ whose Coxeter stack-sorting operators could be worth investigating are the $k$-twist congruences from \cite{Pilaud3} and the Baxter congruence from \cite{Law, Giraudo}. 

\subsection{Upward Projection Maps}
Suppose $\equiv$ is a lattice congruence on the left weak order of a finite Coxeter group $W$. Every congruence class of $\equiv$ has a unique maximal element, and we denote by $\pi_\equiv^\uparrow$ the upward projection map that sends each $w\in W$ to the maximal element of the congruence class containing $w$. In analogy with Definition~\ref{DefCoxStackOp}, we consider the map ${\bf R}_\equiv:W\to W$ defined by ${\bf R}_\equiv(w)=\pi_\equiv^\uparrow(w)w^{-1}$. The long-term dynamical behavior of ${\bf R}_\equiv$ can be more complicated than that of ${\bf S}_\equiv$ because there can be multiple periodic points. It would be interesting to gain an understanding of these new operators, even for specific lattice congruences on symmetric groups. 

As an example, consider the sylvester congruence $\equiv_{\syl}$ on $S_n$. For $w\in S_n$, let $\rev(w)=ww_0$ be the reverse of $w$. Dukes \cite{Dukes} introduced the map $\revstack=\stack\circ\rev$. One can show that ${\bf R}_{\equiv_{\syl}}(w)=\revstack(w)^{-1}$ for all $w\in S_n$. In addition, one can show that the identity permutation $e$ is the only periodic point of ${\bf R}_{\equiv_{\syl}}$ in $S_n$. Experimental data (checked for $n\leq 9$) suggests that the maximum size of a forward orbit of ${\bf R}_{\equiv_{\syl}}:S_n\to S_n$ could be $\left\lceil\frac{n}{2}\right\rceil+2$ for all $n\geq 4$. By contrast, the maximum size of the forward orbit of an element of $S_n$ under $\revstack$ is $n$ (see \cite{Dukes}). 

Similarly, one can show that ${\bf R}_{\equiv_{\des}}(w)={\bf S}_{\equiv_{\des}}(\rev(w))^{-1}$ for all $w\in S_n$ (recall that ${\bf S}_{\equiv_{\des}}$ is the pop-stack-sorting map). The dynamics of ${\bf R}_{\equiv_{\des}}$ seem interesting. For instance, when $n=4$, ${\bf R}_{\equiv_{\des}}$ has one fixed point, one periodic orbit of period $4$, one periodic orbit of period $6$, and $13$ non-periodic points. 

Perhaps one could prove general results, along the same lines as in \cite{DefantCoxeterPop}, for the map ${\bf R}_{\equiv_{\des}}:W\to W$ when $W$ is an arbitrary finite irreducible Coxeter group. 

\subsection{Descents After Coxeter Stack-Sorting}
In Section~\ref{SecDescents}, we proved Theorems~\ref{ThmCoxx6} and~\ref{ThmCoxx1}, which tell us about the maximum number of right descents a permutation in the image of a Coxeter stack-sorting operator on $S_n$ can have. It would be interesting to have analogues of these theorems for Coxeter groups of other types.

\subsection{Stack-Sorting in Type $B$}

We believe there is more hidden structure in the type-$B$ stack-sorting map $\stack_B$. For example, we have the following conjecture. 

\begin{conjecture}\label{ConjCoxx1}
If $n\geq 1$ is odd and $v\in B_n$, then $\lvert \stack_B^{-1}(v)\rvert $ is even. 
\end{conjecture}

In particular, Conjecture~\ref{ConjCoxx1} implies the following weaker conjecture. 

\begin{conjecture}\label{ConjCoxx2}
If $n\geq 1$ is odd and $v\in B_n$, then $\lvert \stack_B^{-1}(v)\rvert \neq 1$. 
\end{conjecture}

Recall from Proposition~\ref{PropUniquely} that a permutation in $S_n$ is uniquely sorted if and only if it is in the image of $\stack$ and has exactly $\frac{n-1}{2}$ right descents. In light of Theorem~\ref{ThmCoxx2}, one might ask if the elements of $B_n$ with exactly $1$ preimage under $\stack_B$ are the elements in the image of $\stack_B$ with exactly $\frac{n}{2}$ right descents. This turns out to be false. For example, $25136847\in B_4$ has exactly $1$ preimage under $\stack_B$, but it has only $1$ right descent. 

\begin{question}
Let $n\geq 2$ be even. Suppose $v\in B_n$ is in the image of $\stack_B$ and has exactly $\frac{n}{2}$ right descents. Is it necessarily true that $\lvert \stack_B^{-1}(v)\rvert =1$? 
\end{question}

\begin{question}
Let $n\geq 2$ be even. How many elements of $B_n$ in the image of $\stack_B$ have exactly $\frac{n}{2}$ right descents?  
\end{question}

Theorem~\ref{ThmCoxx7} tells us that $\stack_B^n(w)=e$ for all $w\in B_n$ and that there exists $v\in B_n$ such that $\stack_B^{n-1}(v)\neq e$.

\begin{question}
How many elements $v$ of $B_n$ are such that $\stack_B^{n-1}(v)\neq e$? 
\end{question}

The sequence of numbers requested in the previous question appears to be new; it begins with the numbers $1,2,6,32,200,1566$.

In \cite{DefantMonotonicity}, the author gave asymptotic lower and upper bounds for $\mathcal D_n$, the average number of iterations of $\stack$ needed to send a permutation in $S_n$ to the identity permutation $e$. Let us similarly define $\mathcal D_n^B$ to be the average number of iterations of $\stack_B$ needed to send an element of $B_n$ to the identity element $e$. We believe the major factors controlling the number of iterations of $\stack$ needed to send permutations to the identity should also be present in type $B$, so we have the following conjecture. 

\begin{conjecture}
We have \[\lim_{n\to\infty}\left(\frac{\mathcal D_n}{n}-\frac{\mathcal D_n^B}{n}\right)=0.\]
\end{conjecture}

\subsection{Affine Stack-Sorting}

We saw in Section~\ref{SecAffine} that the affine stack-sorting map $\widetilde\stack$ shares several of the stack-sorting map's nice properties. In \cite{DefantMonotonicity}, the author proved a theorem about \emph{fertility monotonicity}; it states that $\lvert \stack^{-1}(w)\rvert \leq \lvert \stack^{-1}(\stack(w))\rvert $ for all $w\in S_n$. 

\begin{question}
Is it true that $\lvert \widetilde\stack^{-1}(w)\rvert \leq \lvert \widetilde\stack^{-1}(\widetilde\stack(w))\rvert $ for all $w\in \widetilde S_n$? 
\end{question}

In Section~\ref{SecAffine}, we proved affine analogues of the Decomposition Lemma and the Fertility Formula. These are special cases of the results used in \cite{DefantTroupes} to develop a theory of \emph{troupes}, which are families of binary plane trees that are closely related to the link between the stack-sorting map and free probability theory. It would be very interesting to have a parallel theory of \emph{affine troupes}, especially if such a theory could connect $\widetilde\stack$ with free probability theory (or some variant thereof).

\section{Acknowledgments}
The author thanks Nathan Reading for a helpful conversation. He also thanks the anonymous referees for providing very useful feedback. The author was supported by a Fannie and John Hertz Foundation Fellowship and an NSF Graduate Research Fellowship.

\end{document}